\documentclass[11pt]{article}
\usepackage{smile}

\usepackage{fullpage}
\usepackage{lscape}
\usepackage{bigints}
\usepackage{framed}
\usepackage{mdframed}
\usepackage{enumerate}
\usepackage[inline]{enumitem}
\usepackage[T1]{fontenc}
\usepackage{moresize}
\usepackage{bm}
\usepackage{bbm}
\usepackage{dsfont}
\usepackage{amsmath}
\usepackage{amssymb}
\usepackage{amsthm}
\usepackage{amsfonts}
\usepackage{stmaryrd}
\usepackage{array}
\usepackage{mathrsfs}
\usepackage{mathtools} 
\usepackage{extarrows}
\usepackage{stackrel}
\usepackage{relsize,exscale}
\usepackage{scalerel}
\usepackage[nodisplayskipstretch]{setspace}
\usepackage{color}
\usepackage[usenames,dvipsnames]{xcolor}
\usepackage{cancel}
\usepackage{soul}
\usepackage{undertilde}
\usepackage{xfrac}
\usepackage{siunitx}
\usepackage{graphicx}
\usepackage{float}
\usepackage{rotating}
\usepackage{subcaption}
\usepackage{overpic}
\usepackage[all]{xy}
\usepackage{tikz}
\usetikzlibrary{arrows,matrix,positioning,calc,automata,patterns}
\usepackage{booktabs}
\usepackage{dcolumn}
\usepackage{multirow}
\usepackage{diagbox}
\usepackage{tabularx}
\usepackage{verbatim}
\usepackage{listings}
\usepackage[ruled,vlined]{algorithm2e}
\usepackage{fancyvrb}
\usepackage{hyperref}
\usepackage[round]{natbib}
\usepackage{sectsty}

\hypersetup{
    bookmarks=true,         
    unicode=false,          
    pdftoolbar=true,        
    pdfmenubar=true,        
    pdffitwindow=false,     
    pdfstartview={FitH},    
    pdftitle={My title},    
    pdfauthor={Author},     
    pdfsubject={Subject},   
    pdfcreator={Creator},   
    pdfproducer={Producer}, 
    pdfkeywords={key1, key2}, 
    pdfnewwindow=true,      
    colorlinks=true,        
    linkcolor=blue,         
    citecolor=blue,         
    filecolor=blue,         
    urlcolor=cyan           
}

\def\scaleint#1{\vcenter{\hbox{\scaleto[3ex]{\displaystyle\int}{#1}}}}

\usepackage{stackengine}
\stackMath
\newcommand\tenq[2][1]{%
\def\useanchorwidth{T}%
\ifnum#1>1%
\stackunder[0pt]{\tenq[\numexpr#1-1\relax]{#2}}{\!\scriptscriptstyle\thicksim}%
\else%
\stackunder[1pt]{#2}{\!\scriptstyle\thicksim}%
\fi%
}

\makeatletter
\DeclareRobustCommand\widecheck[1]{{\mathpalette\@widecheck{#1}}}
\def\@widecheck#1#2{%
    \setbox\z@\hbox{\m@th$#1#2$}%
    \setbox\tw@\hbox{\m@th$#1%
       \widehat{%
          \vrule\@width\z@\@height\ht\z@
          \vrule\@height\z@\@width\wd\z@}$}%
    \dp\tw@-\ht\z@
    \@tempdima\ht\z@ \advance\@tempdima2\ht\tw@ \divide\@tempdima\thr@@
    \setbox\tw@\hbox{%
       \raise\@tempdima\hbox{\scalebox{1}[-1]{\lower\@tempdima\box
\tw@}}}%
    {\ooalign{\box\tw@ \cr \box\z@}}}
\makeatother

\def\given{\,|\,}
\def\biggiven{\,\big{|}\,}
\def\Biggiven{\,\Big{|}\,}
\def\tr{\mathop{\text{tr}}\kern.2ex}
\def\tZ{{\tilde Z}}
\def\tX{{\tilde X}}
\def\tY{{\tilde Y}}

\def\tW{{\tilde W}}
\def\P{{\mathrm P}}

\def\E{{\mathrm E}}
\def\R{{\mathbbm R}}

\def\d{{\mathrm d}}

\newcommand{\zahl}[1]{\llbracket #1\rrbracket}
\newcommand\yestag{\addtocounter{equation}{1}\tag{\theequation}}
\newcolumntype{L}[1]{>{\raggedright\let\newline\\\arraybackslash\hspace{0pt}}m{#1}}
\newcolumntype{C}[1]{>{  \centering\let\newline\\\arraybackslash\hspace{0pt}}m{#1}}
\newcolumntype{R}[1]{>{ \raggedleft\let\newline\\\arraybackslash\hspace{0pt}}m{#1}}
\newcolumntype{d}[1]{D{.}{.}{#1}}
\newcolumntype{H}{>{\setbox0=\hbox\bgroup}c<{\egroup}@{}}
\newcolumntype{Z}{>{\setbox0=\hbox\bgroup}c<{\egroup}@{\hspace*{-\tabcolsep}}}
\newcolumntype{b}{X}
\newcolumntype{s}{>{\hsize=.5\hsize}X}

\numberwithin{equation}{section}

\newtheorem{theorem}{Theorem}[section]
\newtheorem{lemma}{Lemma}[section]
\newtheorem{proposition}{Proposition}[section]
\newtheorem{assumption}{Assumption}[section]
\newtheorem{corollary}{Corollary}[section]
\newtheorem{innercustomgeneric}{\customgenericname}
\providecommand{\customgenericname}{}
\newcommand{\newcustomtheorem}[2]{%
  \newenvironment{#1}[1]
  {%
   \renewcommand\customgenericname{#2}%
   \renewcommand\theinnercustomgeneric{##1}%
   \innercustomgeneric
  }
  {\endinnercustomgeneric}
}
\newcustomtheorem{customdefinition}{Definition}
\newcustomtheorem{customdefinitions}{Definitions}
\newcustomtheorem{customtheorem}{Theorem}
\newcustomtheorem{customassumption}{Assumption}
\newcustomtheorem{customlemma}{Lemma}
\newcustomtheorem{customexample}{Example}
\theoremstyle{definition}

\newtheorem{remark}{Remark}[section]

\graphicspath{{./fig3/}}

\allowdisplaybreaks

\begin{document}

\setlength{\abovedisplayskip}{5pt}
\setlength{\belowdisplayskip}{5pt}
\setlength{\abovedisplayshortskip}{5pt}
\setlength{\belowdisplayshortskip}{5pt}
\hypersetup{colorlinks,breaklinks,urlcolor=blue,linkcolor=blue}

\title{\LARGE On boosting the power of Chatterjee's rank correlation}

\author{Zhexiao Lin\thanks{Department of Statistics, University of Washington, Seattle, WA 98195, USA; e-mail: {\tt zxlin@uw.edu}}~~~and~
Fang Han\thanks{Department of Statistics, University of Washington, Seattle, WA 98195, USA; e-mail: {\tt fanghan@uw.edu}}
}

\date{}

\maketitle

\vspace{-1em}

\begin{abstract}
\cite{chatterjee2020new}'s ingenious approach to estimating a measure of dependence first proposed by \cite{MR3024030} based on simple rank statistics has quickly caught attention. This measure of dependence has the unusual property of being between 0 and 1, and being 0 or 1 if and only if the corresponding pair of random variables is independent or one is a measurable function of the other almost surely. However, more recent studies \citep{cao2020correlations,shi2020power} showed that independence tests based on Chatterjee's rank correlation are unfortunately rate-inefficient against various local alternatives and they call for variants. We answer this call by proposing revised Chatterjee's rank correlations that still consistently estimate the same dependence measure but provably achieve near-parametric efficiency in testing against Gaussian rotation alternatives. This is possible via incorporating many right nearest neighbors in constructing the correlation coefficients. We thus overcome the ``only one disadvantage'' of Chatterjee's rank correlation \citep[Section 7]{chatterjee2020new}. 
\end{abstract}


\section{Introduction}\label{sec:intro}

Consider $X,Y$ to be two random scalars defined on the same probability space. In various scenario one is interested in quantifying the strength of association between $X$ and $Y$ as well as determining the validity of the following null hypothesis,
\begin{align}\label{eq:H0}
  H_0: X \text{ and }Y\text{ are independent},
\end{align}
both based on a finite sample of size $n$. These two tasks are usually convoluted and together play a pivotal role in many statistics and scientific practices \citep{pearl2009causality,josse2016measuring,MR3889064}.


For handling the above two tasks, this paper is focused on such nonparametric rank correlations that measure associations between $X$ and $Y$ using only the ranks of the data. Rank correlations are particularly appealing for continuous $X,Y$ as then the corresponding tests under $H_0$ are fully distribution-free.
Early such proposals include Spearman's $\rho$ \citep{Spearman1904}, Kendall's $\tau$ \citep{Kendall1938}, Gini's $\gamma$ \citep{gini1914ammontare}, and Blomqvist's $\beta$ \citep{MR39190}, which however cannot arrive at a consistent test of independence. For the sake of testing consistency, \cite{MR0029139}, \cite{MR0125690}, \cite{yanagimoto1970measures}, and \cite{MR3178526} have proposed variants that not only lead to consistent tests of independence but are also shown to be rate-efficient against many local alternatives; cf. \cite{MR3466185}, \cite{shi2020power}, and \cite{shi2020rate}.

In a recent manuscript that received much attention, \cite{chatterjee2020new} introduced a new rank correlation coefficient that estimates a correlation measure first proposed by Dette, Siburg, and Stoimenov \citep{MR3024030}. Compared to the existing ones, this new pair of correlation measure and coefficient appears to have some unusual properties including, in particular, that
\begin{enumerate}[itemsep=-.5ex,label=(\arabic*)]
  \item the measure is between 0 and 1, is 0 if and only if $X$ and $Y$ are independent, and is 1 if and only if $Y$ is a measurable function of $X$ almost surely;
  \item the correlation coefficient has a very simple expression and is an (almost surely) consistent estimator of the measure as long as $Y$ is not almost surely a constant.
\end{enumerate}

Due to the above attractive properties, Chatterjee's rank correlation is an appealing choice for measuring bivariate association strength, especially in detecting perfect functional dependence \citep{cao2020correlations}. On the other hand, for testing $H_0$ in \eqref{eq:H0}, we have known that Chatterjee's proposal will suffer from a lack of power; cf. the claim made in \citet[Section 7]{chatterjee2020new}, the analysis conducted in \cite{cao2020correlations}  and \cite{shi2020power}, as well as the results in a very recent study \citep{shi2021ac}.  As a matter of fact, the critical detection boundary of the test based on Chatterjee's rank correlation was calculated to be at $n^{-1/4}$, which is much slower than the usual parametric $n^{-1/2}$ rate \citep{auddy2021exact}. These analyses thus motivate revising Chatterjee's original proposal to be able to not only detect perfect functional dependence but also attain (near) parametric efficiency in testing independence.


In this paper, we make such a revision by encouraging incorporating a diverging $M=M_n$ many {\it right} nearest neighbors (NNs) into the construction of the correlation coefficient, which we show is still an almost surely consistent estimator of \cite{MR3024030}'s measure of dependence as long as $M/n\to 0$. One could then regard the revised statistic as an extension of Chatterjee's original one from $1$-NN-based to $M$-NN-based. Notably speaking, similar ideas were already pursued in \citet[Equations (3.3) and (8.5)]{deb2020kernel}; see also \cite{MR3992389} for a related proposal that approximates the mutual information \citep{MR3909934} using $M$-NN-based statistics. However, our approach to incorporating more NNs is distinct from theirs (cf. Remark~\ref{remark:deb} in Section~\ref{sec:prelim}). In addition, for guaranteeing a normal limiting null distribution, 
\cite{deb2020kernel} required a very small $M$ of order Poly-$\log n$ (cf. \citet[Theorem 4.1]{deb2020kernel}). In contrast, the most interesting region in our study takes place when $M$ scales at nearly the same order as $n$; see Section~\ref{sec:alter} ahead.

One main ingredient of this paper pertains to a local power analysis of the proposed revised Chatterjee's rank correlation coefficients. For facilitating the presentation, our attention is restricted to the Gaussian rotation model that is benchmark in independence testing and has been widely adopted in literature; cf. \citet[Page 301]{MR79384},  \citet[Section 4]{MR3737306},  \citet[Section 4.2]{MR4185806}, \citet[Section 5]{shi2020rate}, and \citet[Section 3]{shi2020power}. Considering the Pearson correlation between $X$ and $Y$ in a local alternative sequence to be $\rho_n$, we show that the test based on the revised Chatterjee's rank correlation has power tending to one as long as
\[
 \frac{|\rho_n|}{[(n^{1/2}M^{-3/2}) \vee M^{-1/2}] \wedge [(nM)^{-1/4} \vee (n^{-1/2}M^{1/4})]} \to \infty;
\]
here $\vee$ and $\wedge$ represent the maximum and minimum of the two numbers besides it, respectively. In particular, as $M/n$ slowly converges to 0 , the denominator in the above fraction can be arbitrarily close to $n^{-1/2}$, the well known parametric detection boundary \citep{MR2135927}; on the other hand, as $M$ diverges more and more slowly to infinity, the boundary tends to $n^{-1/4}$, the critical detection boundary of Chatterjee's original statistic derived in \cite{auddy2021exact}.

Technically speaking, our analysis hinges on a careful (sharp up to some $\log n$ terms) calculation of the proposed correlation coefficients' means and variances
under both null and local alternatives. A particular focus is on such $M$ that can scale fast with and even at a rate close to $n$. Analogous results in \cite{chatterjee2020new},  \cite{deb2020kernel}, \cite{auddy2021exact}, and \cite{shi2021ac} are not quite helpful in this regime since they are focused on small $M$ that is either fixed or scales to infinity at a sub-polynomial rate. More specifically,
\begin{itemize}
\item[(i)] we obtain explicit forms of the statistics' means and variances under the null (cf. Theorem~\ref{thm:null,mean,var} ahead) via a brute-force combinatorial analysis that is in contrast to existing ones; the latter is only applicable to small $M$;
\item[(ii)] we establish sharp bounds (up to some $\log n$ terms) for the statistics' means and variances under local alternatives (cf. Theorem~\ref{thm:alter,mean,var} ahead). Notably speaking, the interplay between different units in our formulation of the correlation coefficients is remarkably more sophisticated as $M$ is large, when those units with not enough right nearest neighbors need to be handled carefully, while negligible if $M$ is small. 
\end{itemize}
To complete the story, a central limit theorem of the proposed statistics under the null is also established whenever $M$ scales at a slower rate than $n^{1/4}$, which however is too slow to be helpful in making the corresponding test attain near-parametrical efficiency.


\paragraph*{Paper organization.}
The rest of the paper is organized as followed. Section~\ref{sec:prelim} reviews the correlation measure proposed by \cite{MR3024030} and introduces our revised correlation coefficients of \cite{chatterjee2020new}. Section \ref{sec:test} presents the according tests of independence and establishes their size validity and consistency. Section~\ref{sec:alter} presents a local power analysis of the proposed tests with (sufficient) detection boundaries under Gaussian rotation models provided. Section \ref{sec:sim} illustrates the empirical performance of the proposed statistics via  finite-sample studies. Section \ref{sec:main-proof} provides the proof of the main results in this manuscript, with the rest proofs and auxiliary results relegated to an appendix.

\paragraph*{Notation.}
For any integer $n\ge 1$, let $\zahl{n}:= \{1,2,\ldots,n\}$ and $n!$ be the factorial of $n$.
A set consisting of distinct elements $x_1,\dots,x_n$ is written as either $\{x_1,\dots,x_n\}$ or $\{x_i\}_{i=1}^{n}$.
The corresponding sequence is denoted $[x_1,\dots,x_n]$ or $[x_i]_{i=1}^{n}$.
The notation $\ind(\cdot)$ is saved for the indicator function.
For a sequence of real numbers $a_1,\ldots,a_n$, we use $(a_1,a_2,\ldots,a_n)$ as a shorthand of $(a_1,a_2,\ldots,a_n)^\top$.
For any $a,b \in \R$, write $a \vee b = \max\{a,b\}$ and $a \wedge b = \min\{a,b\}$.
For any two real sequences $\{a_n\}$and $\{b_n\}$, write $a_n \lesssim b_n$ (or equivalently, $b_n \gtrsim a_n$) if there exists a universal constant $C>0$ such that $a_n/b_n \le C$ for all sufficiently large $n$, and write $a_n \prec b_n$ (or equivalently, $b_n \succ a_n$) if $a_n/b_n \to 0$ as $n$ goes to infinity. Write $a_n = O(b_n)$ if $\lvert a_n \rvert \lesssim b_n$ and $a_n = o(b_n)$ if $\lvert a_n \rvert \prec b_n$.
For any random variable $Z$, $\P_Z$ represents its law.

\section{Correlation measures and coefficients}\label{sec:prelim}

In the sequel, when considering correlation, we use the term ``correlation measure'' to represent population quantities and ``correlation coefficient'' to represent sample quantities. Denote the joint bivariate distribution function of $(X,Y)$ by $F_{X,Y}$ and the marginal distribution functions of $
X$ and $Y$ by $F_X$ and $F_Y$, respectively. Let $(X_1,Y_1), \ldots, (X_n,Y_n)$ be $n$ independent copies of $(X,Y)$. 
Throughout the rest of this manuscript, we assume $(X,Y)$ to be continuous, i.e., $F_{X,Y}$ is a continuous function. This requirement ensures that with probability one there is no tie in the observed data. Denote $\mathcal{P}_c$ to be the family of bivariate probability measures of $(X,Y)$ such that it is continuous. 

\subsection{Chatterjee's rank correlation}

This section introduces Chatterjee's approximation strategy to the following measure of dependence between $X$ and $Y$, introduced in \cite{MR3024030}: 
\begin{align}\label{eq:xi}
\xi=\xi(X,Y):=\;&  \frac{\scaleint{4.5ex}\,\Var\big\{\E\big[\ind\big(Y\geq y\big) \given X \big] \big\} \d F_{Y}(y)}{\scaleint{4.5ex}\,\Var\big\{\ind\big(Y\geq y\big)\big\}\d F_{Y}(y)}.
\end{align}

Compared to many other popular ones \citep{MR0029139, MR0125690, yanagimoto1970measures, MR3178526, shi2020power}, the correlation measure $\xi$ is to us appealing due to its consistency against dependence \citep{MR3842884} as well as capability of detecting perfect functional dependence \citep{cao2020correlations}, which we summarize below.


\begin{proposition}[Theorem 2 in \cite{MR3024030}, Theorem 1.1 in \cite{chatterjee2020new}]\label{prop:corr}
For arbitrary $\P_{(X,Y)}$ such that $Y$ is not almost surely a constant, $\xi$ belongs to the interval $[0,1]$ and
\begin{itemize}
\item[(1)] consistency of the measure: $\xi = 0$ if and only if $X$ and $Y$ are independent;
\item[(2)] detectability of perfect functional dependence: $\xi = 1$ if and only if $Y$ is equal to a measurable function of $X$ almost surely.
\end{itemize}
\end{proposition}

To estimate $\xi$, Chatterjee \citep{chatterjee2020new} pioneered an ingenious rank-based approach that, to the authors' knowledge, has not been explored before in literature. To present his idea in a formal way, let's first introduce some necessary notation. Define
\begin{equation}\label{eq:Ri}
R_{i}:=\sum_{j=1}^{n}\ind\Big(Y_{j}\le Y_{i}\Big)
\end{equation}
to be the rank of $Y_i$ among $\{Y_1,\ldots,Y_n\}$. For any $i \in \zahl{n}$ and $m \in \zahl{n}$, define
\[
j_m(i) :=  \begin{cases}
               \text{the index of the $m$-th right NN of $X_i$}, &\text{if the rank of $X_i$ is smaller than $n-m+1$;} \\
                i,  & \text{if not.}
          \end{cases}
\]
In other words, $j_m(i)$ is the index $j \in \zahl{n}$ such that
\[
  \sum_{k=1}^n \ind(X_i < X_k \le X_j) = m
\]
if there exists such a $j \in \zahl{n}$; otherwise, let $j_m(i) = i$.

With these notation introduced, results on Chatterjee's rank correlation coefficient can then be summarized as follows.

\begin{proposition}[Theorems 1.1 and 2.1 in \cite{chatterjee2020new}]\label{prop:chatterjee} Chatterjee's rank correlation coefficient can be formulated as
\begin{align}\label{eq:original}
\xi_n := 1 - \frac{3 \sum_{i=1}^{n} \big\lvert R_{j_1(i)} - R_i \big\rvert }{n^2-1}.
\end{align}
In addition, for any $\P_{(X,Y)}\in\mathcal{P}_c$,
  \begin{enumerate}[itemsep=-.5ex,label=(\roman*)]
    \item \label{thm:chatterjee-1} $\xi_n$ converges almost surely to $\xi$;
\item[(ii)]\label{thm:chatterjee-2} further assuming that $Y$ is independent of $X$, we have $\sqrt{n}\xi_n$ converges in distribution to $N(0,2/5)$, the Gaussian distribution with mean 0 and variance $2/5$.
\end{enumerate}
\end{proposition}


\subsection{The revised Chatterjee's rank correlations}

The formulation of Chatterjee's rank correlation in \eqref{eq:original} suggests it is an 1-NN-based estimator of Dette–Siburg–Stoimenov’s correlation measure $\xi$, and thus intuitively will suffer from similar efficiency loss as other 1-NN-based estimators in various applications \citep{MR2083,MR532236,MR947577,MR1212489,MR1701112,MR3961499}.

Indeed, recent results have exhibited that a test of independence based on Proposition \ref{prop:chatterjee}(ii) is inefficient in common classes of smooth alternatives; cf. \cite{cao2020correlations} and \cite{shi2020power}. In a more recent manuscript, \cite{auddy2021exact} established that the critical detection boundary of $\xi_n$ lies at $n^{-1/4}$, which is substantially slower than the parametric $n^{-1/2}$ one. These results call for variants of $\xi_n$ that are able to boost the power of independence tests
; cf. a clear message delivered in \citet[Remark 4.3]{deb2020kernel}.

In this paper we answer this call by introducing the following {\it revised Chatterjee's rank correlations} that allow one to take each element's $M$ right nearest neighbors into account:
\begin{align}\label{eq:xin}
\xi_{n,M} =\xi_{n,M}\Big(\big[(X_i,Y_i)\big]_{i=1}^{n}\Big):= -2 + \frac{6 \sum_{i=1}^{n} \sum_{m=1}^M \min\big\{R_i, R_{j_m(i)}\big\} }{(n+1)[nM+M(M+1)/4]}.
\end{align}

Several remarks are in order.

\begin{remark}[Formulation of $\xi_{n,M}$] In \eqref{eq:xin}, $M=M_n\in\zahl{n}$ represents the number of right nearest neighbors the proposed correlation coefficient $\xi_{n,M}$ will exploit and is allowed to increase to infinity with $n$. The denominator, $(n+1)[nM+M(M+1)/4]$, is added to ensure $\E(\xi_{n,M})=0$ under $H_0$ (checking Theorem \ref{thm:null,mean,var} ahead). For reasons to be detailed later (cf. Theorem \ref{thm:local power}), we recommend a sufficiently large $M$ for improving testing efficiency against dependence.
\end{remark}

\begin{remark}[Distribution-freeness of $\xi_{n,M}$ under $H_0$] Examining its formulation, it is immediate that for any $M\in\zahl{n}$ the value of $\xi_{n,M}$ only depends on the coordinate-wise ranks of $[(X_i,Y_i)]_{i=1}^n$. The statistic $\xi_{n,M}$ is thus a rank correlation coefficient and accordingly enjoys all the nice properties shared by rank correlations, including in particular the distribution-freeness (i.e., of a distribution that is not dependent on $\P_{(X,Y)}$) under $H_0$ \citep{MR1680991,MR4185806}.
\end{remark}

\begin{remark}[Relation between $\xi_{n,M}$ and $\xi_n$]\label{remark:bias} In contrast to the coefficient $\xi_n$ introduced in \eqref{eq:original}, in constructing $\xi_{n,M}$ we take the minimum instead of absolute difference. This is an idea pursued in \cite{azadkia2019simple} and \cite{deb2020kernel} as well; cf. the construction of the correlation coefficient $T_n$ in \cite{azadkia2019simple} and Section 8.3.1 in \cite{deb2020kernel}. However, it is worth pointing that, as $M=1$,
  \[
    \xi_{n,1} = -2 + \frac{6 \sum_{i=1}^{n} \min\big\{R_i,R_{j_1(i)}\big\} }{(n+1)(n+1/2)}
  \]
reduces to $\xi_n$ with an asymptotically ignorable small order term; note that $|x|+|y|-|x-y|=2\min\{x,y\}$ for any $x,y\geq 0$. More specifically, one has
  \[
    \lvert \xi_{n,1} - \xi_n \rvert = O(1/n), ~~~\text{(with probability 1)}
  \]
so that the difference is of order $n^{-1}$ and thus won't affect the corresponding asymptotic behavior.
\end{remark}

\begin{remark}[Relation to \cite{deb2020kernel}'s proposal]\label{remark:deb} Our idea to scale up $M$ for boosting the power of rank-based tests is of course not new, and is particularly related to an earlier proposal made in \cite{deb2020kernel}; see, e.g.,  Equations (3.3) and (8.5) therein. It is hence helpful to point out our new discoveries. First of all, it was observed that using the {\it right} nearest neighbors, in contrast to using nearest neighbors in both directions, is important. In particular, it plays a central role in our analysis to show that a test based on $\xi_{n,M}$ can reach near-parametric efficiency as $M$ is close to $n$; cf. Theorem \ref{thm:local power} as well as the finite-sample studies in Section \ref{sec:sim}. Secondly, in \eqref{eq:xin} the normalizing constant $(n+1)[nM+M(M+1)/4]$ was carefully chosen so that the expectation of $\xi_{n,M}$ is exactly zero under $H_0$, which holds for arbitrary $M\in\zahl{n}$. This type of normalization is important if $M$ is large since in our formulation there exists an un-ignorable fraction of $i$'s such that $j_m(i)=i$ for some $m\in\zahl{M}$.

\end{remark}

\begin{remark}[Extremal properties of $\xi_{n,M}$] If $Y=f(X)$ for some strictly increasing function $f(\cdot)$, one has $\min\big\{R_i,R_{j_m(i)}\big\} = R_i$ and thus
  \[
    \xi_{n,M} = 1-\frac{3(M+1)/4}{n+(M+1)/4}.
  \]
On the other hand, if $f(\cdot)$ is  a strictly decreasing function, one has $\min\{R_i,R_{j_m(i)}\} = R_{j_m(i)} = R_i-m$ if $R_i \ge m+1$ and $\min\{R_i,R_{j_m(i)}\} = R_i$ if $R_i \le m$. Some simple calculations then yield
  \[
    \xi_{n,M} = 1-\frac{3(M+1)[(5n+1)/4-(2M+1)/3]}{(n+1)[n+(M+1)/4]}.
  \]
In the above two cases, $\xi_{n,M}$ are both equal to 1 up to a bias of order $M/n$. 
\end{remark}

\begin{remark}[Finite-sample range of $\xi_{n,M}$] An equivalent form of \eqref{eq:xin} is
\[
  \xi_{n,M} = -2 + \frac{3 \sum_{i=1}^{n} \sum_{m=1}^M (R_i + R_{j_m(i)} - \lvert R_{j_m(i)} - R_i \rvert) }{(n+1)[nM+M(M+1)/4]}.
\]
The largest possible value of the correlation coefficient is then
\[
  1-\frac{3(M+1)/4}{n+(M+1)/4},
\]
which is attained when $Y=X$ almost surely. On the other hand, since each $R_i$ can appear in the summation of \eqref{eq:xin} for at most $2M$ times, a straightforward lower bound of $\xi_{n,M}$ is
\[
  -\frac{1}{2} + \frac{3[n-(n+1)(M+1)/4]}{2(n+1)[n+(M+1)/4]}.
\]
The finite-sample range of $\xi_{n,M}$ is hence $[-1/2,1]$ up to a bias of order $M/n$, which is analogous to that of Chatterjee's correlation coefficient (cf. Remark 9 under \citet[Theorem 1.1]{chatterjee2020new}).

\end{remark}

\begin{remark}[Computation complexity]\label{remark:time}
To compute \eqref{eq:xin}, one needs to first sort the samples of both $\{X_i\}_{i=1}^n$ and $\{Y_i\}_{i=1}^n$, and then performs the summation over $nM$ terms. The according computation complexity is $O(n\log n + nM)$. If $M$ is $O(\text{Poly-}\log n)$, it is nearly linear. On the other hand, the computation of $\xi_{n,M}$ will tend to be quadratic as $M$ is closer to $n$, a cost seemingly inevitable.
\end{remark}

We close this section by establishing strong consistency for the proposed rank correlation coefficient $\xi_{n,M}$ that is in parallel to Proposition \ref{prop:chatterjee}\ref{thm:chatterjee-1}.

\begin{theorem}[Strong consistency of $\xi_{n,M}$]\label{thm:asconverge}
For any $\P_{(X,Y)}\in \mathcal{P}_c$, $\xi_{n,M}$ converges almost surely to $\xi$ as long as $M/n\to 0$.
\end{theorem}

\section{Tests of independence}\label{sec:test}

\subsection{Elementary properties under the null}

We start with some elementary properties of $\xi_{n,M}$ when $Y$ is independent of $X$. 
To this end, we first establish the corresponding mean and variance of $\xi_{n,M}$.

\begin{theorem}[Mean and variance of $\xi_{n,M}$ under the null]\label{thm:null,mean,var} Assuming $\P_{(X,Y)}\in\mathcal{P}_c$ and $H_0$ holds, then $\E [\xi_{n,M}] = 0$. If further assuming $M \to \infty$ and $M/n\to0$ as $n \to \infty$, we have
  \[
    \Var [\xi_{n,M}] = \Big[ \frac{2}{5}\Big( \frac{1}{nM} \Big) + \frac{8}{15}\Big( \frac{M}{n^2} \Big) \Big] (1+o(1)).
  \]
\end{theorem}

\begin{remark}
Proposition \ref{prop:chatterjee} shows that the asymptotic null variance of $\sqrt{n}\xi_n$ is $2/5$ as $M=1$. In contrast, Theorem \ref{thm:null,mean,var} revealed that $\Var[\xi_{n,M}]$ scales at a rate of order $(nM)^{-1}\vee (M/n^2)$, which is always faster than $n^{-1}$ under the theorem conditions, and two constants $2/5$ and $8/15$ each governs one rate.
Indeed, the variance will achieve its lowest order $n^{-3/2}$ as $M$ is of order $n^{1/2}$, and  is of order close to $n^{-1}$ if $M$ is of order $n^{\gamma}$ with $\gamma$ close to either 0 or 1.
\end{remark}

We then establish a central limit theorem (CLT) for $\xi_{n,M}$ under the null. It is in parallel to Proposition \ref{prop:chatterjee}\ref{thm:chatterjee-2}; notice that, compared to Theorems \ref{thm:asconverge} and \ref{thm:null,mean,var}, a strong scaling requirement, $M\prec n^{1/4}$, is enforced for guaranteeing its validity. 

\begin{theorem}[Central limit theorem for $\xi_{n,M}$ under the null]\label{thm:CLT}
Assume $\P_{(X,Y)}\in\mathcal{P}_c$ and $H_0$ holds. If further assuming $M \to \infty$ and $M \prec n^{1/4}$ as $n \to \infty$, we have
  \[
    \sqrt{nM}\cdot\xi_{n,M} \text{ converges in distribution to } N(0,2/5).
  \]
\end{theorem}

\begin{remark}[Technical ingredients of Theorem \ref{thm:CLT}]
To establish the CLT of $\xi_{n,M}$ under the null, the following H\'ajek representation of $\xi_{n,M}$ is the key:
\[
  \widehat{\xi}_{n,M} = \frac{6n \sum_{i=1}^{n} \sum_{m=1}^M \min\big\{F_Y(Y_i), F_Y(Y_{j_m(i)})\big\}}{(n+1)[nM+M(M+1)/4]}  - \frac{6 \sum_{i \neq j} \min\big\{F_Y(Y_i), F_Y(Y_j) \big\}}{(n-1)(n+1)}.
\]
Noticing that the variance of $\xi_{n,M}$ under $H_0$ is of order $(nM)^{-1}$ as $M \prec n^{1/4}$, we show that $\sqrt{nM} (\xi_{n,M} - \widehat{\xi}_{n,M})$ converges in probability to zero. It then suffices to establish that $\sqrt{nM} \widehat{\xi}_{n,M}$ converges in distribution to $N(0,2/5)$, which is derived via invoking the normal approximation techniques devised in \citet{MR2435859}; assuming $M\prec n^{1/4}$ is crucial here.
\end{remark}

\subsection{Simulation-based tests of independence}\label{sec:sim-based-test}

This section is focused on testing $H_0$ in \eqref{eq:H0} based on the revised Chatterjee's correlation coefficient $\xi_{n,M}$ that was introduced in Section \ref{sec:prelim}.

For a given sample $\big[(X_i,Y_i)\big]_{i=1}^{n}$, let $\xi_{n,M}$ and $\xi_{n,M}^-$ be the correlation coefficients in \eqref{eq:xin} that are calculated based on $\big[(X_i,Y_i)\big]_{i=1}^{n}$ and $\big[(X_i,-Y_i)\big]_{i=1}^{n}$, respectively. We are interested in the following test statistic,
\begin{align}\label{eq:test-statistic}
\xi_{n,M}^{\pm} := \max\big\{\xi_{n,M},\xi_{n,M}^{-}\big\}.
\end{align}

For approximating the above test statistic's distribution under $H_0$, notice that, as $X$ and $Y$ are independent, the joint distribution of $(\xi_{n,M}, \xi_{n,M}^{-})$ is distribution-free and a simulation-based test can then be directly implemented. In detail, choosing the number of simulations to be $B$, in each round $b \in \zahl{B}$, one draws a sample $\big[R_i^{(b)}\big]_{i=1}^n$ from the uniform distribution over all possibly permutations on $\zahl{n}$. We then calculate the value of $\xi_{n,M}^{(b)}$ from \eqref{eq:xin} as follows:
\begin{align*}
&\xi_{n,M}^{(b)}= -2 + \frac{6 \sum_{i=1}^{n} \sum_{m=1}^M \min\big\{R_i^{(b)},R_{g_m(i)}^{(b)}\big\} }{(n+1)[nM+M(M+1)/4]},\\
\text{with }~~ &g_m(i):=\begin{cases}
              i+m, &\text{if } i+m\leq n, \\
                i,  & \text{if } i+m>n.
          \end{cases}
\end{align*}
One could similarly calculate the value of $\xi_{n,M}^{-(b)}$ based on  $\big[R_i^{-(b)}]_{i=1}^n$ with $R_i^{-(b)} := n+1-R_i^{(b)}$ for each $i \in \zahl{n}$.

Notice that under $H_0$, $(\xi_{n,M}^{(b)}, \xi_{n,M}^{-(b)})$ will have the same distribution as $(\xi_{n,M},\xi_{n,M}^-)$ due to independence between $X$ and $Y$ and the distribution-freeness of (relative) ranks; see similar discussions in Section A.2 and Lemma C3 in \cite{MR3737306-supp}. We then consider
 \[
 \xi_{n,M}^{\pm(b)} := \max\big\{\xi_{n,M}^{(b)},\xi_{n,M}^{-(b)}\big\}.
 \]
For a given significance level $\alpha \in (0,1)$, 
the proposed simulation-based test is then
\begin{align}\label{eq:test}
  \sT_{\alpha,B}^{\xi_{n,M}^{\pm}} = \ind\Big[ (1+B)^{-1}\Big\{1+\sum_{b=1}^B\ind\Big(\xi_{n,M}^{\pm(b)} \geq  \xi_{n,M}^{\pm}\Big)\Big\} \leq \alpha\Big],
\end{align}
whose size validity and power consistency are guaranteed by the following theorem.

\begin{theorem}[Size validity and consistency]\label{thm:test}
  \begin{enumerate}[itemsep=-.5ex,label=(\roman*)]
    \item\label{thm:test1} The test $\sT_{\alpha,B}^{\xi_{n,M}^{\pm}}$ is size valid in the sense that for any fixed $\P =\P_{(X,Y)}\in \mathcal{P}_c$ satisfying $H_0$, denoting $\P_{H_0}:=\P^{\otimes n}$ as the corresponding product measure, we have
    \[
       \P_{H_0}\big(\sT_{\alpha,B}^{\xi_{n,M}^{\pm}}=1\big) \leq \alpha
    \]
    holds for any $n\geq 1$, $B\geq 1$, and $M\in\zahl{n}$.

    \item\label{thm:test2} The test $\sT_{\alpha,B}^{\xi_{n,M}^{\pm}}$ is consistent in the sense that for any fixed $\P \in \mathcal{P}_c$ violating $H_0$, denoting $\P_{H_1}$ as the corresponding product measure, we have
    \[
    \lim_{n\to\infty} \P_{H_1}\big(\sT_{\alpha,B}^{\xi_{n,M}^{\pm}}=1\big)=1
    \]
    as long as $B=B_n\to\infty$ and $M/n\to 0$ as $n\to\infty$.
  \end{enumerate}
\end{theorem}


\section{Local power analysis}\label{sec:alter}

This section investigates the local power of the simulation-based test $\sT_{\alpha,B}^{\xi_{n,M}^{\pm}}$ employing the revised Chatterjee's rank correlation $\xi_{n,M}$ introduced in \eqref{eq:xin}. For facilitating presentation, we restrict the attention to the following Gaussian rotation model that is benchmark in independence testing (cf. \cite{MR79384}).



\begin{assumption}\label{asp:local-alter}
 The bivariate random vector $(X,Y)$ belongs to the family of Gaussian distributed ones with mean $\mu$ and covariance matrix $\Sigma$ such that
  \[
    \mu =
    \begin{pmatrix*}
      0\\
      0
    \end{pmatrix*}
    ,~~~~ \Sigma =
    \begin{pmatrix*}
      1 & \rho\\
      \rho & 1
    \end{pmatrix*},~~~{\rm with}~\rho\in (
    -1,1).
  \]
\end{assumption}

For the local power analysis in this specified alternative set, we examine the asymptotic power along a sequence of
alternatives obtained as
\begin{equation}\label{eq:local-alter}
  H_{1,n}:\rho = \rho_n,
\end{equation}
with $\rho_n \to 0$ as $n \to \infty$.


Due to the construction of the test statistic \eqref{eq:test-statistic}, it suffices to consider positive sequences of $\rho_n$'s. We first establish the mean and variance of $\xi_{n,M}$ under \eqref{eq:local-alter}.

\begin{theorem}[Mean and variance of $\xi_{n,M}$ under \eqref{eq:local-alter}]\label{thm:alter,mean,var}
  Suppose that the considered set of local alternatives satisfies Assumption~\ref{asp:local-alter}. Then concerning with any sequence of alternatives given in \eqref{eq:local-alter}, for any positive sequence $\rho_n \to 0$, as long as $M/\log n \to \infty$ and $M (\log n)^{3/2}/n \to 0$, 
  \begin{enumerate}[itemsep=-.5ex,label=(\roman*)]
    \item \label{thm:alter,mean} the mean of $\xi_{n,M}$ under $H_{1,n}$ satisfies
    \[
    \Big\lvert \E_{H_{1,n}} [\xi_{n,M}] \Big\rvert \lesssim \frac{M}{n}\sqrt{\log n}\rho_n + \rho_n^2 + \frac{M}{n^2};
    \]
    in particular, if $\rho_n\succ n^{-1}$, we have
    \[
      \frac{M}{n}\rho_n + \rho_n^2 \lesssim \E_{H_{1,n}} [\xi_{n,M}] \lesssim \frac{M}{n}\sqrt{\log n}\rho_n + \rho_n^2;
    \]
    \item \label{thm:alter,var} the variance of $\xi_{n,M}$ under $H_{1,n}$ satisfies
    \[
      \Var_{H_{1,n}} [\xi_{n,M}] \lesssim \frac{1}{nM} + \frac{M}{n^2} + \frac{M}{n^2}\sqrt{\log n}\rho_n + \frac{1}{n}\rho_n^2.
    \]
  \end{enumerate}
\end{theorem}


\begin{remark}\label{remark:sharpness}
  It is notable that in the above theorem we only establish an upper bound on $ \Var_{H_{1,n}} [\xi_{n,M}]$. However, by calculating the coefficient of each term in the proof of Theorem~\ref{thm:alter,mean,var}\ref{thm:alter,var}, one is also able to show the following lower bound,
  \[
    \Var_{H_{1,n}} [\xi_{n,M}] \gtrsim \frac{1}{nM} + \frac{M}{n^2} + \frac{1}{n}\rho_n^2.
  \]
The upper bound in Theorem~\ref{thm:alter,mean,var}\ref{thm:alter,var} is therefore rate optimal (up to some $\log n$ terms), although the upper bound of $\Var_{H_{1,n}} [\xi_{n,M}]$ in Theorem~\ref{thm:alter,mean,var}\ref{thm:alter,var} has been sufficient for our purpose.
\end{remark}





As a direct consequence of Theorem~\ref{thm:alter,mean,var}, the following corollary establishes a boundary of $\rho_n$ beyond which the mean of $\xi_{n,M}$ under $H_{1,n}$ shall dominate the standard deviation.

\begin{corollary}\label{crl:detection}
  Suppose that the considered set of local alternatives satisfies Assumption~\ref{asp:local-alter}. Then concerning with any sequence of alternatives given in \eqref{eq:local-alter}, for any positive sequence $\rho_n \to 0$, as long as $M/\log n \to \infty$, $M (\log n)^{3/2}/n \to 0$, and
  \[
    \rho_n \succ \zeta_{n,M}:=[(n^{1/2}M^{-3/2}) \vee M^{-1/2}] \wedge [(nM)^{-1/4} \vee (n^{-1/2}M^{1/4})],
  \]
 we have
  \[
    \lim_{n \to \infty} \frac{\E_{H_{1,n}} [\xi_{n,M}]}{\sqrt{\Var_{H_{1,n}} [\xi_{n,M}]}} = +\infty.
  \]
\end{corollary}

\begin{remark}\label{remark:sharpness2}
In view of Theorem \ref{thm:alter,mean,var}\ref{thm:alter,mean} and Remark \ref{remark:sharpness}, as long as $\rho_n \prec \zeta_{n,M}$,  one also has
  \[
    \lim_{n \to \infty} \frac{\E_{H_{1,n}} [\xi_{n,M}]}{\sqrt{\Var_{H_{1,n}} [\xi_{n,M}]}} = 0.
  \]
  Combined with Corollary \ref{crl:detection}, the above equation thus shows that $\zeta_{n,M}$ is the critical boundary determining whether the mean of $\xi_{n,M}$ under $H_{1,n}$ will be dominating or will be dominated by its standard deviation.
\end{remark}

Combining Theorem \ref{thm:alter,mean,var} with Theorem \ref{thm:null,mean,var}, we are now ready to establish a sufficient detection boundary of $\rho_n$ beyond which the proposed test is of a power tending to 1. Invoking Theorem \ref{thm:CLT} further, this detection boundary is sharp as long as $M\prec n^{1/4}$.

\begin{theorem}[Local power analysis]\label{thm:local power}
Suppose that the considered set of local alternatives satisfies Assumption~\ref{asp:local-alter}. Then concerning with any sequence of alternatives given in \eqref{eq:local-alter}, for any sequence $\rho_n \to 0$, as long as $M/\log n \to \infty$, $M (\log n)^{3/2}/n \to 0$, $B\to \infty$,
  \begin{enumerate}[itemsep=-.5ex,label=(\roman*)]
    \item \label{thm:lpa-1} as long as $|\rho_n| \succ \zeta_{n,M}$, we have
  \[
    \lim_{n\to\infty} \P_{H_{1,n}}\Big(\sT_{\alpha,B}^{\xi_{n,M}^{\pm}}=1\Big)=1;
  \]
\item \label{thm:lpa-2} as long as $|\rho_n| \prec \zeta_{n,M}$ and further assuming $M \prec n^{1/4}$, for any sufficiently small $\alpha$, 
it holds that
\[
   \limsup_{n\to\infty} \P_{H_{1,n}}\Big(\sT_{\alpha,B}^{\xi_{n,M}^{\pm}}=1\Big)\leq \beta_{\alpha},
\]
for some $\beta_{\alpha}<1$ that only depends on $\alpha$.
\end{enumerate}
\end{theorem}

\begin{remark}[Relation between $M$ and $\zeta_{n,M}$]\label{remark:db}
Picking $M=n^\gamma$ for some $\gamma \in (0,1)$, the established boundary in Theorem \ref{thm:local power} is $\zeta_{n,M}=n^{-\beta}$ with
  \[
    \beta = \beta(\gamma) := \Big[\Big(\frac{3}{2}\gamma-\frac{1}{2}\Big) \wedge \frac{\gamma}{2}\Big] \vee \Big[\Big(\frac{1}{4}+\frac{\gamma}{4}\Big) \wedge \Big(-\frac{\gamma}{4}+\frac{1}{2}\Big)\Big].
  \]
Easy to check that $\beta$ is piecewise linear with respect to $\gamma$. In particular,
\begin{itemize}
\item[(i)] if $\gamma \to 0$, then $\beta \to 1/4$, corresponding to the detection boundary of $\xi_n$ established in \cite{auddy2021exact};
\item[(ii)] if $\gamma \to 1$, then $\beta \to 1/2$, corresponding to the well known parametric detection boundary \citep{MR2135927}.
\end{itemize}
Figure~\ref{fig:db} plots the relation between $\beta$ and $\gamma$. Note that here $\beta$ is not strictly increasing with $\gamma$; pattern changes at $\gamma=1/2$ (with $\beta(1/2)=3/8$) and $\gamma=2/3$ (with $\beta(2/3)=1/3$), indicating an intriguing bias-variance tradeoff of the test with regard to the choice of $M$. 
\end{remark}

\begin{figure}[t!p]
  \centering
  \includegraphics[width=0.7\textwidth]{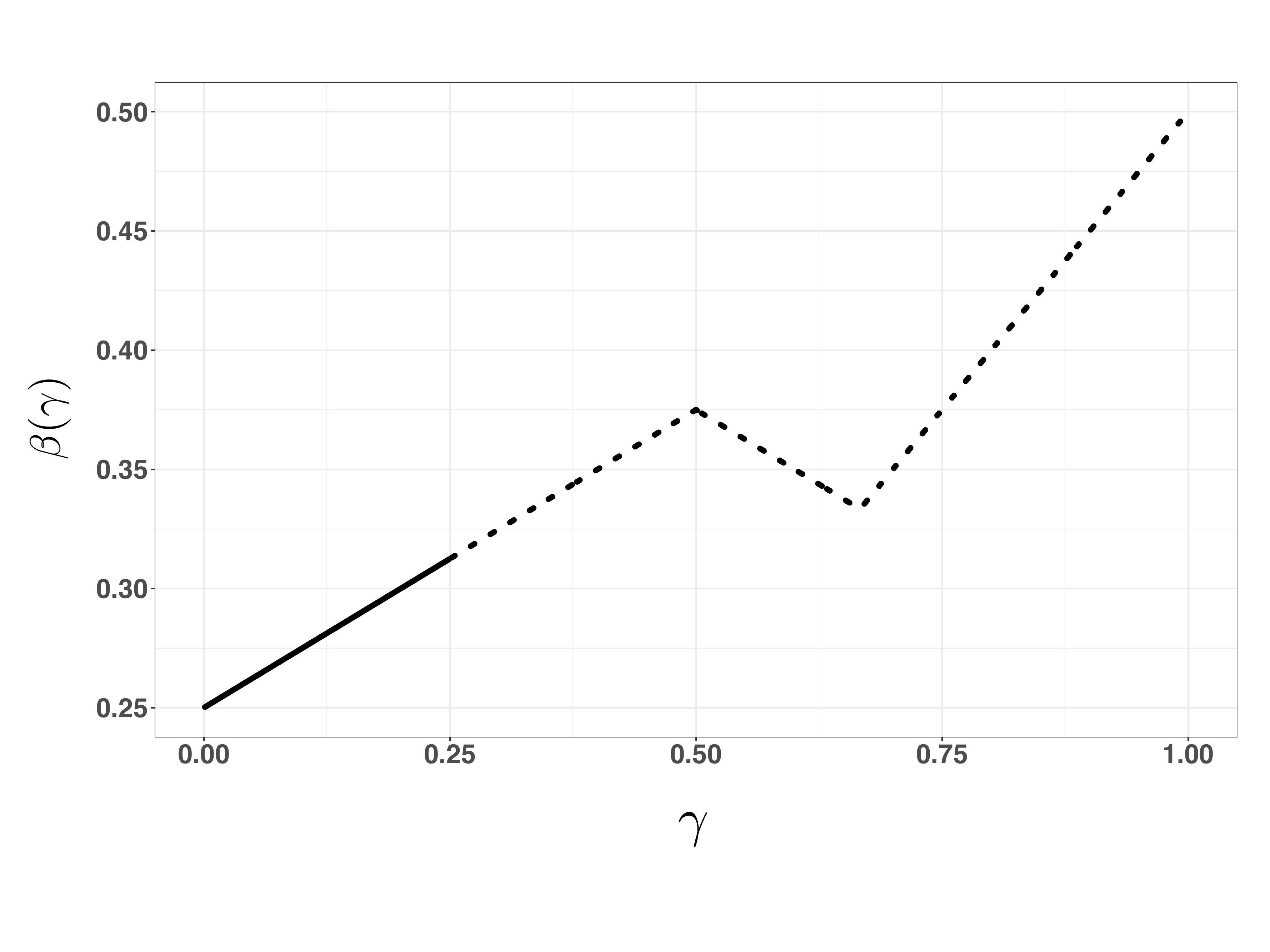} \vspace{-1cm}
  \caption{Relation between $M$ and the corresponding (sufficient) detection boundary. Here $M$ is picked to be $n^{\gamma}$ and the corresponding (sufficient) detection boundary is $\zeta_{n,M}=n^{-\beta(\gamma)}$. The regime of the solid line is the critical detection boundary of $\sT_{\alpha,B}^{\xi_{n,M}^{\pm}}$ (referred to Theorem \ref{thm:local power}\ref{thm:lpa-2}), while the regime of the dotted line is a sufficient detection boundary of $\sT_{\alpha,B}^{\xi_{n,M}^{\pm}}$ (referred to Theorem \ref{thm:local power}\ref{thm:lpa-1}). The change point from critical detection boundary to sufficient detection boundary is $\gamma = 1/4$ (with $\beta(1/4)=5/16$).} 
  \label{fig:db}
\end{figure}

\begin{remark}
\citet[Remark 4.3]{deb2020kernel} conjectured that ``allowing for growing [number of nearest neighbors] could potentially lead to information theoretically efficient estimators''. Theorem \ref{thm:local power} settles this conjecture and shows that a simulation-based test built on the proposed revised Chatterjee's rank correlations can indeed provably achieve near-parametric efficiency as pushing $M$ to be closer and closer to $n$; see also Remarks \ref{remark:deb} and \ref{remark:db} for related discussions.
\end{remark}

\begin{remark}
We conjecture that $\zeta_{n,M}$ is always --- regardless of how fast $M$ grows to infinity with $n$ --- the critical detection boundary of the proposed test in the sense that there exists a constant $C_\alpha<1$ only depending on $\alpha$ such that
\[
 \lim_{n\to\infty} \P_{H_{1,n}}\Big(\sT_{\alpha,B}^{\xi_{n,M}^{\pm}}=1\Big) < C_\alpha
\]
for any sequence $|\rho_n|\prec\zeta_{n,M}$. This conjecture is partially supported by Corollary \ref{crl:detection} and Remark \ref{remark:sharpness2}; see also \citet[Remark 3 and Proposition 1]{MR3961499}. To fully resolve it, however, one needs to obtain more information on the null distribution of $\xi_{n,M}$ beyond its mean and variance calculated in Theorem \ref{thm:null,mean,var} for those $M$'s that are large. This will be an interesting topic for future research. 
\end{remark}

\section{Simulation study} \label{sec:sim}

This section develops additional simulation results to illustrate the finite-sample performance of the developed correlation coefficients and the according tests of independence.

First, we examine the sizes and powers of the proposed tests. To this end, the following four sets of tests of independence are considered:
\begin{itemize}
\item[(T1)] the proposed test $\sT_{\alpha,B}^{\xi_{n,M}^{\pm}}$ with $M=1, 20, 100,$ and 200;
\item[(T2)] the test based on Hoeffding's $D$ \citep{MR0029139};
\item[(T3)] the classic parametric test based on Pearson's correlation coefficient;
\item[(T4)] the test proposed in \citet[Section 8.3.1]{deb2020kernel} with $M$ NNs --- both left and right directions are incorporated --- considered for each index and $M=1,20, 100,$ and $200$.
\end{itemize}

Notice that in implementing T4, following \cite{deb2020kernel} the NNs are calculated based on the ranks of $\big[X_i\big]_{i=1}^{n}$ but not the original data, and we select the right one instead of the left if ties exist. Accordingly, as $M=1$, we select the right nearest neighbor for all points except for the largest. To implement T1 and T4, we use simulation-based tests illustrated in Section \ref{sec:sim-based-test} with $B=10,000$. On the other hand, asymptotic tests are used to implement T2 and T3. Nominal levels are set to be $\alpha=0.05$ for all tests.

We perform simulation studies based on the Gaussian rotation model satisfying Assumption \ref{asp:local-alter} with $n \in \{1000,2000,5000\}$, $\rho_n = \rho_0/\sqrt{n}$, and $\rho_0 \in \{0,1,2,5\}$. The case $\rho_0=0$ corresponds to the case when $H_0$ holds, while the rest three give rise to powers in accordance with different dependence strengths that all shrink to zero. Table~\ref{tab:power} illustrates the rejection frequencies for considered tests over $1,000$ replicates. Three observations are in line. (i) All the tests considered have empirical sizes close to 0.05, indicating that they are all size valid. (ii) The power of T1 increases as $M$ increases, and is close to that of T2 and T3 when $M$ is large. However, the power of T1 decreases for every considered $M$, while that of T2 and T3 remain stable.  Both observations are in line with the theoretical observations made earlier in Theorem \ref{thm:local power}. (iii) For each $M$ set, the power of T1 dominates that of T4, which echos Remark~\ref{remark:deb}.


Secondly, we compare the computation times for $\xi_{n,M}^{\pm}$ with different $n$ and $M$ chosen as before. To this end, we consider a simple bivariate standard Gaussian model and calculate the averaged computation time for each pair of $(n,M)$ over $1,000$ replicates. All experiments are implemented on a laptop with an Apple M1 processor and a 16GB memory. Table~\ref{tab:time} illustrates the computation times for the consider tests, which are compared to these of Hoeffding's $D$ and Pearson sample correlation coefficient. We observe that the computation time indeed increases, and is approximately linear, with regard to $M$.

Lastly, we examine the trajectory of $\xi_{n,M}$ as the data are generated from a bivariate Gaussian distribution with marginal mean 0, variance 1, and correlation $\rho=0, 0.2, 0.4, 0.6, 0.8,$ and $1$. To this end, Figure~\ref{fig:as}  illustrates the boxplots of $\xi_{n,M}$ as $n$ changes from $1000, 2000,$ to 5000, and $M$ changes from $1, 20, 100,$ to 200 over 1,000 replicates. For comparison purpose, Figure~\ref{fig:as} also plots the curve of the population correlation measure $\xi$ as a function of $\rho\in[0,1]$. Three observations are in line. (i) For any $M$ considered, the averaged $\xi_{n,M}$ gets closer and closer to the correlation measure as $n$ increases, which supports Theorem~\ref{thm:asconverge}. (ii) As $\rho=0$, the empirical variance of $\xi_{n,M}$ first decreases and then increases as $M$ increases, which supports Theorem~\ref{thm:alter,mean,var}\ref{thm:alter,var} and Remark~\ref{remark:sharpness}. On the other hand, as $\rho$ becomes large, the empirical variance of $\xi_{n,M}$ turns to be stable and unchanged with $M$. (iii) A bias term exists as $M$ is relatively large compared to $n$, but will shrink towards 0 for each fixed $M$ as $n$ increases. This is as expected (cf. Theorem \ref{thm:asconverge}) and is a common occurrence in nonparametric statistics problems.

{
\renewcommand{\tabcolsep}{1.5pt}
\renewcommand{\arraystretch}{1.1}
\begin{table}[t!p]
\aboverulesep=0ex
\belowrulesep=0ex
\centering
\caption{Rejection frequencies of the tests over 1,000 replicates.}{
\begin{tabular}{C{.5in}C{.75in}C{.75in}C{.75in}C{.75in}C{.75in}C{1in}C{.75in}}
\toprule
\multirow{2}{*}{$\rho_0$} & \multirow{2}{*}{$n$}
   & \multicolumn{4}{c}{$\xi_{n,M}^{\pm}$} & \multirow{2}{*}{Hoeffding}&  \multirow{2}{*}{Pearson} \\
   \cline{3-6}
    &   & $M=1$  & $M=20$  &  $M=100$  & $M=200$ &  &  \\
\midrule
 0 & 1000 & 0.056 & 0.057 & 0.045 & 0.048 & 0.045 & 0.049\\
 & 2000 & 0.056 & 0.057 & 0.040 & 0.054 & 0.057 & 0.044\\
 & 5000 & 0.049 & 0.049 & 0.055 & 0.049 & 0.059 & 0.045\\
\midrule
 1 & 1000 & 0.040 & 0.069 & 0.119 & 0.161 & 0.151 & 0.152\\
 & 2000 & 0.062 & 0.061 & 0.099 & 0.158 & 0.156 & 0.186\\
 & 5000 & 0.046 & 0.054 & 0.081 & 0.091 & 0.137 & 0.152\\
\midrule
 2 & 1000 & 0.075 & 0.154 & 0.365 & 0.427 & 0.422 & 0.528\\
 & 2000 & 0.058 & 0.140 & 0.262 & 0.336 & 0.423 & 0.518\\
 & 5000 & 0.059 & 0.084 & 0.170 & 0.229 & 0.431 & 0.511\\
\midrule
 5 & 1000 & 0.176 & 0.851 & 0.982 & 0.997 & 0.993 & 0.999\\
 & 2000 & 0.131 & 0.706 & 0.964 & 0.980 & 0.998 & 1.000\\
 & 5000 & 0.089 & 0.413 & 0.847 & 0.897 & 0.996 & 1.000\\
\midrule
&  & \multicolumn{4}{c}{\cite{deb2020kernel}} & & \\
   \cline{3-6}
    &   & $M=1$  & $M=20$  &  $M=100$  & $M=200$ &  &  \\
\cline{1-6}
0 & 1000 & 0.054 & 0.047 & 0.047 & 0.050 & & \\
& 2000 & 0.056 & 0.056 & 0.050 & 0.054 & & \\
& 5000 & 0.050 & 0.047 & 0.052 & 0.055 & & \\
\cline{1-6}
1 & 1000 & 0.044 & 0.054 & 0.071 & 0.044 & & \\
& 2000 & 0.062 & 0.057 & 0.059 & 0.060 & & \\
& 5000 & 0.046 & 0.050 & 0.059 & 0.076 & & \\
\cline{1-6}
2 & 1000 & 0.077 & 0.109 & 0.138 & 0.106 & & \\
& 2000 & 0.059 & 0.089 & 0.110 & 0.117 & & \\
& 5000 & 0.059 & 0.067 & 0.091 & 0.093 & & \\
\cline{1-6}
5 & 1000 & 0.188 & 0.626 & 0.779 & 0.633 & & \\
& 2000 & 0.128 & 0.413 & 0.710 & 0.635 & & \\
& 5000 & 0.096 & 0.267 & 0.528 & 0.582 & & \\
\bottomrule
\end{tabular}}
\label{tab:power}
\end{table}
}

{
\renewcommand{\tabcolsep}{1.5pt}
\renewcommand{\arraystretch}{1.0}
\begin{table}[!htbp]
\aboverulesep=0ex
\belowrulesep=0ex
\centering
\caption{Computation times of $\xi_{n,M}^{\pm}$, Hoeffding's $D$, and Pearson's correlation coefficient. The computation times here are in $10^{-2}$ seconds and averaged over 1,000 replicates.}{
\begin{tabular}{C{.75in}C{.75in}C{.75in}C{.75in}C{.75in}C{1in}C{.75in}}
\toprule
   & \multicolumn{4}{c}{$\xi_{n,M}^{\pm}$} & \multirow{2}{*}{Hoeffding}&  \multirow{2}{*}{Pearson} \\
   \cline{2-5}
    &    $M=1$  & $M=20$  &  $M=100$  & $M=200$ &  &  \\
\midrule
$n=1000$ & 0.03 & 0.27 & 1.32 & 2.60 & 0.02 & 0.07\\
$n=2000$ & 0.06 & 0.54 & 2.44 & 4.76 & 0.03 & 0.13\\
$n=5000$ & 0.12 & 1.20 & 5.82 & 11.57 & 0.06 & 0.32\\
\bottomrule
\end{tabular}}
\label{tab:time}
\end{table}
}

\begin{figure}[!htbp]
  \centering
  \includegraphics[width=0.85\textwidth]{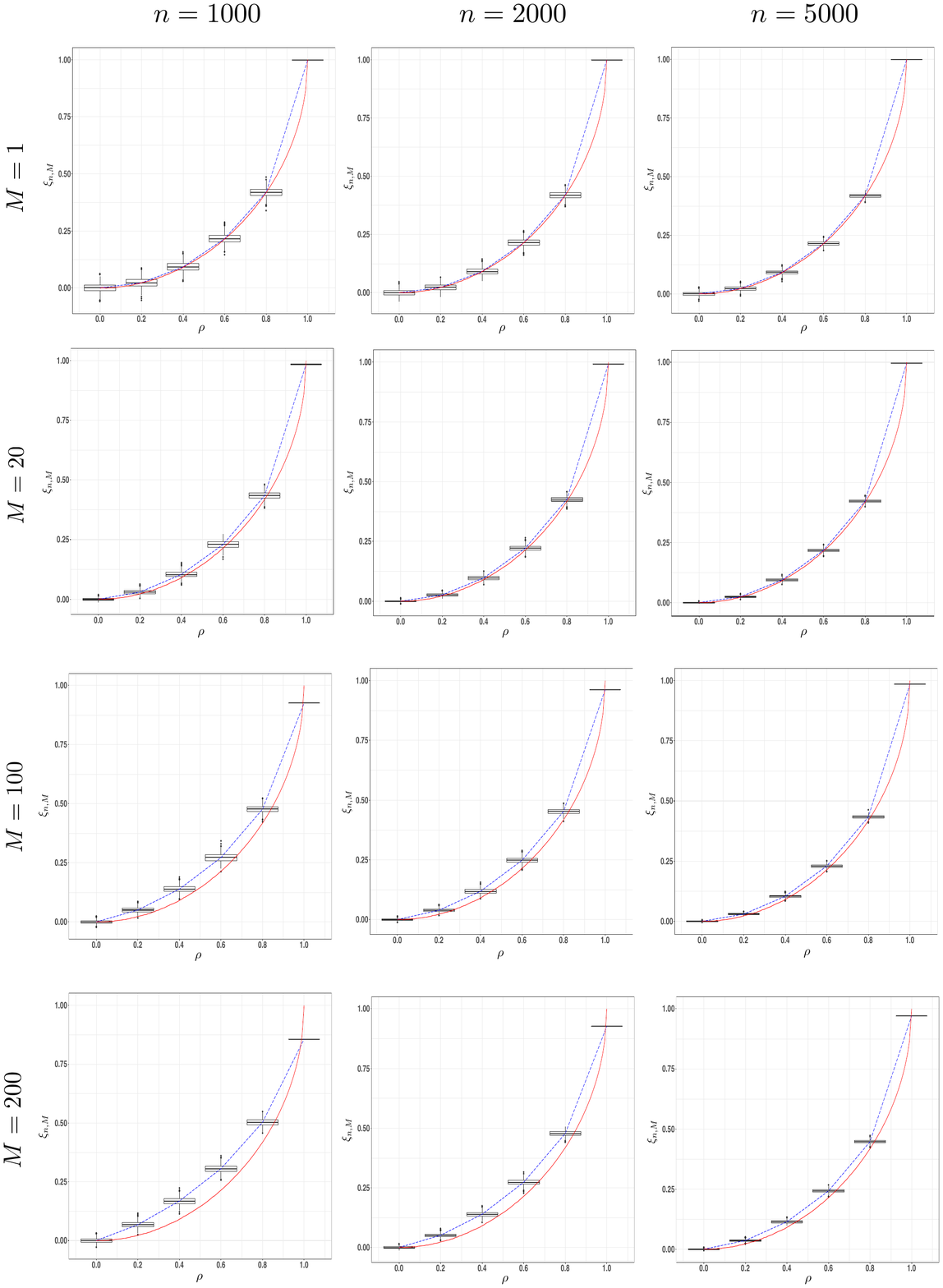}
  \vspace{-1cm}
  \caption{Statistics of the calculated $\xi_{n,M}$'s corresponding to $n=1000, 2000, 5000$ (from left to right) and $M=1,20,100, 200$ (
  from top to bottom) over 1,000 replicates. In each figure the red solid line stands for the theoretical $\xi(\rho)$, the blue dash line stands for the averaged $\xi_{n,M}$ at $\rho=0, 0,2, \ldots, 1$, and in each boxplot the two hinges represent the 25\% and 75\% quantiles.}
  \label{fig:as}
\end{figure}


\section{Proof of the main results}\label{sec:main-proof}

This section provides the proof of Theorems \ref{thm:null,mean,var} and \ref{thm:alter,mean,var}. 
In the following, we use $\mX$ to represent $(X_1,X_2,\ldots,X_n)$. For any function $f$, we use $\lVert f \rVert_{\infty}$ to represent its supremum norm.

\subsection{Proof of Theorem~\ref{thm:null,mean,var}}

\begin{proof}[Proof of Theorem~\ref{thm:null,mean,var}]


The proof of Theorem \ref{thm:null,mean,var} is based on the following lemma, which provides some necessary information on ranks.

\begin{lemma}\label{lemma:perm}
Recall the definition of $R_i$'s in \eqref{eq:Ri}. Assuming $Y$ is continuous, $[R_1, R_2,\ldots, R_n]$ then follows a random permutation satisfying
 \[
\P(R_1=i_1,R_2=i_2,\ldots,R_n=i_n)=1/n!,~~~\text{for each } i_1\ne i_2\ne\ldots\ne i_n\in\zahl{n}.
 \]
Furthermore,
  \begin{align*}
    &\E [R_1] = \frac{n+1}{2},~ \E [\min\{R_1,R_2\}] = \frac{n+1}{3},~  \Var [R_1] = \frac{(n+1)(n-1)}{12},~ \Cov[R_1,R_2] = -\frac{n+1}{12},\\
    & \Cov[R_1,\min\{R_2,R_3\}] = -\frac{n+1}{12},~ \Cov[R_1,\min\{R_1,R_2\}] = \frac{(n-2)(n+1)}{24},\\
    & \Cov[\min\{R_1,R_2\},\min\{R_3,R_4\}] = -\frac{4(n+1)}{45},~ \Cov[\min\{R_1,R_2\},\min\{R_1,R_3\}] = \frac{(n+1)(4n-17)}{180},\\
    & \Var[\min\{R_1,R_2\}] = \frac{(n-2)(n+1)}{18}.
  \end{align*}
\end{lemma}

\vspace{0.5cm}

{\bf Step I.} This step calculates $\E_{H_0}[\xi_{n,M}]$ (in the sequel shorthanded as $\E[\xi_{n,M}]$). We have
\begin{align*}
  \E [\xi_{n,M}] &= \E \Big[-2 + \frac{6 \sum_{i=1}^{n} \sum_{m=1}^M \min\{R_{j_m(i)},R_i\} }{(n+1)[nM+M(M+1)/4]} \Big]\\
  &= -2 + \frac{6}{(n+1)[nM+M(M+1)/4]} \E \Big[ \sum_{i=1}^{n} \sum_{m=1}^M \min\{R_{j_m(i)},R_i\} \Big]\\
  &= -2 + \frac{6n}{(n+1)[nM+M(M+1)/4]} \sum_{m=1}^M \E \Big[ \min\{R_{j_m(1)},R_1\} \Big].
  \yestag\label{eq:null,mean}
\end{align*}
For any $m \in \zahl{M}$,
\begin{align*}
  & \E \Big[ \min\{R_{j_m(1)},R_1\} \Big]\\
  = & \E \Big[ \min\{R_{j_m(1)},R_1\} \ind(j_m(1) \neq 1) \Big] + \E \Big[ \min\{R_{j_m(1)},R_1\} \ind(j_m(1) = 1) \Big]\\
  = & \E \Big[ \min\{R_{j_m(1)},R_1\} \ind(j_m(1) \neq 1) \Big] + \E \Big[ R_1 \ind(j_m(1) = 1) \Big]\\
  = & \P \Big(j_m(1) \neq 1 \Big) \E \Big[ \min\{R_{j_m(1)},R_1\} \Biggiven j_m(1) \neq 1 \Big] + \P \Big(j_m(1) = 1 \Big) \E \Big[ R_1 \Biggiven j_m(1) = 1 \Big]\\
  = & \P \Big(j_m(1) \neq 1 \Big) \E \Big[ \E \Big[ \min\{R_{j_m(1)},R_1\} \Biggiven j_m(1) \neq 1 , \mX \Big] \Biggiven j_m(1) \neq 1 \Big]\\
  & + \P \Big(j_m(1) = 1 \Big) \E \Big[ \E \Big[ R_1 \Biggiven j_m(1) = 1 , \mX \Big] \Biggiven j_m(1) = 1 \Big].
\end{align*}
Conditional on $\{j_m(1) \neq 1\}$ and $\mX$, $j_m(1)$ is an index different from $1$. Under independence between $X$ and $Y$, $[R_i]_{i=1}^n$ is independent of $\mX$. Then from Lemma~\ref{lemma:perm},
\[
  \E \Big[ \min\{R_{j_m(1)},R_1\} \Biggiven j_m(1) \neq 1, \mX \Big] = \frac{n+1}{3}
~~~{\rm and}~~~
  \E \Big[ R_1 \Biggiven j_m(1) = 1 , \mX \Big] = \frac{n+1}{2}.
\]

On the other hand, for any $\mX$, the cardinality of $\{i: j_m(i) = i\}$ is $m$. Then $\P (j_m(1) = 1 ) = m/n$ and $\P (j_m(1) \neq 1 ) = 1 - m/n$ since $[X_i]_{i=1}^n$ are i.i.d. Then
\[
  \E \Big[ \min\{R_{j_m(1)},R_1\} \Big] = \frac{n+1}{3} \Big(1-\frac{m}{n}\Big) + \frac{n+1}{2} \frac{m}{n} = \frac{n+1}{3} + \frac{n+1}{6} \frac{m}{n}.
\]
Substituting it into \eqref{eq:null,mean} yields
\[
  \E [\xi_{n,M}] = -2 + \frac{6n}{(n+1)[nM+M(M+1)/4]} \sum_{m=1}^M \Big[ \frac{n+1}{3} + \frac{n+1}{6} \frac{m}{n} \Big] = 0.
\]

{\bf Step II.} This step calculates the variance of $\xi_{n,M}$ under $H_0$ (shorthanded as $\Var[\xi_{n,M}]$). For this, we have
\begin{align*}
  \Var [\xi_{n,M}] =& \Var \Big[-2 + \frac{6 \sum_{i=1}^{n} \sum_{m=1}^M \min\{R_{j_m(i)},R_i\} }{(n+1)[nM+M(M+1)/4]} \Big]\\
  =& \frac{36}{(n+1)^2[nM+M(M+1)/4]^2} \Var \Big[ \sum_{i=1}^{n} \sum_{m=1}^M \min\{R_i,R_{j_m(i)}\} \Big]\\
  =& \frac{36}{(n+1)^2[nM+M(M+1)/4]^2} \Big\{ \E \Big[ \Var \Big[ \sum_{i=1}^{n} \sum_{m=1}^M \min\{R_i,R_{j_m(i)}\} \Biggiven \mX \Big] \Big] +  \\
  &\Var \Big[ \E \Big[ \sum_{i=1}^{n} \sum_{m=1}^M \min\{R_i,R_{j_m(i)}\} \Biggiven \mX \Big] \Big]\Big\}.
  \yestag\label{eq:null,var}
\end{align*}

Under the independence of $X$ and $Y$, for the second term in \eqref{eq:null,var}, we have
\[
  \Var \Big[ \E \Big[ \sum_{i=1}^{n} \sum_{m=1}^M \min\{R_i,R_{j_m(i)}\} \Biggiven \mX \Big] \Big] = 0.
\]

We then decompose the first term in \eqref{eq:null,var} as
\begin{align*}
  & \Var \Big[ \sum_{i=1}^{n} \sum_{m=1}^M \min\{R_i,R_{j_m(i)}\} \Biggiven \mX \Big]\\
  =& \sum_{i=1}^{n} \sum_{m=1}^M \Var \Big[ \min\{R_i,R_{j_m(i)}\} \Biggiven \mX \Big] + \sum_{i=1}^{n} \sum^M_{\substack{m,m'=1\\m \neq m'}} \Cov \Big[ \min\{R_i,R_{j_m(i)}\}, \min\{R_i,R_{j_{m'}(i)}\} \Biggiven \mX \Big]\\
  & + \sum^n_{\substack{i,l=1\\i \neq l}} \sum^M_{m,m'=1} \Cov \Big[ \min\{R_i,R_{j_m(i)}\}, \min\{R_l,R_{j_{m'}(l)}\} \Biggiven \mX \Big].
  \yestag\label{eq:null,var,decompose}
\end{align*}

We consider each term in \eqref{eq:null,var,decompose} seperately and proceed in three substeps. From the independece of $[R_i]_{i=1}^n$ and $\mX$, we assume $[X_i]_{i=1}^n$ is increasing without loss of generality.

{\bf Step II-1.} For any $\mX$ and $m \in \zahl{M}$, the number of $i \in \zahl{n}$ such that $j_m(i)=i$ is $m$. Then the number of pairs $(i,m)$ among $\{(i,m): i \in \zahl{n},m \in \zahl{M}\}$ such that $j_m(i) = i$ is $M(M+1)/2$. Then
\begin{align*}
  & \sum_{i=1}^{n} \sum_{m=1}^M \Var \Big[ \min\{R_i,R_{j_m(i)}\} \Biggiven \mX \Big] 
  =  \frac{M(M+1)}{2} \Var [R_1] + \Big[ nM - \frac{M(M+1)}{2} \Big] \Var [\min\{R_1,R_2\}].
\end{align*}

{\bf Step II-2.} For any $m\ne m' \in \zahl{M}$, the number of $i \in \zahl{n}$ such that $j_m(i)=i,j_{m'}(i)=i$ is $\min\{m,m'\}$, and the number of $i$ such that $j_m(i)=i,j_{m'}(i)\neq i$ or $j_m(i) \neq i,j_{m'}(i)= i$ is $\lvert m-m' \rvert$. Then
\begin{align*}
  & \sum_{i=1}^{n} \sum^M_{\substack{m,m'=1\\m \neq m'}} \Cov \Big[ \min\{R_i,R_{j_m(i)}\}, \min\{R_i,R_{j_{m'}(i)}\} \Biggiven \mX \Big] \\
  = & \frac{(M-1)M(M+1)}{3} \Var [R_1] + \frac{(M-1)M(M+1)}{3} \Cov[R_1,\min\{R_1,R_2\}]\\
  & + \Big[ nM(M-1) - \frac{2}{3} (M-1)M(M+1)\Big] \Cov[\min\{R_1,R_2\},\min\{R_1,R_3\}].
\end{align*}

{\bf Step II-3.} This substep is relatively sophisticated. Assume $i,l \in \zahl{n}$ and $i<l$. Then for any $m,m' \in \zahl{M}$, the possible cases for $(i,l,m,m')$, the number of such $(i,l,m,m')$ and the corresponding covariance value are as followed:
\begin{enumerate}[itemsep=-.5ex,label=(\alph*)]
  \item $j_m(i) \neq i,j_{m'}(l) \neq l,j_m(i) = j_{m'}(l)$. The value is $\Cov[\min\{R_1,R_2\},\min\{R_1,R_3\}]$.

  Let $j_m(i) = j_{m'}(l) = k$. For any $k \ge M+1$, there can be $\binom{M}{2}$ pairs $(i,l)$ and the corresponding $(m,m')=(k-i,k-l)$. For any $3 \le k \le M$, there can be $\binom{k-1}{2}$ pairs $(i,l)$. Then the total number is
  \[
    (n-M)\binom{M}{2} + \sum_{k=3}^M \binom{k-1}{2} = \frac{(n-M)(M-1)M}{2} + \frac{(M-2)(M-1)M}{6}.
  \]

  \item $j_m(i) \neq i,j_{m'}(l) \neq l,j_m(i) = l$. The value is $\Cov[\min\{R_1,R_2\},\min\{R_1,R_3\}]$.

  Let $j_m(i) = l = k$. For any $M+1 \le k \le n-M$, $m,m'$ are arbitrary. Then there can be $M^2$ pairs $(m,m')$ and the corresponding $(i,l)=(k-m,k)$. For any $k \le M$, $m'$ is arbitrary, but $m \le k-1$. Then there can be $(k-1)M$ pairs $(m,m')$. For any $k \ge n-M+1$, $m$ is arbitrary, but $m' \le n-k$. Then there can be $(n-k)M$ pairs $(m,m')$. Then the total number is
  \[
    (n-2M)M^2 + \sum_{k=1}^M (k-1)M + \sum_{k=n-M+1}^{n-1} (n-k)M = (n-2M)M^2 + (M-1)M^2.
  \]

  \item $j_m(i) \neq i,j_{m'}(l) \neq l,j_m(i) \neq l, j_m(i) \neq j_{m'}(l)$. The value is $\Cov[\min\{R_1,R_2\},\min\{R_3,R_4\}]$.

  The total number is $n(n-1)M^2/2$ minus the sum of other cases.

  \item $j_m(i) = i,j_{m'}(l) \neq l$. The value is $\Cov[R_1,\min\{R_2,R_3\}]$.

  For any pairs $(i,l)$ and $i \ge n-M+1$, the number of $m$ such that $j_m(i) = i$ is $i-(n-M)$, and the number of $m'$ such that $j_{m'}(l) \neq l$ is $n-l$. Then there can be $[i-(n-M)](n-l)$ pairs $(m,m')$. Then the total number is
  \[
    \sum^n_{\substack{i,l=n-M+1\\i < l}} [i-(n-M)](n-l) = \frac{(M-2)(M-1)M(M+1)}{24}.
  \]

  \item $j_m(i) \neq i,j_{m'}(l) = l,j_m(i) = l$. The value is $\Cov[R_1,\min\{R_1,R_2\}]$.

  For any $l \ge n-M+1$, the number of $m'$ such that $j_{m'}(l) = l$ is $l-(n-M)$. The number of $i$ such that $j_m(i) = l$ is $M$ and the corresponding $m=l-i$. Then there can be $[l-(n-M)]M$ pairs $(m,m')$. Then the total number is
  \[
    \sum_{l=n-M+1}^n [l-(n-M)]M = \frac{M^2(M+1)}{2}.
  \]

  \item $j_m(i) \neq i,j_{m'}(l) = l,j_m(i) \neq l$. The value is $\Cov[R_1,\min\{R_2,R_3\}]$.

  We first consider the number of $(i,l,m,m')$ such that $j_m(i) \neq i,j_{m'}(l) = l$.

  For any $l \ge n-M+1$, the number of $m'$ such that $j_{m'}(l) = l$ is $l-(n-M)$. For any $i \le n-M$, the number of $m$ such that $j_m(i) \neq i$ is $M$. For any $n-M+1 \le i < l$, the number of $m$ such that $j_m(i) \neq i$ is $n-i$. Then for any $l \ge n-M+1$, the number of pairs $(i,m)$ is
  \[
    (n-M)M + \sum_{i=n-M+1}^{l-1} (n-i) = (l-1)M - \frac{[l-(n-M)][l-(n-M)-1]}{2}.
  \]

  Then the total number of $j_m(i) \neq i,j_{m'}(l) = l$ is
  \begin{align*}
    & \sum_{l=n-M+1}^n \Big[ (l-1)M - \frac{[l-(n-M)][l-(n-M)-1]}{2} \Big] \Big[l-(n-M) \Big]\\
    = & \frac{nM^2(M+1)}{2} - \frac{M(M+1)(7M^2+7M-2)}{24}.
  \end{align*}

  Combined with (e), the total number of this case is
  \[
    \frac{nM^2(M+1)}{2} - \frac{M(M+1)(7M^2+19M-2)}{24}.
  \]

  \item $j_m(i) = i,j_{m'}(l) = l$. The value is $\Cov[R_1,R_2]$.

  For any pairs $(i,l)$ and $i \ge n-M+1$, the number of pairs $(m,m')$ such that $j_m(i) = i,j_{m'}(l) = l$ is $[i-(n-M)][l-(n-M)]$. Then the total number is
  \[
    \sum^n_{\substack{i,l=n-M+1\\i < l}} [i-(n-M)][l-(n-M)] = \frac{(M-1)M(M+1)(3M+2)}{24}.
  \]

\end{enumerate}

Notice that we assume $i<l$, and the number for each case is the same for $i>l$.

Together with \eqref{eq:null,var,decompose}, Lemma~\ref{lemma:perm} and the number of each case, we obtain
\begin{align*}
  & \Var \Big[ \sum_{i=1}^{n} \sum_{m=1}^M \min\{R_i,R_{j_m(i)}\} \Biggiven \mX \Big]\\
  = & nM \Big(\frac{1}{18}n^2\Big) + \frac{1}{3}M^3 \Big(\frac{1}{12}n^2\Big) + \frac{1}{3}M^3 \Big(\frac{1}{24}n^2\Big) + \Big( nM^2 - nM - \frac{2}{3}M^3 \Big) \Big(\frac{1}{45}n^2\Big)\\
  & + \Big( nM^2 - nM - \frac{2}{3}M^3 \Big) \Big(\frac{1}{45}n^2\Big) + \Big( 2nM^2 - 2M^3 \Big) \Big(\frac{1}{45}n^2\Big) + \Big( n^2 M^2 - nM^3 \Big) \Big(-\frac{4}{45}n\Big)\\
  & + M^3 \Big(\frac{1}{24}n^2\Big) + nM^3 \Big(-\frac{1}{12}n\Big) + o(n^3M) + o(n^2M^3)\\
  = & \Big[ \frac{1}{90}n^3M + \frac{2}{135}n^2M^3 \Big] (1+o(1)).
\end{align*}

Substituting them to \eqref{eq:null,var}, we have obtained
\[
  \Var [\xi_{n,M}] = \Big[ \frac{2}{5}\Big( \frac{1}{nM} \Big) + \frac{8}{15}\Big( \frac{M}{n^2} \Big) \Big] (1+o(1)),
\]
and thus finished the proof.
\end{proof}

\subsection{Proof of Theorem~\ref{thm:alter,mean,var}}

\begin{proof}[Proof of Theorem~\ref{thm:alter,mean,var}\ref{thm:alter,mean}] Resembling the proof of Theorem \ref{thm:null,mean,var}, in the following we shorthand $\E_{H_1}[\xi_{n,M}]$ and $\Var_{H_1}[\xi_{n,M}]$ by $\E[\xi_{n,M}]$ and $\Var[\xi_{n,M}]$.

From \eqref{eq:xin},
\begin{align*}
  \E [\xi_{n,M}] &= \E \Big[-2 + \frac{6 \sum_{i=1}^{n} \sum_{m=1}^M \min\{R_{j_m(i)},R_i\} }{(n+1)[nM+M(M+1)/4]} \Big]\\
  &= -2 + \frac{6}{(n+1)[nM+M(M+1)/4]} \E \Big[ \sum_{i=1}^{n} \sum_{m=1}^M \min\{R_{j_m(i)},R_i\} \Big].
  \yestag\label{eq:local,eq1}
\end{align*}
Notice that, for any $i \in \zahl{n}$ and $m \in \zahl{M}$, if $j_m(i) \neq i$, then
\[
  \min\{R_i,R_{j_m(i)}\} = \sum_{k=1}^n \ind(Y_k \le \min\{Y_i,Y_{j_m(i)}\}) = 1 + \sum_{k \neq i, k \neq j_m(i)} \ind(Y_k \le \min\{Y_i,Y_{j_m(i)}\});
\]
if $j_m(i) = i$, then
\[
  \min\{R_i,R_{j_m(i)}\} = R_i = \sum_{k=1}^n \ind(Y_k \le Y_i) = 1 + \sum_{k \neq i} \ind(Y_k \le Y_i).
\]

We accordingly have
\begin{align*}
  & \frac{1}{nM} \E \Big[ \sum_{i=1}^{n} \sum_{m=1}^M \min\{R_i,R_{j_m(i)}\} \Big] \\
  = & \frac{1}{nM} \E \Big[ \sum_{i=1}^{n} \sum_{m=1}^M \Big(\min\{R_i,R_{j_m(i)}\} \ind(j_m(i) \neq i) + R_i \ind(j_m(i) = i) \Big) \Big]\\
  = & \frac{1}{M} \E \Big[  \sum_{m=1}^M \Big( \min\{R_1,R_{j_m(1)}\} \ind(j_m(1) \neq 1) + R_1 \ind(j_m(1) = 1) \Big) \Big] \\
  = & \E \Big[ \min\{R_1,R_{j_U(1)}\} \ind(j_U(1) \neq 1) + R_1 \ind(j_U(1) = 1) \Big]\\
  = & \E \Big[ \Big[ 1 + \sum_{k \neq 1, k \neq j_U(1)} \ind(Y_k \le \min\{Y_1,Y_{j_U(1)}\}) \Big] \ind(j_U(1) \neq 1) \Big] + \E \Big[ \Big[ 1 + \sum_{k \neq 1} \ind(Y_k \le Y_1) \Big] \ind(j_U(1) = 1) \Big]\\
  = & 1 + \E \Big[ \sum_{k \neq 1, k \neq j_U(1)} \ind(Y_k \le \min\{Y_1,Y_{j_U(1)}\}) \ind(j_U(1) \neq 1) \Big] + \E \Big[ \sum_{k \neq 1} \ind(Y_k \le Y_1) \ind(j_U(1) = 1) \Big]\\
  = & 1 + (n-2) \E \Big[\ind(Y \le \min\{Y_1,Y_{j_U(1)}\}) \ind(j_U(1) \neq 1) \Big] + (n-1) \E \Big[\ind(Y \le Y_1) \ind(j_U(1) = 1) \Big]\\
  = & 1 + \E \Big[\ind(Y \le Y_1) \ind(j_U(1) = 1) \Big] + (n-2) \E \Big[\ind(Y \le \min\{Y_1,Y_{j_U(1)}\}) \Big],
  \yestag\label{eq:local,eq2}
\end{align*}
where $U$ follows a uniform distribution over $\zahl{m}$ (cf. \eqref{eq:as,U}) and $Y \sim F_Y$ is independent of $\big[(X_i,Y_i)\big]_{i=1}^n$. Then it suffices to establish the rate of
\[
  \E \Big[\ind(Y \le \min\{Y_1,Y_{j_U(1)}\}) \Big].
\]

We consider the expectation conditional on $\mX$ and $U$. Denote the conditional distribution of $Y$ given $X$ by $F_{Y|X}$ and the density function of $Y$ by $f_Y$. Then conditional on $\mX$ and $U$,
\begin{align*}
  & \E \Big[ \ind(Y \le \min\{Y_1,Y_{j_U(1)}\}) \Biggiven \mX,U \Big]\\
  = & \E \Big[ \ind(Y \le \min\{Y_1,Y_{j_U(1)}\}) \ind(j_U(1) \neq 1) \Biggiven \mX,U \Big] + \E \Big[ \ind(Y \le \min\{Y_1,Y_{j_U(1)}\}) \ind(j_U(1) = 1) \Biggiven \mX,U \Big]\\
  = & \int \Big[1-F_{Y|X=X_1}(y)\Big] \Big[1-F_{Y|X=X_{j_U(1)}}(y)\Big] f_Y(y) \d y \ind(j_U(1) \neq 1) \\
  & + \int \Big[1-F_{Y|X=X_1}(y)\Big] f_Y(y) \d y \ind(j_U(1) = 1)\\
  = & \int \Big[1-F_{Y|X=X_1}(y)\Big] \Big[1-F_{Y|X=X_{j_U(1)}}(y)\Big] f_Y(y) \d y \\
  & + \int \Big[1-F_{Y|X=X_1}(y)\Big] F_{Y|X=X_1}(y) f_Y(y) \d y \ind(j_U(1) = 1)\\
  = & \int \Big[1 - F_Y(y) \Big]^2 f_Y(y) \d y + 2 \int \Big[1 - F_Y(y)\Big] \Big[F_Y(y) -F_{Y|X=X_1}(y)\Big] f_Y(y) \d y\\
  & + \int \Big[F_Y(y) -F_{Y|X=X_1}(y)\Big]^2 f_Y(y) \d y + \int \Big[1 - F_Y(y)\Big] \Big[F_{Y|X=X_1}(y) - F_{Y|X=X_{j_U(1)}}(y) \Big] f_Y(y) \d y \\
  & + \int \Big[F_Y(y) -F_{Y|X=X_1}(y)\Big] \Big[F_{Y|X=X_1}(y) - F_{Y|X=X_{j_U(1)}}(y) \Big] f_Y(y) \d y\\
  & + \int \Big[1-F_Y(y)\Big] F_Y(y) f_Y(y) \d y \ind(j_U(1) = 1)\\
  & + \int \Big[F_Y(y) - F_{Y|X=X_1}(y)\Big] \Big[F_Y(y) + F_{Y|X=X_1}(y) - 1\Big] f_Y(y) \d y \ind(j_U(1) = 1)\\
  =: & T_1 + 2T_2 + T_3 + T_4 + T_5 + T_6 + T_7.
  \yestag\label{eq:local,taylor}
\end{align*}

For $T_1$, we have
\[
  T_1 = \int \Big[1 - F_Y(y) \Big]^2 f_Y(y) \d y = \P(Y_1 \le \min\{Y_2,Y_3\}) = \frac{1}{3},
  \yestag\label{eq:local,T1}
\]
where $Y_1,Y_2,Y_3$ are three independent copies of $Y$ from $F_Y$.

For $T_2$, from Fubini's theorem,
\begin{align*}
  \E [T_2] &= \E \Big[ \int \Big[1 - F_Y(y)\Big] \Big[F_Y(y) -F_{Y|X=X_1}(y)\Big] f_Y(y) \d y \Big]\\
  &= \int \Big[1 - F_Y(y)\Big] \E \Big[F_Y(y) -F_{Y|X=X_1}(y)\Big] f_Y(y) \d y =0,
  \yestag\label{eq:local,T2}
\end{align*}
since $X_1$ is from $F_X$ and is independent of $U$.




For $T_3$, we have the following lemma.
\begin{lemma}\label{lemma:local,T3} We have
  \[
    \E [T_3] = \Big[ \int f_Y^3(y) \d y \Big] \rho_n^2 + o(\rho_n^2).
  \]
\end{lemma}

For $T_4$ and $T_5$, we have the following two lemmas.
\begin{lemma}\label{lemma:local,T4} We have $T_4\geq0$ and
  \[
    \frac{M}{n}\rho_n + o(\rho_n^2) \lesssim E[T_4] \lesssim \frac{M}{n} \sqrt{\log n} \rho_n + \frac{M}{n^2} + o(\rho_n^2).
  \]
\end{lemma}

\begin{lemma}\label{lemma:local,T5}
  \[
    \E [\lvert T_5 \rvert] \lesssim \frac{M}{n^2} + o(\rho_n^2).
  \]
\end{lemma}




For $T_6$,
\begin{align*}
  \E [T_6] &= \int \Big[1-F_Y(y)\Big] F_Y(y) f_Y(y) \d y \E [\ind(j_U(1) = 1)] \\
  &= \P(Y_1 \le Y_2 \le Y_3) \Big[ \frac{1}{M} \sum_{m=1}^M  \E [\ind(j_m(1) = 1)] \Big] \\
  & = \frac{1}{6} \Big[ \frac{1}{M} \sum_{m=1}^M \frac{m}{n} \Big]= \frac{M+1}{12n},
  \yestag\label{eq:local,T6}
\end{align*}
where $Y_1,Y_2,Y_3$ are three independent copies of $Y$ from $F_Y$.

For $T_7$, we have the following lemma.

\begin{lemma}\label{lemma:local,T7}
  $
    \lvert E[T_7] \rvert = o(\rho_n^2).
  $
\end{lemma}

Summing up \eqref{eq:local,T1}, \eqref{eq:local,T2}, Lemma~\ref{lemma:local,T3}, Lemma~\ref{lemma:local,T4}, Lemma~\ref{lemma:local,T5}, \eqref{eq:local,T6}, Lemma~\ref{lemma:local,T7}, we establish the rate of $\E [\ind(Y \le \min\{Y_1,Y_{j_U(1)}\}) ]$:
\[
  \frac{M}{n}\rho_n + \rho_n^2 \lesssim \E \Big[\ind(Y \le \min\{Y_1,Y_{j_U(1)}\}) \Big] - \frac{1}{3} -\frac{M+1}{12n}  \lesssim \frac{M}{n}\sqrt{\log n}\rho_n + \rho_n^2 + \frac{M}{n^2}.
  \yestag\label{eq:local,rate}
\]

Combining \eqref{eq:local,eq1} and \eqref{eq:local,eq2} yields
\begin{align*}
  \E [\xi_{n,M}] =& -2 + \frac{6nM}{(n+1)[nM+M(M+1)/4]} \\
  & \Big\{ 1 + \E \Big[\ind(Y \le Y_1) \ind(j_U(1) = 1) \Big] + (n-2) \E \Big[\ind(Y \le \min\{Y_1,Y_{j_U(1)}\}) \Big] \Big\}\\
  =& -2 + \frac{6n(n-2)M}{(n+1)[nM+M(M+1)/4]} \Big[ \frac{1}{3}
  + \frac{M+1}{12n}\Big] \\
  & + \frac{6n(n-2)M}{(n+1)[nM+M(M+1)/4]} \Big\{ \E \Big[\ind(Y \le \min\{Y_1,Y_{j_U(1)}\}) \Big] - \frac{1}{3} -\frac{M+1}{12n} \Big\} \\
  & + \frac{6nM}{(n+1)[nM+M(M+1)/4]} \Big\{ 1 + \E \Big[\ind(Y \le Y_1) \ind(j_U(1) = 1) \Big] \Big\}\\
  = & -2 + \frac{2(n-2)}{n+1} + 6 \Big\{ \E \Big[\ind(Y \le \min\{Y_1,Y_{j_U(1)}\}) \Big] - \frac{1}{3} -\frac{M+1}{12n} \Big\} (1+o(1))\\
  & + \frac{6nM}{(n+1)[nM+M(M+1)/4]} \Big\{ 1 + \E \Big[\ind(Y \le Y_1) \ind(j_U(1) = 1) \Big] \Big\}\\
  = & 6 \Big\{ \E \Big[\ind(Y \le \min\{Y_1,Y_{j_U(1)}\}) \Big] - \frac{1}{3} -\frac{M+1}{12n} \Big\} (1+o(1)) \\
  & - \frac{3(M+1)/2}{(n+1)[n+(M+1)/4]} + \frac{6nM}{(n+1)[nM+M(M+1)/4]} \E \Big[\ind(Y \le Y_1) \ind(j_U(1) = 1) \Big]\\
  = & 6 \Big\{ \E \Big[\ind(Y \le \min\{Y_1,Y_{j_U(1)}\}) \Big] - \frac{1}{3} -\frac{M+1}{12n} \Big\} (1+o(1)) + O\Big(\frac{M}{n^2}\Big),
\end{align*}
where the last step is due to
\[
  0 \le \E \Big[\ind(Y \le Y_1) \ind(j_U(1) = 1) \Big] \le \P(j_U(1) = 1) = \frac{M+1}{2n}.
\]

We thus proved the first claim by \eqref{eq:local,rate}. Regarding the second claim, notice that if $\rho_n\succ n^{-1}$, we have $\rho_n^2\succ M/n^2$ and accordingly
\[
  \frac{M}{n}\rho_n + \rho_n^2 \lesssim \E [\xi_{n,M}] \lesssim \frac{M}{n}\sqrt{\log n}\rho_n + \rho_n^2,
\]
and thus finish the whole proof.
\end{proof}

\begin{proof}[Proof of Theorem~\ref{thm:alter,mean,var}\ref{thm:alter,var}]
Invoking the law of total variance yields
\[
  \Var [\xi_{n,M}] = \E \Big[\Var\Big[\xi_{n,M} \Biggiven \mX\Big] \Big] + \Var \Big[\E\Big[\xi_{n,M} \Biggiven \mX\Big] \Big].
  \yestag\label{eq:var,decompose}
\]




For the first term in \eqref{eq:var,decompose}, we have
\begin{align*}
  &\E \Big[\Var\Big[\xi_{n,M} \Biggiven \mX\Big] \Big] \\
  = &\frac{36}{(n+1)^2[nM+M(M+1)/4]^2} \Big\{ \sum_{i=1}^{n} \sum_{m=1}^M \E \Big[ \Var \Big[ \min\{R_i,R_{j_m(i)}\} \Biggiven \mX \Big] \Big]\\
  & + \sum_{i=1}^{n} \sum^M_{\substack{m,m'=1\\m \neq m'}} \E \Big[ \Cov \Big[ \min\{R_i,R_{j_m(i)}\}, \min\{R_i,R_{j_{m'}(i)}\} \Biggiven \mX \Big] \Big]\\
  & + \sum^n_{\substack{i,l=1\\i \neq l}} \sum^M_{m,m'=1} \E \Big[ \Cov \Big[ \min\{R_i,R_{j_m(i)}\}, \min\{R_l,R_{j_{m'}(l)}\} \Biggiven \mX \Big] \ind\Big(j_m(i)=l \Big) \Big] \\
  & + \sum^n_{\substack{i,l=1\\i \neq l}} \sum^M_{m,m'=1} \E \Big[ \Cov \Big[ \min\{R_i,R_{j_m(i)}\}, \min\{R_l,R_{j_{m'}(l)}\} \Biggiven \mX \Big] \ind\Big(j_{m'}(l)=i \Big) \Big] \\
  & + \sum^n_{\substack{i,l=1\\i \neq l}} \sum^M_{m,m'=1} \E \Big[ \Cov \Big[ \min\{R_i,R_{j_m(i)}\}, \min\{R_l,R_{j_{m'}(l)}\} \Biggiven \mX \Big] \ind\Big(j_m(i)=j_{m'}(l) \Big) \Big] \\
  & + \sum^n_{\substack{i,l=1\\i \neq l}} \sum^M_{m,m'=1} \E \Big[ \Cov \Big[ \min\{R_i,R_{j_m(i)}\}, \min\{R_l,R_{j_{m'}(l)}\} \Biggiven \mX \Big] \\
  &~~~ \ind\Big(j_m(i) \neq l, j_{m'}(l) \neq i, j_m(i) \neq j_{m'}(l) \Big) \Big] \Big\}.
  \yestag\label{eq:alter,var,decompose}
\end{align*}

We then establish the following lemma.

\begin{lemma}\label{lemma:alter,var,decompose}
  Let $r_n = \frac{M}{n} + \frac{M}{n}\sqrt{\log n}\rho_n + \rho_n^2$. Then for any $i,l \in \zahl{n}$ and $m,m' \in \zahl{M}$,
  \begin{align*}
    & \Big\lvert \frac{1}{n^2} \E \Big[ \Var \Big[ \min\{R_i,R_{j_m(i)}\} \Biggiven \mX \Big] \Big] - \frac{1}{18} \Big\rvert \lesssim r_n,\\
    & \Big\lvert \frac{1}{n^2} \E \Big[ \Cov \Big[ \min\{R_i,R_{j_m(i)}\}, \min\{R_i,R_{j_{m'}(i)}\} \Biggiven \mX \Big] \Big] - \frac{1}{45} \Big\rvert \lesssim r_n,~~m \neq m',\\
    & \Big\lvert \frac{1}{n} \E \Big[ \Cov \Big[ \min\{R_i,R_{j_m(i)}\}, \min\{R_l,R_{j_{m'}(l)}\} \Biggiven \mX \Big] \ind\Big(j_m(i)=l \Big) \Big] - \frac{1}{45} \Big\rvert \lesssim r_n,~~i \neq l,\\
    & \Big\lvert \frac{1}{n} \E \Big[ \Cov \Big[ \min\{R_i,R_{j_m(i)}\}, \min\{R_l,R_{j_{m'}(l)}\} \Biggiven \mX \Big] \ind\Big(j_{m'}(l)=i \Big) \Big] - \frac{1}{45} \Big\rvert \lesssim r_n,~~i \neq l,\\
    & \Big\lvert \frac{1}{n} \E \Big[ \Cov \Big[ \min\{R_i,R_{j_m(i)}\}, \min\{R_l,R_{j_{m'}(l)}\} \Biggiven \mX \Big] \ind\Big(j_m(i)=j_{m'}(l) \Big) \Big] - \frac{1}{45} \Big\rvert \lesssim r_n,~~i \neq l,\\
    & \Big\lvert \frac{1}{n} \E \Big[ \Cov \Big[ \min\{R_i,R_{j_m(i)}\}, \min\{R_l,R_{j_{m'}(l)}\} \Biggiven \mX \Big] \\
    &~~~ \ind\Big(j_m(i) \neq l, j_{m'}(l) \neq i, j_m(i) \neq j_{m'}(l) \Big) \Big] + \frac{4}{45} \Big\rvert \lesssim r_n,~~i \neq l.
  \end{align*}
\end{lemma}

Notice that in \eqref{eq:alter,var,decompose}, the number in the first sum is $nM$, the second is $nM(M-1)$ and the remaining four are $n(n-1)M^2$. Combining \eqref{eq:alter,var,decompose} and Lemma~\ref{lemma:alter,var,decompose}, we obtain
\begin{align*}
  & \Big\lvert \E \Big[\Var\Big[\xi_{n,M} \Biggiven \mX\Big] \Big] \Big\rvert\\
  \lesssim & \frac{1}{n^4M^2} \Big[ \Big( n^3M + n^3M(M-1) + 4n^2(n-1)M^2\Big) r_n \\
  & + \Big\lvert\frac{1}{18}n^3M + \frac{1}{45}n^3M(M-1) + \frac{3}{45}n^2(n-1)M^2 - \frac{4}{45}n^2(n-1)M^2\Big\rvert\Big]\\
  \lesssim & \frac{1}{n^4M^2} \Big[ n^3M^2 r_n + n^3 M\Big]\\
  = & \frac{1}{nM} + \frac{M}{n^2} + \frac{M}{n^2}\sqrt{\log n}\rho_n + \frac{1}{n}\rho_n^2.
  \yestag\label{eq:alter,var,decompose1}
\end{align*}

For the second term in \eqref{eq:var,decompose}, we establish the following lemma. 

\begin{lemma}\label{lemma:alter,var,decompose2}
 Recalling
  $
    r_n = \frac{M}{n} + \frac{M}{n}\sqrt{\log n}\rho_n + \rho_n^2,
  $
  we have
  \[
    \Var \Big[\E\Big[\xi_{n,M} \Biggiven \mX\Big] \Big] \lesssim \frac{1}{n} r_n.
  \]
\end{lemma}

Combining \eqref{eq:alter,var,decompose1} and Lemma~\ref{lemma:alter,var,decompose2}, we obtain
\[
  \Var_{H_{1,n}} [\xi_{n,M}] \lesssim \frac{1}{nM} + \frac{M}{n^2} + \frac{M}{n^2}\sqrt{\log n}\rho_n + \frac{1}{n}\rho_n^2,
\]
and thus complete the proof.
\end{proof}

\appendix

\section{Proof of the rest results}\label{sec:proof}

\paragraph*{Additional notation.}
We use $\stackrel{\sf d}{\longrightarrow}$, $\stackrel{\sf p}{\longrightarrow}$, and $\stackrel{\sf a.s.}{\longrightarrow}$ to denote convergence in distribution, convergence in probability, and almost sure convergence, respectively. For a sequence of random variables $[X_n]_n$ and a real sequences $[a_n]_n$, we write $X_n=O_{\P}(a_n)$ if for any $\epsilon>0$ there exists $C>0$ such that $\P(|X_n|\ge C|a_n|)<\epsilon$ for all $n$ large enough, and $X_n=o_{\P}(a_n)$ if for any $c>0$, $\lim_{n\to\infty}\P(|X_n|\ge c|a_n|)=0$.

\subsection{Proofs of results in Section~\ref{sec:prelim}}

\subsubsection{Proof of Remark~\ref{remark:bias}}

\begin{proof}[Proof of Remark~\ref{remark:bias}]

From the definition of Chatterjee's correlation coefficient,
\begin{align*}
  \xi_n &= 1 - \frac{3 \sum_{i=1}^{n} \lvert R_{j_1(i)} - R_i \rvert }{n^2-1}\\
  &=  1 - \frac{3}{n^2-1} \Big[ \sum_{i=1}^n \Big(R_i + R_{j_1(i)} - 2\min\{R_i,R_{j_1(i)}\} \Big) \Big]\\
  &= \frac{6}{n^2-1} \sum_{i=1}^n \min\{R_i,R_{j_1(i)}\} - \frac{3}{n^2-1} \sum_{i=1}^n \Big( R_i + R_{j_1(i)}\Big) + 1\\
  &= \frac{n+1/2}{n-1} \xi_{n,1} - \frac{3}{n^2-1} \sum_{i=1}^n \Big( R_i + R_{j_1(i)}\Big) + \frac{3n}{n-1}.
\end{align*}
Noticing $[R_i]_{i=1}^n$ is a permutation of $\zahl{n}$,
\[
  \Big\lvert \sum_{i=1}^n \Big( R_i + R_{j_1(i)}\Big) - n(n+1) \Big\rvert \le n-1,
\]
since for any $i,j \in \zahl{n}$, the difference of $R_i$ and $R_j$ can be at most $n-1$. Then
\begin{align*}
  \Big\lvert \frac{n+1/2}{n-1} \xi_{n,1} - \xi_n \Big\rvert &= \Big\lvert \frac{3}{n^2-1} \sum_{i=1}^n \Big( R_i + R_{j_1(i)}\Big) - \frac{3n}{n-1} \Big\rvert\\
  &\le \frac{3}{n^2-1} (n-1) + \Big\lvert \frac{3n(n+1)}{n^2-1}  - \frac{3n}{n-1} \Big\rvert\\
  &= \frac{3}{n+1}.
\end{align*}
Since $\xi_{n,1}$ is bounded, the proof is complete.
\end{proof}

\subsubsection{Proof of Theorem~\ref{thm:asconverge}}

\begin{proof}[Proof of Theorem~\ref{thm:asconverge}]
Let $G(y):= \E [\ind(Y\geq y)]$ and $G_X(y): = \E [\ind(Y\geq y) \given X ]$. We shorthand $F_Y$ by $F$. Let $Q := \int \Var (G_X(y)) \d F(y)$ and $S := \int G(y) (1-G(y)) \d F(y)$. Then from \eqref{eq:xi}, $\xi = Q/S$. Denote
\[
  Q_{n,M} := \frac{1}{nM} \sum_{i=1}^{n} \sum_{m=1}^M \min\{F_n(Y_i), F_n(Y_{j_m(i)})\} - \frac{1}{n} \sum_{i=1}^n G_n(Y_i)^2,
\]
where $F_n(y) := \frac{1}{n} \sum_{i=1}^n \ind(Y_i \le y)$ and $G_n(y) := \frac{1}{n} \sum_{i=1}^n \ind(Y_i \ge y)$ are the empirical counterparts of $F$ and $G$, respectively. 

{\bf Step I.} This step establishes the almost surely convergence of $Q_{n,M}$ to $Q$.

{\bf Step I-1.} Define the population counterpart of $Q_{n,M}$ by
\[
  Q'_{n,M} = \frac{1}{nM} \sum_{i=1}^{n} \sum_{m=1}^M \min\{F(Y_i), F(Y_{j_m(i)})\} - \frac{1}{n} \sum_{i=1}^n G(Y_i)^2,
\]
Since $G$ and $G_n$ are bounded by 1, we have
\[
  \lvert Q_{n,M} - Q'_{n,M} \rvert \le 3\Delta_n,
\]
where $\Delta_n := \sup_{y \in \R} \lvert F_n(y) - F(y) \rvert + \sup_{y \in \R} \lvert G_n(y) - G(y) \rvert$.

From Glivenko-Cantelli Theorem, $\Delta_n \to 0$ almost surely. Since $\lvert \Delta_n \rvert \le 2$, then $\lim_{n\to\infty} \E[\Delta_n] = 0$ and hence
\[
  \lim_{n\to\infty} \E\lvert Q_{n,M} - Q'_{n,M} \rvert = 0.
  \yestag\label{eq:as,diff,q}
\]

{\bf Step I-2.} For any $i \in \zahl{n}$ and $m \in \zahl{M}$,
\[
  \min\Big\{F(Y_i), F(Y_{j_m(i)})\Big\} = \int \ind(Y_i \ge y) \ind(Y_{j_m(i)} \ge y) \d F(y).
\]
Conditional on $\mX$, $[Y_i]_{i=1}^n$ are independent. Then for any $y \in \R$,
\[
  \E [\ind(Y_i \ge y) \ind(Y_{j_m(i)} \ge y)] = G_{X_i}(y) G_{X_{j_m(i)}}(y).
\]
From Fubini's theorem and $\big[(X_i,Y_i)\big]_{i=1}^n$ are i.i.d,
\begin{align*}
  \E [Q'_{n,M}] =& \int \Big[ \frac{1}{M} \sum_{m=1}^M \E[G_{X_1}(y) G_{X_{j_m(1)}}(y)] - G(y)^2 \Big] \d F(y)\\
  =& \int \Big[ \E[G_{X_1}(y)^2] - G(y)^2 \Big] \d F(y) \\
  &+ \int \Big[ \frac{1}{M} \sum_{m=1}^M \E[G_{X_1}(y) (G_{X_{j_m(1)}}(y) - G_{X_1}(y))] \Big] \d F(y).
  \yestag\label{eq:as,exp,q}
\end{align*}

The first term is exactly $Q$ from the definition. 

{\bf Step I-3.} This substep proves that the second term converges to zero.

Consider $U$ to be the uniform distribution over $\zahl{M}$, i.e., for any $m \in \zahl{M}$,
\[
  \P(U = m) = 1/M.
  \yestag\label{eq:as,U}
\]
Denote the probability measure of $X$ by $\mu$. 
For any $x \in \R$ and $\epsilon>0$, since $X_{j_U(1)} = X_1$ when $j_U(1) = 1$, we have
\begin{align*}
  & \limsup_{n\to\infty} \P\Big( \Big\lvert X_{j_U(1)} - X_1 \Big\rvert \ge \epsilon \Biggiven X_1 = x\Big) \\
  = & \limsup_{n\to\infty} \P\Big( \Big\lvert X_{j_U(1)} - X_1 \Big\rvert \ge \epsilon , j_U(1) \neq 1 \Biggiven X_1 = x\Big) \\
  = & \limsup_{n\to\infty} \P\Big( \Big\lvert X_{j_U(1)} - x \Big\rvert \ge \epsilon , j_U(1) \neq 1 \Biggiven X_1 = x\Big) \\
  = & \limsup_{n\to\infty} \P\Big({\rm Bin}(n-1, \mu([x,x+\epsilon])) \le U - 1, j_U(1) \neq 1 \Big)\\
  \le & \limsup_{n\to\infty} \P\Big({\rm Bin}(n-1, \mu([x,x+\epsilon])) \le U - 1\Big)\\
  \le & \limsup_{n\to\infty} \P\Big({\rm Bin}(n-1, \mu([x,x+\epsilon])) \le M - 1\Big) = 0,
\end{align*}
since $U\leq M$ and $M/n \to 0$.

Then from dominated convergence theorem, $\lim_{n\to\infty} \P( \lvert X_{j_U(1)} - X_1 \rvert \ge \epsilon ) = 0$ and then
\[
  X_{j_U(1)} \stackrel{\sf p}{\longrightarrow} X_1.
\]

For any $m \in \zahl{M}$ and $j \in \zahl{n}$, there can be at most one $i$ such that $X_{j_m(i)} = X_j$ and $X_i \neq X_j$ since there is no tie with probability one. Then from the proof of Lemma 9.4 in \cite{chatterjee2020new}, for any nonnegative measurable function $f$,
\[
  \E [f(X_{j_m(1)})] \le 2 \E [f(X_1)],
\]
and thus
\[
  \E [f(X_{j_U(1)})] = \frac{1}{M} \sum_{m=1}^M \E [f(X_{j_m(1)})] \le 2 \E [f(X_1)].
  \yestag\label{eq:as,exp,ineq}
\]

For any measurable function $f$ and probability measure $\nu$, from Lemma 9.5 in \cite{chatterjee2020new}, essentially Lusin's theorem, for any $\epsilon>0$, there exists a compactly supported continuous function $g$ such that $\nu(\{x:f(x) \neq g(x)\}) < \epsilon$. We take $\nu$ to be the probability measure of $X_1$. Then for any $\delta>0$,
\begin{align*}
  & \P( \lvert f(X_{j_U(1)} - f(X_1) \rvert \ge \delta ) \\
  \le & \P( \lvert g(X_{j_U(1)} ) - g(X_1) \rvert \ge \delta ) + \P(f(X_1) \neq g(X_1)) + \P(f(X_{j_U(1)}) \neq g(X_{j_U(1)})).
\end{align*}
From continuous mapping theorem and $X_{j_U(1)} \stackrel{\sf p}{\longrightarrow} X_1$,
\[
  \lim_{n\to\infty} \P( \lvert g(X_{j_U(1)}) - g(X_1) \rvert \ge \delta ) = 0.
\]
From the definition of $g$,
\[
  \P(f(X_1) \neq g(X_1)) < \epsilon.
\]
From \eqref{eq:as,exp,ineq},
\[
  \P(f(X_{j_U(1)}) \neq g(X_{j_U(1)})) \le 2 \P(f(X_1) \neq g(X_1)) < 2\epsilon.
\]
Combining the above derivations together yields
\[
  f\big(X_{j_U(1)} \big)\stackrel{\sf p}{\longrightarrow} f\big(X_1\big).
\]

For any $y \in \R$, $G_x(y)$ is a measurable function with respect to $x$. We take $f(x)$ to be $G_x(y)$. Since $G_x(y)$ is bounded by 1, then
\[
  \lim_{n\to\infty} \E \Big[\lvert G_{X_{j_U(1)}}(y) - G_{X_1}(y) \rvert \Big] = 0
\]
and thus
\begin{align*}
  & \limsup_{n\to\infty} \Big\lvert \frac{1}{M} \sum_{m=1}^M \E\Big[G_{X_1}(y) (G_{X_{j_m(1)}}(y) - G_{X_1}(y))\Big] \Big\rvert \\
  \le & \limsup_{n\to\infty} \frac{1}{M} \sum_{m=1}^M \E\Big\lvert G_{X_{j_m(1)}}(y) - G_{X_1}(y) \Big\rvert\\
  = & \limsup_{n\to\infty} \E \Big\lvert G_{X_{j_U(1)}}(y) - G_{X_1}(y) \Big\vert  =0.
\end{align*}
Invoking dominated convergence theorem then implies
\[
  \lim_{n\to\infty} \int \Big[ \frac{1}{M} \sum_{m=1}^M \E[G_{X_1}(y) (G_{X_{j_m(1)}}(y) - G_{X_1}(y))] \Big] \d F(y) = 0.
\]

{\bf Step I-4.} Combining \eqref{eq:as,diff,q} and \eqref{eq:as,exp,q} shows
\[
  \lim_{n\to\infty} \E[Q_{n,M}] = Q.
\]

From Lemma 9.11 in \cite{chatterjee2020new}, essentially bounded difference inequality, there exists a constant $C>0$ such that for any $n$ and $t>0$,
\[
  \P(\lvert Q_{n,M} - \E[Q_{n,M}] \rvert \ge t) \le 2\exp(-Cnt^2).
\]
Then using Borel–Cantelli Lemma, $Q_{n,M} - \E[Q_{n,M}] \stackrel{\sf a.s.}{\longrightarrow} 0$ and then
\[
  Q_{n,M} \stackrel{\sf a.s.}{\longrightarrow} Q.
\]

{\bf Step II.} Let
\[
  S_n = \frac{1}{n} \sum_{i=1}^n G_n(Y_i)(1-G_n(Y_i)).
\]
From the proof of Theorem 1.1 in \cite{chatterjee2020new}, $S_n \stackrel{\sf a.s.}{\longrightarrow} S$ and then
\[
  Q_{n,M} / S_n \stackrel{\sf a.s.}{\longrightarrow} \xi.
  \yestag\label{eq:as,converge}
\]
From \eqref{eq:Ri}, for any $i \in \zahl{n}$ and $m \in \zahl{M}$,
\[
  \min\{F_n(Y_i), F_n(Y_{j_m(i)})\} = \frac{1}{n} \min\{R_i, R_{j_m(i)}\}.
\]
Some simple calculation gives
\begin{align*}
  Q_{n,M} & = \frac{1}{n^2M} \sum_{i=1}^{n} \sum_{m=1}^M \min\{R_i, R_{j_m(i)}\} - \frac{(n+1)(2n+1)}{6n^2}~~~{\rm and}~~~ S_n = \frac{1}{6}\Big(1-\frac{1}{n^2}\Big),
\end{align*}
and accordingly
\begin{align*}
  \frac{Q_{n,M}}{S_n} & = \frac{6}{M(n+1)(n-1)} \sum_{i=1}^{n} \sum_{m=1}^M \min\{R_i, R_{j_m(i)}\} - \frac{2n+1}{n-1}\\
  & = \frac{n+(M+1)/4}{n-1} \xi_{n,M} + \frac{M-1}{2(n-1)}.
\end{align*}
Then due to \eqref{eq:as,converge} and $M/n \to 0$ as $n \to \infty$, we obtain
\[
  \xi_{n,M} \stackrel{\sf a.s.}{\longrightarrow} \xi.
\]
This completes the proof.
\end{proof}

\subsection{Proofs of results in Section \ref{sec:test}}

\subsubsection{Proof of Theorem~\ref{thm:CLT}}

\begin{proof}[Proof of Theorem~\ref{thm:CLT}]

The proof is divided into three steps.

{\bf Step I.} This step establishes the H\'ajek representation of $\xi_{n,M}$.

Let $F_Y^{(n)}$ be the empirical cumulative distribution function of the sample $[Y_i]_{i=1}^n$. Since
\[
  \sum_{i \neq j} \min\big\{R_i, R_j \big\} = \frac{(n-1)n(n+1)}{3},
\]
we have by \eqref{eq:xin}
\begin{align*}
  \xi_{n,M} =& -2 + \frac{6 \sum_{i=1}^{n} \sum_{m=1}^M \min\big\{R_i, R_{j_m(i)}\big\} }{(n+1)[nM+M(M+1)/4]}\\
  =& \frac{6}{(n+1)[nM+M(M+1)/4]} \sum_{i=1}^{n} \sum_{m=1}^M \min\big\{R_i, R_{j_m(i)}\big\} - \frac{6}{(n-1)n(n+1)} \sum_{i \neq j} \min\big\{R_i, R_j \big\}\\
  =& \frac{6n}{(n+1)[nM+M(M+1)/4]} \sum_{i=1}^{n} \sum_{m=1}^M \min\big\{F_Y^{(n)}(Y_i), F_Y^{(n)}(Y_{j_m(i)})\big\} \\
  & - \frac{6}{(n-1)(n+1)} \sum_{i \neq j} \min\big\{F_Y^{(n)}(Y_i), F_Y^{(n)}(Y_j) \big\},
\end{align*}
Introduce
\begin{align*}
  \widehat{\xi}_{n,M} =& \frac{6n}{(n+1)[nM+M(M+1)/4]} \sum_{i=1}^{n} \sum_{m=1}^M \min\big\{F_Y(Y_i), F_Y(Y_{j_m(i)})\big\} \\
  & - \frac{6}{(n-1)(n+1)} \sum_{i \neq j} \min\big\{F_Y(Y_i), F_Y(Y_j) \big\}\\
{\rm and}~~~  \sigma_{n,M}^2 =& \frac{2}{5}\Big( \frac{1}{nM} \Big) + \frac{8}{15}\Big( \frac{M}{n^2} \Big).
\end{align*}
The goal of this step is to show
\[
  \lim_{n \to \infty} \sigma_{n,M}^{-2} \E\Big(\xi_{n,M} - \widehat{\xi}_{n,M}\Big)^2 = 0.
\]

{\bf Step I-1.} This substep calculates $\E[(\widehat{\xi}_{n,M})^2]$. Notice that under independence of $X$ and $Y$ and $Y$ is continuous, $[F_Y(Y_i)]_{i=1}^n$ are i.i.d from uniform distribution $U(0,1)$ over $[0,1]$. Then
\begin{align*}
  \E[\widehat{\xi}_{n,M}] =& \frac{6n}{(n+1)[nM+M(M+1)/4]} \Big[ \frac{1}{2} \frac{M(M+1)}{2} + \frac{1}{3} \Big(nM - \frac{M(M+1)}{2}\Big) \Big] \\
  & - \frac{6}{(n-1)(n+1)} \Big(\frac{1}{3} n(n-1)\Big) = 0,
\end{align*}
since $\E[U_1] = 1/2$ and $\E[\min\{U_1,U_2\}] = 1/3$ for two independent copies $U_1$ and $U_2$ from $U(0,1)$.

For four independent copies $U_1,U_2,U_3,U_4$ from $U(0,1)$, simple calculations show
\begin{align*}
  & \Var [U_1] = \frac{1}{12},~ \Cov[U_1,U_2] = 0,~ \Cov[U_1,\min\{U_2,U_3\}] = 0,~ \Cov[U_1,\min\{U_1,U_2\}] = \frac{1}{24},\\
  & \Cov[\min\{U_1,U_2\},\min\{U_3,U_4\}] = 0, \Cov[\min\{U_1,U_2\},\min\{U_1,U_3\}] = \frac{1}{45}, \Var[\min\{U_1,U_2\}] = \frac{1}{18}.
\end{align*}
Then analogous to the proof of Theorem~\ref{thm:null,mean,var},
\begin{align*}
  & \Var \Big[\sum_{i=1}^{n} \sum_{m=1}^M \min\big\{F_Y(Y_i), F_Y(Y_{j_m(i)})\big\} \Big]\\
  = & nM \Big(\frac{1}{18}\Big) + \frac{1}{3}M^3 \Big(\frac{1}{12}\Big) + \frac{1}{3}M^3 \Big(\frac{1}{24}\Big) + \Big( nM^2 - nM - \frac{2}{3}M^3 \Big) \Big(\frac{1}{45}\Big)\\
  & + \Big( nM^2 - nM - \frac{2}{3}M^3 \Big) \Big(\frac{1}{45}\Big) + \Big( 2nM^2 - 2M^3 \Big) \Big(\frac{1}{45}\Big) + M^3 \Big(\frac{1}{24}\Big) + o(nM) + o(M^3)\\
  = & \frac{4}{45} nM^2 + \Big[\frac{1}{90}nM + \frac{1}{108}M^3 \Big] (1+o(1)).
\end{align*}
We also have
\begin{align*}
  & \Var \Big[\sum_{i \neq j} \min\big\{F_Y(Y_i), F_Y(Y_j)\big\} \Big]\\
  = & n(n-1) \Var[\min\{U_1,U_2\}] + n(n-1)(4n-7) \Cov[\min\{U_1,U_2\},\min\{U_1,U_3\}] \\
  & + n(n-1)(n-2)(n-3) \Cov[\min\{U_1,U_2\},\min\{U_3,U_4\}] \\
  = & \frac{4}{45}n^3(1+O(n^{-1}))
\end{align*}
and 
\begin{align*}
  & \Cov\Big[\sum_{i=1}^{n} \sum_{m=1}^M \min\big\{F_Y(Y_i), F_Y(Y_{j_m(i)})\big\}, \sum_{i \neq j} \min\big\{F_Y(Y_i), F_Y(Y_j)\big\} \Big]\\
  = & \frac{M(M+1)}{2} \Cov\Big[U_1, \sum_{i \neq j} \min\{U_i,U_j\} \Big] + \Big[ nM - \frac{M(M+1)}{2} \Big] \Cov\Big[\min\{U_1,U_2\}, \sum_{i \neq j} \min\{U_i,U_j\} \Big]\\
  = & \frac{M(M+1)}{2} \Big[ 2(n-1) \Cov[U_1,\min\{U_1,U_2\}] + (n-1)(n-2) \Cov[U_1,\min\{U_2,U_3\}] \Big]\\
  & + \Big[ nM - \frac{M(M+1)}{2} \Big] \Big[ 2 \Var[\min\{U_1,U_2\}] + 4(n-2) \Cov[\min\{U_1,U_2\},\min\{U_1,U_3\}] \\
  & + (n-2)(n-3) \Cov[\min\{U_1,U_2\},\min\{U_3,U_4\}] \Big]\\
  = & \frac{4}{45}n^2M - \frac{1}{360}nM^2 (1+o(1)).
\end{align*}
Then 
\begin{align*}
  & \E[(\widehat{\xi}_{n,M})^2] = \Var[\widehat{\xi}_{n,M}] \\
  = & \frac{36n^2}{(n+1)^2[nM+M(M+1)/4]^2} \Var \Big[\sum_{i=1}^{n} \sum_{m=1}^M \min\big\{F_Y(Y_i), F_Y(Y_{j_m(i)})\big\} \Big] \\
  & + \frac{36}{(n-1)^2(n+1)^2} \Var \Big[\sum_{i \neq j} \min\big\{F_Y(Y_i), F_Y(Y_j)\big\} \Big] - \frac{72n}{(n-1)(n+1)^2[nM+M(M+1)/4]} \\
  & \Cov\Big[\sum_{i=1}^{n} \sum_{m=1}^M \min\big\{F_Y(Y_i), F_Y(Y_{j_m(i)})\big\}, \sum_{i \neq j} \min\big\{F_Y(Y_i), F_Y(Y_j)\big\} \Big]\\
  = & \sigma_{n,M}^2 (1+o(1)).
  \yestag\label{eq:hajek1}
\end{align*}

{\bf Step I-2.} This substep calculates $\E[\xi_{n,M}\widehat{\xi}_{n,M}]$. The following lemma provides some necessary information to this end. 

\begin{lemma}\label{lemma:perm1}
Let $[U_i]_{i=1}^n$ be a sequence of i.i.d random variables from $U(0,1)$. Let $[R_i^U]_{i=1}^n$ be the corresponding ranks. Then
  \begin{align*}
    &\Cov [R_1^U,U_1] = \frac{n-1}{12},~ \Cov[R_1^U,U_2] = -\frac{1}{12},~ \Cov[R_1^U,\min\{U_2,U_3\}] = -\frac{1}{12},\\
    & \Cov[R_1^U,\min\{U_1,U_2\}] = \frac{n-2}{24},~ \Cov[\min\{R_1^U,R_2^U\},\min\{U_3,U_4\}] = -\frac{4}{45},\\
    & \Cov[\min\{R_1^U,R_2^U\},\min\{U_1,U_3\}] = \frac{4n-17}{180},~ \Cov[\min\{R_1^U,R_2^U\},\min\{U_1,U_2\}] = \frac{n-2}{18}.
  \end{align*}
\end{lemma}

Notice that the ranks $[R_i]_{i=1}^n$ of $[Y_i]_{i=1}^n$ are the ranks of $[F_Y(Y_i)]_{i=1}^n$, and $[F_Y(Y_i)]_{i=1}^n$ are i.i.d from $U(0,1)$. Then by Lemma~\ref{lemma:perm1} and analogous to the proof of Theorem~\ref{thm:null,mean,var},
\begin{align*}
  & \Cov \Big[\sum_{i=1}^{n} \sum_{m=1}^M \min\big\{R_i, R_{j_m(i)}\big\}, \sum_{i=1}^{n} \sum_{m=1}^M \min\big\{F_Y(Y_i), F_Y(Y_{j_m(i)})\big\} \Big]\\
  = & nM \Big(\frac{1}{18}n\Big) + \frac{1}{3}M^3 \Big(\frac{1}{12}n\Big) + \frac{1}{3}M^3 \Big(\frac{1}{24}n\Big) + \Big( nM^2 - nM - \frac{2}{3}M^3 \Big) \Big(\frac{1}{45}n\Big)\\
  & + \Big( nM^2 - nM - \frac{2}{3}M^3 \Big) \Big(\frac{1}{45}n\Big) + \Big( 2nM^2 - 2M^3 \Big) \Big(\frac{1}{45}n\Big) + \Big( n^2 M^2 - nM^3 \Big) \Big(-\frac{4}{45}\Big)\\
  & + M^3 \Big(\frac{1}{24}n\Big) + nM^3 \Big(-\frac{1}{12}\Big) + o(n^2M) + o(nM^3)\\
  = & \Big[ \frac{1}{90}n^2M + \frac{2}{135}nM^3 \Big] (1+o(1)).
\end{align*}



Notice that the values of 
\[
  \Cov \Big[ R_i , \sum_{k \neq \ell} \min\big\{F_Y(Y_k), F_Y(Y_\ell) \big\}\Big] \text{ and } \Cov \Big[ \min\big\{R_i, R_{j}\big\}, \sum_{k \neq \ell} \min\big\{F_Y(Y_k), F_Y(Y_\ell) \big\}\Big]
\]
are homogenous for $i \neq j \in \zahl{n}$, and $\sum_{i=1}^n R_i$, $\sum_{i \neq j} \min\{R_i, R_{j}\}$ are constants. Then for any $i \neq j \in \zahl{n}$,
\begin{align*}
  &\Cov \Big[ R_i , \sum_{k \neq \ell} \min\big\{F_Y(Y_k), F_Y(Y_\ell) \big\}\Big] \\
  =& \frac{1}{n} \Cov \Big[ \sum_{i=1}^n R_i , \sum_{k \neq \ell} \min\big\{F_Y(Y_k), F_Y(Y_\ell) \big\}\Big] = 0,\\
  &\Cov \Big[ \min\big\{R_i, R_{j}\big\}, \sum_{k \neq \ell} \min\big\{F_Y(Y_k), F_Y(Y_\ell) \big\}\Big] \\
  =& \frac{1}{n(n-1)} \Cov \Big[ \sum_{i \neq j} \min\big\{R_i, R_{j}\big\}, \sum_{k \neq \ell} \min\big\{F_Y(Y_k), F_Y(Y_\ell) \big\}\Big] = 0,
\end{align*}
and thus
\begin{align*}
  \Cov \Big[ \sum_{i=1}^{n} \sum_{m=1}^M \min\big\{R_i, R_{j_m(i)}\big\}, \sum_{i \neq j} \min\big\{F_Y(Y_i), F_Y(Y_j) \big\} \Big] = 0.
\end{align*}

Since $\sum_{i \neq j} \min\big\{R_i, R_j \big\}$ is a constant, then
\begin{align*}
  & \E[\xi_{n,M}\widehat{\xi}_{n,M}] = \Cov[\xi_{n,M},\widehat{\xi}_{n,M}] \\
  = & \frac{36n}{(n+1)^2[nM+M(M+1)/4]^2} \Cov \Big[\sum_{i=1}^{n} \sum_{m=1}^M \min\big\{R_i, R_{j_m(i)}\big\}, \sum_{i=1}^{n} \sum_{m=1}^M \min\big\{F_Y(Y_i), F_Y(Y_{j_m(i)})\big\} \Big]\\
  = & \sigma_{n,M}^2 (1+o(1)).
  \yestag\label{eq:hajek2}
\end{align*}


{\bf Step I-3.}
From Theorem~\ref{thm:null,mean,var}, $\E[(\xi_{n,M})^2] = \sigma_{n,M}^2 (1+o(1))$. Combining \eqref{eq:hajek1} and \eqref{eq:hajek2}, we obtain
\[
  \lim_{n \to \infty} \sigma_{n,M}^{-2} \E\Big(\xi_{n,M} - \widehat{\xi}_{n,M}\Big)^2 = \lim_{n \to \infty} \sigma_{n,M}^{-2} \Big( \E[(\xi_{n,M})^2] + \E[(\widehat{\xi}_{n,M})^2] - 2 \E[\xi_{n,M}\widehat{\xi}_{n,M}] \Big) = 0.
  \yestag\label{eq:hajekvar1}
\]

{\bf Step II.} This step transforms $\widehat{\xi}_{n,M}$ to a function of sequence with local dependence. To this end, define
\begin{align*}
  \widecheck{\xi}_{n,M} =& \frac{6n}{(n+1)[nM+M(M+1)/4]} \sum_{i=1}^{n} \sum_{m=1}^M \min\big\{F_Y(Y_i), F_Y(Y_{j_m(i)})\big\} \\
  & - \frac{12}{n+1} \sum_{i=1}^n \Big( -\frac{1}{2} F_Y(Y_i)^2 + F_Y(Y_i) - \frac{1}{3} \Big) - \frac{2n}{n+1}.
  \yestag\label{eq:xicheck}
\end{align*}
We then have
\begin{align*}
  & \widecheck{\xi}_{n,M} - \widehat{\xi}_{n,M} \\
  =& \frac{6n}{n+1} \Big[ \frac{1}{n(n-1)} \sum_{i \neq j} \min\big\{F_Y(Y_i), F_Y(Y_j) \big\} - \frac{1}{3} - \frac{2}{n} \sum_{i=1}^n \Big( -\frac{1}{2} F_Y(Y_i)^2 + F_Y(Y_i) - \frac{1}{3} \Big) \Big].
\end{align*}

Notice that for two independent copies $U_1$ and $U_2$ from $U(0,1)$, $\E[\min\{U_1,U_2\}] = 1/3$ and $\E[\min\{U_1,U_2\} \given U_1] = -U_1^2/2 + U_1$. Then by H\'ajek projection of U-statistics \citep[Theorem 12.3]{MR1652247}, 
\[
  \E [(\widecheck{\xi}_{n,M} - \widehat{\xi}_{n,M})^2] = O(n^{-2}),
\]
so that 
\[
  \widecheck{\xi}_{n,M} - \widehat{\xi}_{n,M} = O_{\P} (n^{-1}).
\]
Since $\sigma_{n,M}^2 \succ n^{-2}$, we then have
\[
  \lim_{n \to \infty} \sigma_{n,M}^{-2} \E\Big(\widecheck{\xi}_{n,M} - \widehat{\xi}_{n,M}\Big)^2 = 0 ~~\text{ so that }~~ \sigma_{n,M}^{-1} (\widecheck{\xi}_{n,M} - \widehat{\xi}_{n,M}) = o_{\P}(1).
  \yestag\label{eq:hajekvar2}
\]

{\bf Step III.} This step establishes the central limit theorem of $\widecheck{\xi}_{n,M}$. We first cite a result on the central limit theorem for a sequence with local dependence. For related definitions of interaction rule and Kantorovich-Wasserstein distance, please refer to \citet[Section 2]{MR2435859}.

\begin{lemma}\label{lemma:CLT,local}\citep[Theorem 2.5]{MR2435859}
  Let $\cW$ be a measurable space and $f:\cW^n \to \R$ be a measurable map admitting a symmetric interaction rule $G$. Let $[W_i]_{i=1}^\infty$ be a sequence of i.i.d $\cW$-valued random variables and $\mW = (W_1, \ldots, W_n)$. Let $T = f(\mW)$ and $\sigma^2 = \Var(T)$. Let $[\tW_i]_{i=1}^n$ be independent copies of $[W_i]_{i=1}^n$. Let $\mW^j = (W_1, \ldots, W_{j-1}, \tW_j, W_{j+1}, \ldots, W_n)$ for each $j \in [n]$ and define $\Delta_j f = f(\mW) - f(\mW^j)$. Let $D = \max_{j \in \zahl{n}} \lvert \Delta_j f \rvert$. Let $G'$ be an arbitrary symmetric extension of $G$ on $\cW^{n+4}$ and take
  \[
    \delta = 1 + {\rm ~the~degree~of~the~vertex~1~in~} G'(W_1,\ldots,W_{n+4}).
  \]
  Then
  \[
    \delta_T \le \frac{Cn^{1/2}}{\sigma^2} \Big(\E[D^8]\Big)^{1/4} \Big(\E[\delta^4]\Big)^{1/4} + \frac{1}{2\sigma^3} \sum_{j=1}^n \E[\lvert \Delta_j f \rvert^3],
  \]
  where $C$ is a universal constant and $\delta_T$ is the Kantorovich-Wasserstein distance between the law of $(T-\E[T])/\sqrt{\Var[T]}$ and the standard Gaussian law. 
\end{lemma}

Let $Z_i = (X_i,Y_i) \in \R^2$ and $\mZ = (Z_1,\ldots,Z_n) \in (\R^2)^n$. To apply Lemma~\ref{lemma:CLT,local}, we take $\cW$ to be $\R^2$, $f$ to be $\widecheck{\xi}_{n,M}(\cdot)$, and $T = \widecheck{\xi}_{n,M}(\mZ)$.

{\bf Step III-1.} This substep identifies an interaction rule. For each $i \in \zahl{n}$ and $m \in \zahl{M}$, let $j_m^{\mZ}(i)$ be defined by samples $\mZ$ and $\cJ_M^{\mZ}(i) = \{j_m^{\mZ}(i),m \in \zahl{M}\}$ be the set of indices containing all right $m$-NN of $i$ with $m \in \zahl{M}$. Let $\cE(G)$ be the edge set of graph $G$. Given any $\mZ \in (\R^2)^n$, let $G(\mZ)$ be the (undirected) graph on $\zahl{n}$ such that $i$ and $j$ are connected, i.e., $(i,j) \in \cE(G(\mZ))$, if and only if $j \in \cJ_M^{\mZ}(i)$ or $i \in \cJ_M^{\mZ}(j)$. It is easy to check $G$ is symmetric.

To prove $G$ is a interaction rule, we consider any $\mZ \in (\R^2)^n$ and $\tilde{\mZ} = (\tZ_1,\ldots,\tZ_n) \in (\R^2)^n$, where $\tZ_i = (\tX_i,\tY_i)$ for each $i \in \zahl{n}$. 
For any $i \in \zahl{n}$, let $\mZ^i$ be the vector obtained by replacing $Z_i$ with $\tZ_i$ in $\mZ$. For any $i,j \in \zahl{n}$ and $i \neq j$, let $\mZ^{ij}$ be the vector obtained by replacing $Z_i$ with $\tZ_i$ and $Z_j$ with $\tZ_j$ in $\mZ$. For any $k,\ell \in \zahl{n}$ and $k \neq \ell$, assume $k$ and $\ell$ are not connected in $G(\mZ),G(\mZ^k),G(\mZ^\ell),G(\mZ^{k\ell})$. It suffices to verify $f(\mZ) - f(\mZ^k) = f(\mZ^\ell) - f(\mZ^{k\ell})$.

Define $A_M^{\mZ}(i) := \{j:j_m^{\mZ}(j) = i{~\rm for~some~}m \in \zahl{M}\}$. We shorthand $\cJ^{\mZ}_M(i),A^{\mZ}_M(i)$ by $\cJ_M(i),A_M(i)$; $\cJ^{\mZ^k}_M(i),A^{\mZ^k}_M(i)$ by $\cJ^k_M(i),A^k_M(i)$; $\cJ^{\mZ^\ell}_M(i),A^{\mZ^\ell}_M(i)$ by $\cJ^\ell_M(i),A^\ell_M(i)$; $\cJ^{\mZ^{k\ell}}_M(i),A^{\mZ^{k\ell}}_M(i)$ by $\cJ^{k\ell}_M(i),A^{k\ell}_M(i)$.

Define
\[
  f_i(\mZ) = \sum_{m=1}^M \min\big\{F_Y(Y_i), F_Y(Y_{j_m(i)})\big\}
\]
for each $i \in \zahl{n}$, which is a function of $Y_i$ and $Y_j,j \in \cJ^{\mZ}_M(i)$. We then establish $f_i(\mZ) - f_i(\mZ^k) = f_i(\mZ^\ell) - f_i(\mZ^{k\ell})$ for each $i \in \zahl{n}$.

If $i=k$, then $\ell \notin \cJ_M(i) \cap \cJ_M^\ell(i)$ since $(k,\ell) \notin \cE(G(\mZ)) \cup \cE(G(\mZ^\ell))$. Then $f_i(\mZ) = f_i(\mZ^\ell)$. We also have $f_i(\mZ^k) = f_i(\mZ^{k\ell})$ since $(k,\ell) \notin \cE(G(\mZ^k)) \cup \cE(G(\mZ^{k\ell}))$. Then $f_i(\mZ) - f_i(\mZ^k) = f_i(\mZ^\ell) - f_i(\mZ^{k\ell})$.

If $i \in A_M(k)$ and $i \neq k,\ell$, we have $k \in \cJ_M(i)$. Then $k \in \cJ^\ell_{M+1}(i)$ since we only replace $Z_\ell$ by $\tZ_\ell$ for $\ell \neq i,k$. Then $\ell \notin \cJ^\ell_{M+1}(i)$ since $(k,\ell) \notin \cE(G(\mZ^\ell))$ and $i \neq k,\ell$. Notice that $k \in \cJ_M(i)$ implies $\ell \notin \cJ_{M+1}(i)$ since $(k,\ell) \notin \cE(G(\mZ))$. Then $f_i(\mZ) = f_i(\mZ^\ell)$ by combining $\ell \notin \cJ_{M+1}(i)$ and $\ell \notin \cJ^\ell_{M+1}(i)$. On the other hand, $\ell \notin \cJ_{M+1}(i)$ implies $\ell \notin \cJ^k_{M}(i)$ and $\ell \notin \cJ^\ell_{M+1}(i)$ implies $\ell \notin \cJ^{k\ell}_{M}(i)$ since we only replace $Z_k$ by $\tZ_k$ for $k \neq i,\ell$. Then $f_i(\mZ^k) = f_i(\mZ^{k\ell})$ by combining $\ell \notin \cJ^k_{M}(i)$ and $\ell \notin \cJ^{k\ell}_{M}(i)$. Then $f_i(\mZ) - f_i(\mZ^k) = f_i(\mZ^\ell) - f_i(\mZ^{k\ell})$.

If $i=\ell$ or $i \in A_M(\ell),i \neq k,\ell$, we can establish $f_i(\mZ) = f_i(\mZ^k),f_i(\mZ^\ell) = f_i(\mZ^{k\ell})$ in the same way and then $f_i(\mZ) - f_i(\mZ^k) = f_i(\mZ^\ell) - f_i(\mZ^{k\ell})$. The symmetry of the problem implies that the result still holds for $i \in A_M^k(k) \cup A_M^k(\ell) \cup A_M^\ell(k) \cup A_M^\ell(\ell) \cup A_M^{k\ell}(k) \cup A_M^{k\ell}(\ell)$. If $i \notin A_M(k) \cup A_M(\ell) \cup A_M^k(k) \cup A_M^k(\ell) \cup A_M^\ell(k) \cup A_M^\ell(\ell) \cup A_M^{k\ell}(k) \cup A_M^{k\ell}(\ell)$ and $i \neq k,\ell$, then $k,\ell \notin \cJ_M(i) \cup \cJ^k_M(i) \cup \cJ^\ell_M(i) \cup \cJ^{k\ell}_M(i)$ and then $f_i(\mZ) = f_i(\mZ^k) = f_i(\mZ^\ell) = f_i(\mZ^{k\ell})$.

Then it is easy to verify $f(\mZ) - f(\mZ^k) = f(\mZ^\ell) - f(\mZ^{k\ell})$ from the definition of $f$ and the fact that $f_i(\mZ) - f_i(\mZ^k) = f_i(\mZ^\ell) - f_i(\mZ^{k\ell})$ for any $i \in \zahl{n}$.

{\bf Step III-2.} This substep applies Lemma~\ref{lemma:CLT,local}.

For the symmetric extension of $G$ on $(\R^2)^{n+4}$, given any $\mZ' \in (\R^2)^{n+4}$, let $G'(\mZ')$ be the graph on $\zahl{n+4}$ such that $i$ and $j$ are connected if and only if $j \in \cJ^{\mZ'}_{M+4}(i)$ or $i \in \cJ^{\mZ'}_{M+4}(j)$. If $j \in \cJ^{\mZ}_{M}(i)$ for some $\mZ \in (\R^2)^n$, then $j \in \cJ^{\mZ'}_{M+4}(i)$, where $\mZ'$ is obtained by adding four elements to $\mZ$. This is equivalent to $(i,j) \in \cE(G'(\mZ'))$ if $\mZ$ is embedded in $\mZ'$. Then $(i,j) \in \cE(G(\mZ))$ implies $(i,j) \in \cE(G'(\mZ'))$. It is easy to check $G'$ is symmetric. And the degree of vertex 1 can be at most $2(M+4)$ for $G'$ since the cardinalities of $\cJ^{\mZ'}_{M+4}(1)$ and $A^{\mZ'}_{M+4}(1)$ are both bounded by $M+4$ for any $\mZ' \in (\R^2)^{n+4}$.

For any $i \in \zahl{n}$, we consider the value $f_i(\mZ) - f_i(\mZ^k)$.

If $i \in A_M(k) \setminus A_M^k(k)$ and $i \neq k$, then 
\[
  f_i(\mZ) - f_i(\mZ^k) = \min\big\{F_Y(Y_i), F_Y(Y_k)\big\} - \min\big\{F_Y(Y_i), F_Y(Y_{j^k_M(i)})\big\}.
\]

If $i \in A_M^k(k) \setminus A_M(k)$ and $i \neq k$, then 
\[
  f_i(\mZ) - f_i(\mZ^k) = \min\big\{F_Y(Y_i), F_Y(Y_{j_M(i)})\big\} - \min\big\{F_Y(Y_i), F_Y(\tY_k)\big\}.
\]

If $i \in A_M^k(k) \cap A_M(k)$ and $i \neq k$, then 
\[
  f_i(\mZ) - f_i(\mZ^k) = \min\big\{F_Y(Y_i), F_Y(Y_k)\big\} - \min\big\{F_Y(Y_i), F_Y(\tY_k)\big\}.
\]

If $i \notin A_M^k(k) \cup A_M(k)$ and $i \neq k$, then 
\[
  f_i(\mZ) - f_i(\mZ^k) = 0.
\]

Since $0 \le F_Y \le 1$, then
\begin{align*}
  & \Big\lvert \sum_{i=1}^n f_i(\mZ) - \sum_{i=1}^n f_i(\mZ^k) \Big\rvert \\
  \le & \vert f_k(\mZ) - f_k(\mZ^k) \rvert + \lvert A_M(k) \setminus A_M^k(k) \rvert + \lvert A_M^k(k) \setminus A_M(k) \rvert + \lvert A_M^k(k) \cap A_M(k) \rvert\\
  \le & 3M,
\end{align*}
where $\lvert A \rvert$ is the cardinality of $A$ when $A$ is a set, and the last step is from $\lvert A_M^k(k) \cup A_M(k) \rvert \le 2M$.

From \eqref{eq:xicheck},
\begin{align*}
   \lvert \Delta_k f \rvert =  \lvert f(\mZ) - f(\mZ^k) \rvert  \le&  \frac{6n}{(n+1)[nM+M(M+1)/4]} \Big\lvert \sum_{i=1}^n f_i(\mZ) - \sum_{i=1}^n f_i(\mZ^k) \Big\rvert \\
  & + \frac{12}{n+1} \Big\lvert \frac{1}{2} \Big(F_Y(Y_k)^2 - F_Y(\tY_k)^2\Big) - \Big(F_Y(Y_k) -F_Y(\tY_k)\Big) \Big\rvert  \lesssim  \frac{1}{n}.
\end{align*}

From Theorem~\ref{thm:null,mean,var}, \eqref{eq:hajekvar1}, and \eqref{eq:hajekvar2},
\[
  \Var[\widecheck{\xi}_{n,M}] = \sigma_{n,M}^2 (1+o(1)).
\]

Then from Lemma~\ref{lemma:CLT,local},
\begin{align*}
  \delta_{T} \lesssim & \frac{n^{1/2}}{\sigma_{n,M}^2} n^{-2} M + \frac{1}{\sigma_{n,M}^3} n^{-2}= \frac{n^{-3/2}M}{n^{-1}M^{-1} + n^{-2}M} + \frac{n^{-2}}{n^{-3/2}M^{-3/2} + n^{-3}M^{3/2}}.
\end{align*}
If $M \prec n^{1/4}$, then $\delta_{T} \to 0$ and then 
\[
  (\widecheck{\xi}_{n,M}-\E[\widecheck{\xi}_{n,M}])/\sqrt{\Var[\widecheck{\xi}_{n,M}]} \stackrel{\sf d}{\longrightarrow} N(0,1).
\]
Since $\Var[\widecheck{\xi}_{n,M}] = \sigma_{n,M}^2 (1+o(1))$ and it is easy to check $\E[\widecheck{\xi}_{n,M}] =0$, then
\[
  \sigma_{n,M}^{-1} \widecheck{\xi}_{n,M} \stackrel{\sf d}{\longrightarrow} N(0,1).
\]
Due to \eqref{eq:hajekvar1} and \eqref{eq:hajekvar2},
\[
  \sigma_{n,M}^{-1} \xi_{n,M} \stackrel{\sf d}{\longrightarrow} N(0,1).
\]
Notice that when $M \prec n^{1/4}$, $\sigma^2_{n,M} = (2/5)(nM)^{-1}(1+o(1))$. We then have
\[
  \sqrt{nM} \xi_{n,M} \stackrel{\sf d}{\longrightarrow} N(0,2/5),
\]
and thus complete the proof.
\end{proof}

\subsubsection{Proof of Theorem~\ref{thm:test}}

\begin{proof}[Proof of Theorem~\ref{thm:test}\ref{thm:test1}] 

Notice that
\[
(1+B)^{-1}\Big\{1+\sum_{b=1}^B\ind\Big(\xi_{n,M}^{\pm(b)} \geq  \xi_{n,M}^{\pm}\Big)\Big\} \geq (1+B)^{-1}\Big\{1+\sum_{b=1}^B\ind^{\circ}\Big(\xi_{n,M}^{\pm(b)},  \xi_{n,M}^{\pm}, U_b\Big)\Big\},
\]
where $U_1,\ldots,U_B$ are i.i.d Bernoulli random variables of equal probabilities to be 0 or 1, and
\[
\ind^{\circ}(x,y,u):=
\begin{cases}
\ind(x\geq y), & \text{if }x\ne y,\\
u, & \text{if }x=y.
\end{cases}
\]
Under $H_0$, we have $\xi_{n,M}^{\pm (1)}, \ldots, \xi_{n,M}^{\pm (B)}, \xi_{n,M}^{\pm}$ are i.i.d. and accordingly
\[
(1+B)^{-1}\Big\{1+\sum_{b=1}^B\ind^{\circ}\Big(\xi_{n,M}^{\pm(b)},  \xi_{n,M}^{\pm}, U_b\Big)\Big\}
\]
is discretely uniformly distributed over
\[
\Big\{\frac{1}{1+B}, \frac{2}{1+B},\ldots, \frac{1+B}{1+B} \Big\}.
\]
As a consequence, under $H_0$
\begin{align*}
\P\Big[(1+B)^{-1}\Big\{1+\sum_{b=1}^B\ind\Big(\xi_{n,M}^{\pm(b)} \geq  \xi_{n,M}^{\pm}\Big) \Big\}\leq \alpha\Big]&\leq \P\Big[(1+B)^{-1}\Big\{1+\sum_{b=1}^B\ind^{\circ}\Big(\xi_{n,M}^{\pm(b)},  \xi_{n,M}^{\pm}, U_b\Big)\Big\}\leq \alpha\Big]\\
&=\frac{\lfloor (1+B)\alpha\rfloor}{1+B}\leq \alpha.
\end{align*}
This completes the proof.
\end{proof}

\begin{proof}[Proof of Theorem~\ref{thm:test}\ref{thm:test2}]

Let
\[
  \Xi_n^{(b)}:=\ind\Big(\xi_{n,M}^{\pm(b)} \geq  \xi_{n,M}^{\pm}\Big).
\]

Then
\[
  (1+B)^{-1}\Big\{1+\sum_{b=1}^B\ind \Big(\xi_{n,M}^{\pm(b)} \geq  \xi_{n,M}^{\pm}\Big) \Big\} = (1+B)^{-1} + (1+B)^{-1} \sum_{b=1}^B \Xi_n^{(b)}.
\]

Since $M/n \to 0$, from Theorem~\ref{thm:asconverge}, we have
\[
  \xi_{n,M}^{(b)} \stackrel{\sf a.s.}{\longrightarrow} 0,~~ \xi_{n,M}^{-(b)} \stackrel{\sf a.s.}{\longrightarrow} 0,~~ \xi_{n,M} \stackrel{\sf a.s.}{\longrightarrow} \xi_{H_1},
\]
where $\xi_{H_1} > 0$ is the value of $\xi$ under the fix alternative $H_1$. One then has
\begin{align*}
  \Xi_n^{(b)} &= \ind\Big(\xi_{n,M}^{\pm(b)} \geq  \xi_{n,M}^{\pm}\Big) \\
  &= \ind\Big(\max\big\{\xi_{n,M}^{(b)},\xi_{n,M}^{-(b)}\big\} \geq  \max\big\{\xi_{n,M},\xi_{n,M}^{-}\big\} \Big)\\
  & \le \ind\Big(\max\big\{\xi_{n,M}^{(b)},\xi_{n,M}^{-(b)}\big\} \geq  \xi_{n,M} \Big)\\
  & \le \ind\Big(\xi_{n,M}^{(b)} \geq \xi_{n,M} \Big) + \ind\Big(\xi_{n,M}^{-(b)} \geq \xi_{n,M} \Big) \stackrel{\sf a.s.}{\longrightarrow} 0.
\end{align*}

Notice that $B = B_n \to \infty$ and $[\Xi_n^{(b)}]_{b=1}^{B_n}$ are exchangeable. Applying Lemma 1.1 in \cite{MR790492}, we obtain
\[
  B_n^{-1} \sum_{b=1}^{B_n} \Xi_n^{(b)} = \E \Big( \Xi_n^{(1)} \Biggiven \cG_n\Big),
\]
where $\cG_n$ is the $\sigma$-field generated by
\[
  \cG_n = \sigma \Big( \sum_{b=1}^{B_n} \Xi_n^{(b)}, \sum_{b=1}^{B_{n+1}} \Xi_{n+1}^{(b)} , \ldots \Big).
\]
Since $[\cG_n]_{n=1}^\infty$ is decreasing and $\cG_n \to \cG_{\infty}$, where $\cG_{\infty} = \bigcap_{n=1}^\infty \cG_n$, $0 \le \Xi_n^{(1)} \le 1, \Xi_n^{(1)} \stackrel{\sf a.s.}{\longrightarrow} 0$, applying Lemma 2(c) in \cite{MR548906} yields
\[
\E\Big(\Xi_n^{(1)}\Biggiven \cG_n\Big)
\stackrel{\sf a.s.}{\longrightarrow}
\E\Big(0\Biggiven \cG_{\infty}\Big)=0,
\]
which implies
\[
  B_n^{-1}\sum_{b=1}^{B_n}\Xi_n^{(b)} \stackrel{\sf a.s.}{\longrightarrow} 0.
\]
We accordingly have
\[
  \lim_{n\to\infty} \P_{H_1}\big(\sT_{\alpha,B}^{\xi_{n,M}^{\pm}}=1\big)=\P\Big[(1+B)^{-1}\Big\{1+\sum_{b=1}^B\ind\Big(\xi_{n,M}^{\pm(b)} \geq  \xi_{n,M}^{\pm}\Big) \Big\}\leq \alpha\Big] = 1,
\]
and complete the proof.
\end{proof}

\subsection{Proofs of results in Section~\ref{sec:alter}}

\subsubsection{Proof of Theorem~\ref{thm:local power}}

\begin{proof}[Proof of Theorem~\ref{thm:local power}\ref{thm:lpa-1}]

By symmetry, without loss of generality we may consider the sequence to be positive. For any integer $B\geq 1$ and $x\in\R$, with a little bit abuse of notation, define
\[
G_{\xi^\pm}(x) := \P_{H_0}(\xi^{\pm}_{n,M}<x)~~~{\rm and}~~~G_B := \frac{1}{B}\sum_{b=1}^B\ind(\xi_{n,M}^{\pm (b)}<x).
\]
Considering any given sample $\big[(X_i,Y_i)\big]_{i=1}^{n}$,
\begin{align*}
  & \P \big(\sT_{\alpha,B}^{\xi_{n,M}^{\pm}}=1 \biggiven \big[(X_i,Y_i)\big]_{i=1}^{n}\big)\\
  =& \P\Big[(1+B)^{-1}\Big\{1+\sum_{b=1}^B\ind\Big(\xi_{n,M}^{\pm(b)} \geq  \xi_{n,M}^{\pm}\Big) \Big\}\leq \alpha \Biggiven \big[(X_i,Y_i)\big]_{i=1}^{n} \Big]\\
  =& \P\Big[(1+B)^{-1}\Big\{1+\sum_{b=1}^B\ind\Big(\xi_{n,M}^{\pm(b)} < \xi_{n,M}^{\pm}\Big) \Big\} \ge 1-\alpha \Biggiven \big[(X_i,Y_i)\big]_{i=1}^{n} \Big]\\
  =& \P\Big[ G_B(\xi_{n,M}^{\pm}) \ge 1-\Big(1+\frac{1}{B}\Big)\alpha \Biggiven \big[(X_i,Y_i)\big]_{i=1}^{n} \Big]\\
  \ge & \P\Big[ G_B(\xi_{n,M}^{\pm}) \ge 1-\alpha \Biggiven \big[(X_i,Y_i)\big]_{i=1}^{n} \Big]  .
\end{align*}

Using Hoeffding's inequality, for any $\epsilon>0$ and $x>0$,
\[
  \P \Big(G_{\xi^{\pm}}(x)-G_B(x) > \epsilon \Big) \le \exp(-2B\epsilon^2).
\]

Then
\[
  \P \big(\sT_{\alpha,B}^{\xi_{n,M}^{\pm}}=1 \biggiven \big[(X_i,Y_i)\big]_{i=1}^{n}\big) \ge \P\Big( G_{\xi^\pm}(\xi_{n,M}^{\pm}) \ge 1-\alpha+\epsilon \Biggiven \big[(X_i,Y_i)\big]_{i=1}^{n} \Big) - \exp(-2B\epsilon^2).
\]

Picking $\epsilon = \alpha/2$ and then taking expectation over both sides yields
\[
  \P_{H_{1,n}}\Big(\sT_{\alpha,B}^{\xi_{n,M}^{\pm}}=1\Big) \ge \P_{H_{1,n}}\Big(G_{\xi^\pm}(\xi_{n,M}^{\pm}) \ge 1-\alpha/2 \Big) - \exp(-B\alpha^2/2).
\]

Introduce $t_\alpha := \sqrt{2\E_{H_0}[(\xi_{n,M}^\pm)^2]/\alpha}$. Then from Markov's inequality, we have
\[
  \P_{H_0} \Big(\xi_{n,M}^\pm \ge t_\alpha\Big) \le \P_{H_0} \Big((\xi_{n,M}^\pm)^2 \ge t_\alpha^2\Big) \le \frac{\E_{H_0}[(\xi_{n,M}^\pm)^2]}{t_\alpha^2} = \frac{\alpha}{2}.
\]

Then
\[
  \P_{H_{1,n}}\Big(\sT_{\alpha,B}^{\xi_{n,M}^{\pm}}=1\Big) \ge \P_{H_{1,n}}\Big(\xi_{n,M}^{\pm} \ge t_\alpha \Big) - \exp(-B\alpha^2/2).
\]

Notice that $\E_{H_0}[\xi_{n,M}] = \E_{H_0}[\xi_{n,M}^-] = 0$ from Theorem~\ref{thm:null,mean,var}, then
\begin{align*}
  \E_{H_0}[(\xi_{n,M}^\pm)^2] &= \E_{H_0} [\max\{\xi_{n,M},\xi_{n,M}^-\}]^2\\
  & \le \E_{H_0} [\lvert \xi_{n,M} \rvert + \lvert \xi_{n,M}^- \rvert]^2\\
  & \le 2 \Big( \E_{H_0} [(\xi_{n,M})^2] + \E_{H_0} [(\xi_{n,M}^-)^2] \Big)\\
  & = 2 \Big( \Var_{H_0} [\xi_{n,M}] + \Var_{H_0} [\xi_{n,M}^-] \Big)\\
  & = 4\Var_{H_0} [\xi_{n,M}].
\end{align*}

From the assumptions and Corollary~\ref{crl:detection},
\[
  \lim_{n \to \infty} \frac{\E_{H_{1,n}} [\xi_{n,M}]}{\sqrt{\Var_{H_{1,n}} [\xi_{n,M}]}} = +\infty.
\]

From Theorem~\ref{thm:null,mean,var}, we also have
\[
  \lim_{n \to \infty} \frac{\E_{H_{1,n}} [\xi_{n,M}]}{\sqrt{\Var_{H_0} [\xi_{n,M}]}} = +\infty.
\]

Then for any $K>0$, there exist $N_K>0$ such that for any $n>N_K$,
\[
  \E_{H_{1,n}} [\xi_{n,M}] \ge K \sqrt{\Var_{H_{1,n}} [\xi_{n,M}]},~ \E_{H_{1,n}} [\xi_{n,M}] \ge K \sqrt{\E_{H_0}[(\xi_{n,M}^\pm)^2]}.
\]

From Markov's inequality, for any $t>0$,
\[
  \P_{H_{1,n}}\Big( \xi_{n,M} \ge \E_{H_{1,n}} [\xi_{n,M}] - t \Big) \ge 1 - \frac{\Var_{H_{1,n}} [\xi_{n,M}]}{t^2}.
\]

Take $t = (1 - \sqrt{2/\alpha}/K)E_{H_{1,n}} [\xi_{n,M}]$, where we assume $K$ is large enough so that $t$ is positive. Then for any $n>N_K$,
\begin{align*}
  \P_{H_{1,n}}\Big(\sT_{\alpha,B}^{\xi_{n,M}^{\pm}}=1\Big) &\ge \P_{H_{1,n}}\Big(\xi_{n,M}^{\pm} > \sqrt{\frac{2}{\alpha}} \sqrt{\E_{H_0}[(\xi_{n,M}^\pm)^2]}\Big) - \exp(-B\alpha^2/2)\\
  &\ge \P_{H_{1,n}}\Big(\xi_{n,M} > \sqrt{\frac{2}{\alpha}} \sqrt{\E_{H_0}[(\xi_{n,M}^\pm)^2]}\Big) - \exp(-B\alpha^2/2)\\
  &\ge \P_{H_{1,n}}\Big(\xi_{n,M} > \sqrt{\frac{2}{\alpha}} \frac{1}{K} \E_{H_{1,n}} [\xi_{n,M}]\Big) - \exp(-B\alpha^2/2)\\
  &\ge 1 - \frac{\Var_{H_{1,n}} [\xi_{n,M}]}{t^2} - \exp(-B\alpha^2/2)\\
  &= 1 - \frac{\Var_{H_{1,n}} [\xi_{n,M}]}{(1 - \sqrt{2/\alpha}/K)^2(E_{H_{1,n}} [\xi_{n,M}])^2}- \exp(-B\alpha^2/2)\\
  &\ge 1 - \frac{1/K^2}{(1 - \sqrt{2/\alpha}/K)^2} - \exp(-B\alpha^2/2) \\
  &= 1 - \frac{1}{(K-\sqrt{2/\alpha})^2} - \exp(-B\alpha^2/2).
\end{align*}

For any $\epsilon>0$, we have $(K-\sqrt{2/\alpha})^{-2} < \epsilon$ for $K$ sufficiently large, then together with $B \to \infty$ as $n \to \infty$,
\[
  \liminf_{n \to \infty} \P_{H_{1,n}}\Big(\sT_{\alpha,B}^{\xi_{n,M}^{\pm}}=1\Big) \ge 1-\epsilon.
\]

Since $\epsilon$ is arbitrary, we then have
\[
  \lim_{n\to\infty} \P_{H_{1,n}}\Big(\sT_{\alpha,B}^{\xi_{n,M}^{\pm}}=1\Big) = 1,
\]
which completes the proof.
\end{proof}

\begin{proof}[Proof of Theorem~\ref{thm:local power}\ref{thm:lpa-2}]
Since $B \to \infty$ as $n \to \infty$, we can assume $B>\alpha^{-1}-1$. Let $\xi_{n,M}^{[1]}, \ldots, \xi_{n,M}^{[B]}$ be a rearrangement of $\xi_{n,M}^{(1)}, \ldots, \xi_{n,M}^{(B)}$ such that 
\[
  \xi_{n,M}^{[1]} \le \xi_{n,M}^{[2]} \le \cdots \le \xi_{n,M}^{[B]}.
\]
From \eqref{eq:test},
\begin{align*}
  \sT_{\alpha,B}^{\xi_{n,M}^{\pm}} & = \ind\Big[ (1+B)^{-1}\Big\{1+\sum_{b=1}^B\ind\Big(\xi_{n,M}^{\pm(b)} \geq  \xi_{n,M}^{\pm}\Big)\Big\} \leq \alpha\Big]\\
  & \le \ind\Big[ (1+B)^{-1}\Big\{1+\sum_{b=1}^B\ind\Big(\xi_{n,M}^{(b)} \geq  \xi_{n,M}^{\pm}\Big)\Big\} \leq \alpha\Big]\\
  & = \ind\Big[ \sqrt{nM} \xi_{n,M}^{\pm} > \sqrt{nM} \xi_{n,M}^{[1+B-\lfloor \alpha(1+B) \rfloor]} \Big],
\end{align*}
where $\lfloor \cdot\rfloor$ is the floor function.

Let $\Phi_{2/5}(\cdot),\Phi_{2/5}^{-1}(\cdot)$ be the cumulative distribution function and quantile function of $N(0,2/5)$. Then for any $y \in \R$, from Theorem~\ref{thm:CLT},
\[
  B^{-1} \sum_{b=1}^B \ind(\sqrt{nM}\xi_{n,M}^{(b)} \le y) \stackrel{\sf p}{\longrightarrow} \Phi_{2/5}(y).
\]
Then from Theorem 3.1 in \citet{MR57521}, we have
\[
  \sqrt{nM} \xi_{n,M}^{[1+B-\lfloor \alpha(1+B) \rfloor]} \stackrel{\sf p}{\longrightarrow} \Phi_{2/5}^{-1}(1-\alpha).
\]
Then
\begin{align*}
  &\limsup_{n\to\infty} \P_{H_{1,n}}\Big(\sT_{\alpha,B}^{\xi_{n,M}^{\pm}}=1\Big) \\
  =& \limsup_{n\to\infty} \P_{H_{1,n}}\Big( \sqrt{nM} \xi_{n,M}^{\pm} > \Phi_{2/5}^{-1}(1-\alpha) \Big)\\
  \le& \limsup_{n\to\infty} \frac{nM \E_{H_{1,n}}[(\xi_{n,M}^{\pm})^2]}{[\Phi_{2/5}^{-1}(1-\alpha)]^2}\\
  \le& \frac{2}{[\Phi_{2/5}^{-1}(1-\alpha)]^2} \limsup_{n\to\infty} nM \Big( \E_{H_{1,n}}[(\xi_{n,M})^2] + \E_{H_{1,n}}[(\xi_{n,M}^{-})^2] \Big)\\
  = & \frac{2}{[\Phi_{2/5}^{-1}(1-\alpha)]^2} \limsup_{n\to\infty} nM \Big( \Var_{H_{1,n}}[\xi_{n,M}] + (\E_{H_{1,n}}[\xi_{n,M}])^2 + \Var_{H_{1,n}}[\xi_{n,M}^{-}] + (\E_{H_{1,n}}[\xi_{n,M}^{-}])^2 \Big).
\end{align*}
If $|\rho_n| \prec \zeta_{n,M}$ and $M \prec n^{1/4}$, from Theorem~\ref{thm:alter,mean,var}, we have
\[
  \limsup_{n\to\infty} nM \Var_{H_{1,n}}[\xi_{n,M}] = O(1)~~{\rm and}~~ \limsup_{n\to\infty} nM (\E_{H_{1,n}}[\xi_{n,M}])^2 = O(1).
\]
It is easy to check that the above bounds also hold for $\xi_{n,M}^{-}$. Then we obtain
\[
  \limsup_{n\to\infty} \P_{H_{1,n}}\Big(\sT_{\alpha,B}^{\xi_{n,M}^{\pm}}=1\Big) \le \frac{C}{[\Phi_{2/5}^{-1}(1-\alpha)]^2},
\]
for a universal constant $C>0$.

Then for any sufficiently small $\alpha$ such that $C[\Phi_{2/5}^{-1}(1-\alpha)]^{-2} < 1$, we take $\beta_\alpha = C[\Phi_{2/5}^{-1}(1-\alpha)]^{-2}$ and complete the proof. 
\end{proof}

\subsection{Proofs of results in Section \ref{sec:main-proof}}

\subsubsection{Proof of Lemma~\ref{lemma:perm}}

\begin{proof}[Proof of Lemma~\ref{lemma:perm}]
The first claim is a well known property of ranks. For the rest, 

(1) $\E [R_1] = \frac{1}{n} \sum_{k=1}^{n} k = \frac{n+1}{2}$.

(2) For $\E [\min\{R_1,R_2\}]$, we consider the set of pairs $\{(a,b):a,b \in \zahl{n},a \neq b\}$. The cardinality of this set is $n(n-1)$ and then we have $\P (R_1 = a, R_2 = b) = 1/[n(n-1)]$ for any $a,b \in \zahl{n},a \neq b$. The number of pairs $(a,b)$ with $a \neq b$ such that $\min\{a,b\} = k$ is $2(n-k)$ for any $k \in \zahl{n}$. Then
\[
  \E [\min\{R_1,R_2\}] = \frac{1}{n(n-1)} \sum_{k=1}^{n-1} 2(n-k)k = \frac{n+1}{3}.
\]

(3) For $\Var[R_1]$, we have
\[
  \Var [R_1] = \E [R_1^2] - [\E [R_1]]^2 = \frac{1}{n} \sum_{k=1}^{n} k^2 - \Big(\frac{n+1}{2} \Big)^2 = \frac{(n+1)(n-1)}{12}.
\]

(4) For $\Cov[R_1,R_2]$, we have
\[
  \Cov[R_1,R_2] = \E [R_1 R_2] - \E [R_1] \E [R_2] = \frac{1}{n(n-1)} \sum_{\substack{k,l=1\\k \neq l}}^n kl - \Big(\frac{n+1}{2} \Big)^2 = -\frac{n+1}{12},
\]
since
\[
  \sum_{\substack{k,l=1\\k \neq l}}^n kl = \frac{(n-1)n(n+1)(3n+2)}{12}.
\]

(5) For $\Cov[R_1,\min\{R_2,R_3\}]$, we consider the set of triples $\{(a,b,c):a,b,c \in \zahl{n},a \neq b \neq c\}$. The cardinality of this set is $n(n-1)(n-2)$ and then we have $\P (R_1 = a, R_2 = b, R_3 = c) = 1/[n(n-1)(n-2)]$ for any $a,b,c \in \zahl{n},a \neq b \neq c$. For any $k,l \in \zahl{n},k<l$, the number of triples $(a,b,c)$ with $a \neq b \neq c$ such that 
\[
a = k, \min\{b,c\} = l 
\]
is $2(n-l)$, and the number of triples such that 
\[
a = l, \min\{b,c\} = k 
\]
is $2(n-k-1)$. Then the total number of triples is $4n-2(k+l)-2$. Then
\[
  \E [R_1 \min\{R_2,R_3\}] = \frac{1}{n(n-1)(n-2)} \sum_{\substack{k,l=1\\k < l}}^n kl [4n-2(k+l)-2] = \frac{(n+1)(2n+1)}{12},
\]
since
\[
  \sum_{\substack{k,l=1\\k \neq l}}^n kl (k+l) = \frac{(n-1)n^2(n+1)^2}{3}.
\]
We accordingly derive
\[
  \Cov[R_1,\min\{R_2,R_3\}] = \E [R_1 \min\{R_2,R_3\}] - \E [R_1] \E [\min\{R_1,R_2\}] = -\frac{n+1}{12}.
\]

(6) For $\Cov[R_1,\min\{R_1,R_2\}]$, for any $k \in \zahl{n}$, the number of pairs $(a,b)$ with $a \neq b$ such that 
\[
a=k,\min\{a,b\}=k 
\]
is $n-k$. For any $k,l \in \zahl{n},k<l$, the number of pairs $(a,b)$ with $a \neq b$ such that 
\[
a=k,\min\{a,b\}=l
\]
is 0, and such that 
\[
a=l,\min\{a,b\}=k 
\]
is 1. We then conclude
\[
  \E [R_1 \min\{R_1,R_2\}] = \frac{1}{n(n-1)} \Big[ \sum_{\substack{k,l=1\\k < l}}^n kl + \sum_{k=1}^n (n-k)k^2 \Big] = \frac{(n+1)(5n+2)}{24},
\]
which implies
\[
  \Cov[R_1,\min\{R_1,R_2\}] = \E [R_1 \min\{R_1,R_2\}] - \E [R_1] \E [\min\{R_1,R_2\}] = \frac{(n-2)(n+1)}{24}.
\]

(7) For $\Cov[\min\{R_1,R_2\},\min\{R_3,R_4\}]$, we consider the set of $\{(a,b,c,d):a,b,c,d \in \zahl{n},a \neq b \neq c \neq d\}$. The cardinality of this set is $n(n-1)(n-2)(n-3)$ and then we have $\P (R_1 = a, R_2 = b, R_3 = c, R_4=d) = 1/[n(n-1)(n-2)(n-3)]$ for any $a,b,c,d \in \zahl{n},a \neq b \neq c \neq d$. For any $k,l \in \zahl{n-1},k<l$, the number of pairs $(a,b,c,d)$ with $a \neq b \neq c \neq d$ such that 
\[
\min\{a,b\}=k,\min\{c,d\} =l
\]
is $4(n-l)(n-k-2)$, and such that 
\[
\min\{a,b\}=l,\min\{c,d\} =k 
\]
is also $4(n-l)(n-k-2)$. Then
\begin{align*}
  \E [\min\{R_1,R_2\}\min\{R_3,R_4\}] & = \frac{1}{n(n-1)(n-2)(n-3)} \sum_{\substack{k,l=1\\k < l}}^{n-1} 8(n-l)(n-k-2)kl \\
  & = \frac{(n+1)(5n+1)}{45},
\end{align*}
since
\[
  \sum_{\substack{k,l=1\\k \neq l}}^n k^2l^2 = \frac{(n-1)n(n+1)(2n-1)(2n+1)(5n+6)}{180}.
\]
We accordingly obtain
\begin{align*}
  \Cov[\min\{R_1,R_2\},\min\{R_3,R_4\}] &= \E [\min\{R_1,R_2\}\min\{R_3,R_4\}] - [\E [\min\{R_1,R_2\}]]^2 \\
  & = -\frac{4(n+1)}{45}.
\end{align*}

(8) For $\Cov[\min\{R_1,R_2\},\min\{R_1,R_3\}]$, for any $k \in \zahl{n-1}$, the number of triples $(a,b,c)$ with $a \neq b \neq c$ such that 
\[
\min\{a,b\}=k,\min\{a,c\}=k
\]
is $(n-k)(n-k-1)$. For any $k,l \in \zahl{n-1},k<l$, the number of triples $(a,b,c)$ with $a \neq b \neq c$ such that 
\[
\min\{a,b\}=k,\min\{a,c\} =l 
\]
is $2(n-l)$, and such that 
\[
\min\{a,b\}=l,\min\{a,c\} =k 
\]
is also $2(n-l)$. Then
\begin{align*}
  \E [\min\{R_1,R_2\}\min\{R_1,R_3\}] &= \frac{1}{n(n-1)(n-2)} \Big[ \sum_{\substack{k,l=1\\k < l}}^{n-1} 4(n-l)kl + \sum_{k=1}^{n-1} (n-k)(n-k-1)k^2 \Big] \\
  & = \frac{(n+1)(8n+1)}{60},
\end{align*}
which yields
\begin{align*}
  \Cov[\min\{R_1,R_2\},\min\{R_1,R_3\}] &= \E [\min\{R_1,R_2\}\min\{R_1,R_3\}] - [\E [\min\{R_1,R_2\}]]^2 \\
  & = \frac{(n+1)(4n-17)}{180}.
\end{align*}

(9) For $\Var[\min\{R_1,R_2\}]$, we have
\[
  \E [\min\{R_1,R_2\}^2] = \frac{1}{n(n-1)} \sum_{k=1}^{n-1} 2(n-k)k^2 = \frac{n(n+1)}{6},
\]
implying
\[
  \Var[\min\{R_1,R_2\}] = \E [\min\{R_1,R_2\}^2] - [\E [\min\{R_1,R_2\}]]^2 = \frac{(n-2)(n+1)}{18}.
\]
The whole proof is thus complete.
\end{proof}

\subsubsection{Proof of Lemma~\ref{lemma:local,T3}}

\begin{proof}[Proof of Lemma~\ref{lemma:local,T3}]

Under the assumptions of the theorem,
\[
Y|X = x \sim N(\mu_x,\sigma^2),
\]
where $\mu_x = \rho_n x$ and $\sigma^2 = 1- \rho_n^2$. Then for any $x \in \R$, $\mu_x$ is of order $\rho_n$ as $\rho_n \to 0$. Accordingly,
\[
\lim_{\rho_n \to 0}  \rho_n^{-2} (1/\sigma - 1) = 1/2.
\yestag\label{eq:local,rho}
\]

Denote the cumulative distribution function of the standard normal by $\Phi(\cdot)$. Employing Taylor's expansion at $y$,
\begin{align*}
  F_{Y|X=X_1}(y) - F_Y(y) &= \Phi\Big(\frac{y-\mu_{X_1}}{\sigma}\Big) - \Phi(y) \\
  &= \Phi\Big(\frac{y-\rho_nX_1}{\sigma}\Big) - \Phi(y) \\
  &=  f_Y(y) \Big(\frac{y-\rho_nX_1}{\sigma} -y \Big) + \frac{1}{2} f'_Y(y_x) \Big(\frac{y-\rho_nX_1}{\sigma} -y \Big)^2,
\end{align*}
where $y_x$ is between $y$ and $(y-\rho_n X_1)/\sigma$.

Then
\begin{align*}
  &\Big[F_Y(y) -F_{Y|X=X_1}(y)\Big]^2 \\
  = &f_Y^2(y) \Big(\frac{y-\rho_nX_1}{\sigma} -y \Big)^2 + f_Y(y) f'_Y(y_x) \Big(\frac{y-\rho_nX_1}{\sigma} -y \Big)^3 + \frac{1}{4}[f'_Y(y_x)]^2 \Big(\frac{y-\rho_nX_1}{\sigma} -y \Big)^4.
\end{align*}

(1) We first handle the last two terms. Notice that from \eqref{eq:local,rho},
\[
  \Big\lvert \frac{y-\rho_nX_1}{\sigma} -y \Big\rvert \le \Big( \frac{1}{\sigma} - 1\Big) \lvert y \rvert + \frac{1}{\sigma} \rho_n \lvert X_1 \rvert = \Big[\frac{1}{2}\lvert y \rvert \rho_n^2 + \lvert X
  _1 \rvert \rho_n\Big](1+o(1)).
  \yestag\label{eq:local,diff}
\]

Since $f_Y'$ is uniformly bounded for the normal distribution, we obtain
\begin{align*}
  & \Big\lvert f_Y(y) f'_Y(y_x) \Big(\frac{y-\rho_nX_1}{\sigma} -y \Big)^3 + \frac{1}{4}[f'_Y(y_x)]^2 \Big(\frac{y-\rho_nX_1}{\sigma} -y \Big)^4 \Big\rvert\\
  \le & 4 f_Y(y) \lVert f'_Y \rVert_{\infty} \Big[\frac{1}{8}\lvert y \rvert^3 \rho_n^6 + \lvert X
  _1 \rvert^3 \rho_n^3\Big](1+o(1)) + 2 \lVert f'_Y \rVert_{\infty}^2 \Big[\frac{1}{16} y^4 \rho_n^8 + X_1^4 \rho_n^4\Big](1+o(1)),
\end{align*}
where we use \eqref{eq:local,diff} and the inequality $(a+b)^n \le 2^{n-1}(|a|^n + |b|^n)$ for any $a,b \in \R$ and $n \in \bN$.

Since the fourth moments of $X$ and $Y$ are bounded for the normal distribution, then
\[
  \E \Big[ \int \Big\lvert f_Y(y) f'_Y(y_x) \Big(\frac{y-\rho_nX_1}{\sigma} -y \Big)^3 + \frac{1}{4}[f'_Y(y_x)]^2 \Big(\frac{y-\rho_nX_1}{\sigma} -y \Big)^4 \Big\rvert f_Y(y) \d y\Big] = o(\rho_n^2).
\]

(2) We then turn to the first term. For it, we have
\begin{align*}
  & \Big\lvert \Big(\frac{y-\rho_nX_1}{\sigma} -y \Big)^2 - \frac{1}{\sigma^2} X_1^2 \rho_n^2 \Big\rvert\\
  \le & \Big( \frac{1}{\sigma} - 1\Big)^2 y^2 + 2\frac{1}{\sigma}\Big( \frac{1}{\sigma} - 1\Big)\lvert y \rvert \lvert X_1 \rvert \rho_n\\
  = & \Big[ \frac{1}{4}y^2 \rho_n^4 + \lvert y \rvert \lvert X_1 \rvert \rho_n^3 \Big](1+o(1)).
\end{align*}

(3) Combining the above two steps yields
\begin{align*}
  \E [T_3] &= \E \Big[ \int \Big[F_Y(y) -F_{Y|X=X_1}(y)\Big]^2 f_Y(y) \d y \Big]\\
  &= \E \Big[ \int \Big(\frac{y-\rho_nX_1}{\sigma} -y \Big)^2 f_Y^3(y) \d y \Big] + o(\rho_n^2)\\
  &= \frac{1}{\sigma^2} \Big[ \int f_Y^3(y) \d y \Big] \E [X_1^2] \rho_n^2 + o(\rho_n^2)\\
  &= \frac{1}{\sigma^2} \Big[ \int f_Y^3(y) \d y \Big] \rho_n^2 + o(\rho_n^2)\\
  &= \Big[ \int f_Y^3(y) \d y \Big] \rho_n^2 + o(\rho_n^2),
  \yestag\label{eq:local,T3}
\end{align*}
and we thus complete the proof.
\end{proof}

\subsubsection{Proof of Lemma~\ref{lemma:local,T4}}

\begin{proof}[Proof of Lemma~\ref{lemma:local,T4}]

Before proving Lemma~\ref{lemma:local,T4} (as well as Lemma \ref{lemma:local,T5} ahead), we first establish the following two lemmas about $k$-nearest neighbors.

The first lemma establishes the convergence rate of $X_{j_m(1)}$ to $X_{1}$ for any $m \in \zahl{M}$, where $M$ is allowed to increase with $n$. 



\begin{lemma}\label{lemma:dist}
  Let $W_{n,1},\ldots,W_{n,n}$ be $n$ independent copies from a probability measure $\P_n$ that is over $\R$ and is allowed to change with $n$. Assume $\P_n$ is supported on $[-D_n,D_n]$ for some positive constant $D_n$ also allowed to change with $n$. Then for any $m \in \zahl{M}$,
  \[
    \E [W_{n,j_m(1)} - W_{n,1}] \le 2 \frac{m}{n}D_n,
  \]
  where $j_m(1)$ is the index of the $m$-th right NN of $W_{n,1}$ and is $1$ if $m$-th right NN does not exist.
\end{lemma}

We then cite a result on the lower bound of the $k$-nearest neighbor distance without proof (although we consider the triangular array setting, the proof is the same).

\begin{lemma}\label{lemma:dist,lower}\citep[Theorem 4.1]{MR3445317}
Let $W_{n,1},\ldots,W_{n,n} \in \R$ be $n$ independent copies from a probability measure with density $f_n$. Assume that $\{\lVert f_i \rVert_{\infty}\}_{i=1}^\infty$ are uniformly bounded by a universal constant $C_0$. If $k/\log n \to \infty$ as $n \to \infty$, then with probability one, for all $n$ large enough,
  \[
    \inf_{x \in \R} \Big\lvert W_{(k)}(x) - x \Big\rvert \ge C \frac{k}{n},
  \]
where $C>0$ is a constant only depending on $C_0$, and $W_{(k)}(x)$ is the k-nearest neighbor of $x$ among $\{W_{n,i}\}_{i=1}^n$.
\end{lemma}

{\bf Upper bound.} To prove Lemma~\ref{lemma:local,T4}, we first consider the upper bound. Since $T_4$ depends on the discrepancy between $X_1$ and $X_{j_U(1)}$, and $X$ has unbounded support, we first truncate it to two parts and obtain
\begin{align}\label{eq:han-T4}
  \E [T_4] = \E \Big[ T_4 \ind\Big(\max_{i \in \zahl{n}} \lvert X_i \rvert \le D_n\Big) \Big] + \E \Big[ T_4 \ind\Big(\max_{i \in \zahl{n}} \lvert X_i \rvert > D_n\Big) \Big],
\end{align}
where we take $D_n = \sqrt{2\log(n^3/M)}$. 

(1) For the second term on the righthand side of \eqref{eq:han-T4}, noticing $|T_4|$ is bounded by 2, by Gaussian tail probability, 
\begin{align*}
  \E \Big[ \Big\lvert T_4 \ind\Big(\max_{i \in \zahl{n}} \lvert X_i \rvert > D_n\Big) \Big\rvert \Big] \le 2\P\Big(\max_{i \in \zahl{n}} \lvert X_i \rvert > D_n\Big) \le 4n\exp(-D_n^2/2),
\end{align*}

(2) For the first term on the righthand side of \eqref{eq:han-T4}, we perform an analysis similar to that of Lemma \ref{lemma:local,T3}. In detail, invoking Taylor expansion at $y$ gives
\begin{align*}
  F_{Y|X=X_1}(y) = & F_Y(y) + f_Y(y) \Big(\frac{y-\rho_nX_1}{\sigma} -y \Big) + \frac{1}{2} f'_Y(y) \Big(\frac{y-\rho_nX_1}{\sigma} -y \Big)^2 + \\
  & +\frac{1}{6} f''_Y(y_x) \Big(\frac{y-\rho_nX_1}{\sigma} -y \Big)^3,\\
  F_{Y|X=X_{j_U(1)}}(y) = & F_Y(y) + f_Y(y) \Big(\frac{y-\rho_nX_{j_U(1)}}{\sigma} -y \Big) + \frac{1}{2} f'_Y(y) \Big(\frac{y-\rho_nX_{j_U(1)}}{\sigma} -y \Big)^2 + \\
  & + \frac{1}{6} f''_Y(y'_x) \Big(\frac{y-\rho_nX_{j_U(1)}}{\sigma} -y \Big)^3,
\end{align*}
where $y_x$ is between $y$ and $(y-\rho_n X_1)/\sigma$, and $y'_x$ is between $y$ and $(y-\rho_n X_{j_U(1)})/\sigma$.

Then
\begin{align*}
  &F_{Y|X=X_1}(y) - F_{Y|X=X_{j_U(1)}}(y) \\
  = & \frac{1}{\sigma} f_Y(y) \Big( X_{j_U(1)} - X_1 \Big) \rho_n  + \frac{1}{2\sigma} f'_Y(y) \Big( X_{j_U(1)} - X_1 \Big) \Big[ 2 \Big( \frac{1}{\sigma} -1 \Big) y - \frac{1}{\sigma} \Big( X_1 + X_{j_U(1)} \Big) \rho_n \Big] \rho_n\\
  & +  \frac{1}{6} f''_Y(y_x) \Big(\frac{y-\rho_nX_1}{\sigma} -y \Big)^3 - \frac{1}{6} f''_Y(y'_x) \Big(\frac{y-\rho_nX_{j_U(1)}}{\sigma} -y \Big)^3.
\end{align*}

(2.1) We first consider the first term. Conditional on the event $\{\max_{i \in \zahl{n}} \lvert X_i \rvert \le D_n\}$, $[X_i]_{i=1}^n$ are still i.i.d, but with the probability measure $X \given \{\lvert X \rvert \le D_n \}$, which is supported on $[-D_n,D_n]$. Then from Lemma~\ref{lemma:dist}, for any $y \in \R$,
\begin{align*}
  & \E \Big[ \frac{1}{\sigma} f_Y(y) \Big( X_{j_U(1)} - X_1 \Big) \rho_n\ind\Big(\max_{i \in \zahl{n}} \lvert X_i \rvert \le D_n\Big) \Big]\\
  \le & \E \Big[ \frac{1}{\sigma} f_Y(y) \Big( X_{j_U(1)} - X_1 \Big) \rho_n \Biggiven \max_{i \in \zahl{n}} \lvert X_i \rvert \le D_n \Big]\\
  \le & 2 f_Y(y)\Big[\frac{M}{n}D_n \rho_n \Big]\frac{1}{\sigma} \\
  = & 2 f_Y(y)\Big[ \frac{M}{n}D_n \rho_n \Big] \Big[ 1 + \frac{1}{2} \rho_n^2 + o(\rho_n^2)\Big].
\end{align*}

(2.2) For the second term, for any $y \in \R$,
\begin{align*}
  &  \E \Big[ \Big\lvert \frac{1}{2\sigma} f'_Y(y) \Big( X_{j_U(1)} - X_1 \Big) \Big[ 2 \Big( \frac{1}{\sigma} -1 \Big) y - \frac{1}{\sigma} \Big( X_1 + X_{j_U(1)} \Big) \rho_n \Big] \rho_n \ind\Big(\max_{i \in \zahl{n}} \lvert X_i \rvert \le D_n\Big) \Big\rvert \Big]\\
  = & \frac{1}{2} f'_Y(y) \E \Big[\Big\lvert \Big( X_{j_U(1)} - X_1 \Big) \Big[ y \rho_n^2 - \Big( X_1 + X_{j_U(1)} \Big) \rho_n \Big] \rho_n \ind\Big(\max_{i \in \zahl{n}} \lvert X_i \rvert \le D_n\Big) \Big\rvert \Big] (1+o(1)) \\
  \le & \Big[ \frac{1}{2} \lVert f_Y' \rVert_{\infty} \lvert y \rvert \E \Big[ \Big( X_{j_U(1)} - X_1 \Big) \ind\Big(\max_{i \in \zahl{n}} \lvert X_i \rvert \le D_n\Big) \Big] \rho_n^3 \\
  & + \frac{1}{2} \lVert f_Y' \rVert_{\infty} \E \Big[ \Big( X_{j_U(1)} - X_1 \Big) \Big\lvert X_1 + X_{j_U(1)} \Big\rvert \ind\Big(\max_{i \in \zahl{n}} \lvert X_i \rvert \le D_n\Big) \Big] \rho_n^2 \Big] (1+o(1))\\
  \le &  \Big[ \lVert f_Y' \rVert_{\infty} \lvert y \rvert \frac{M}{n} D_n \rho_n^3 + 2 \lVert f_Y' \rVert_{\infty} \frac{M}{n} D_n^2 \rho_n^2 \Big] (1+o(1)),
\end{align*}
where the last step is due to that $\lvert X_1 + X_{j_U(1)} \rvert \le 2D_n$ given the event $\{\max_{i \in \zahl{n}} \lvert X_i \rvert \le D_n\}$ and Lemma~\ref{lemma:dist}.

(2.3) For the third term, for any $y \in \R$, from \eqref{eq:local,diff},
\begin{align*}
  &\E \Big[\Big\lvert\Big\{ \frac{1}{6} f''_Y(y_x) \Big(\frac{y-\rho_nX_1}{\sigma} -y \Big)^3 - \frac{1}{6} f''_Y(y'_x) \Big(\frac{y-\rho_nX_{j_U(1)}}{\sigma} -y \Big)^3\Big\} \ind\Big(\max_{i \in \zahl{n}} \lvert X_i \rvert \le D_n\Big) \Big\rvert \Big]\\
  \le & \frac{1}{6} \lVert f''_Y \rVert_{\infty} \E \Big[ \Big(\lvert y \rvert^3 \rho_n^6 + 4 \lvert X_1 \rvert^3 \rho_n^3 + 4 \lvert X_{j_U(1)} \rvert^3 \rho_n^3 \Big) \ind\Big(\max_{i \in \zahl{n}} \lvert X_i \rvert \le D_n\Big) \Big] (1+o(1))\\
  \le & \frac{1}{6} \lVert f''_Y \rVert_{\infty} \E \Big[ \Big(\lvert y \rvert^3 \rho_n^6 + 20 \lvert X_1 \rvert^3 \rho_n^3 + 16 \lvert X_{j_U(1)} - X_1 \rvert^3 \rho_n^3 \Big) \ind\Big(\max_{i \in \zahl{n}} \lvert X_i \rvert \le D_n\Big) \Big] (1+o(1))\\
  \le & \frac{1}{6} \lVert f''_Y \rVert_{\infty} \Big[ \lvert y \rvert^3 \rho_n^6 + 20 \E [\lvert X_1 \rvert^3] \rho_n^3 + 128 \frac{M}{n} D_n^3 \rho_n^3 \Big] (1+o(1)),
\end{align*}
where the last step is due to Lemma~\ref{lemma:dist} and the fact that $\lvert X_{j_U(1)} - X_1 \rvert \le 2D_n$ given the event $\{\max_{i \in \zahl{n}} \lvert X_i \rvert \le D_n\}$. 

(2.4) Summarizing the above three steps, we obtain
\begin{align*}
  & \E \Big[ T_4 \ind\Big(\max_{i \in \zahl{n}} \lvert X_i \rvert \le D_n\Big) \Big]\\
  =& \E \Big[ \int \Big[1 - F_Y(y)\Big] \Big[F_{Y|X=X_1}(y) - F_{Y|X=X_{j_U(1)}}(y) \Big] f_Y(y) \d y \ind\Big(\max_{i \in \zahl{n}} \lvert X_i \rvert \le D_n\Big) \Big]\\
  =& \int \Big[1 - F_Y(y)\Big] \E \Big[\Big(F_{Y|X=X_1}(y) - F_{Y|X=X_{j_U(1)}}(y) \Big) \ind\Big(\max_{i \in \zahl{n}} \lvert X_i \rvert \le D_n\Big) \Big] f_Y(y) \d y\\
  \le & 2 \Big[ \int \Big[1 - F_Y(y)\Big] f_Y^2(y) \d y \Big] \frac{M}{n}D_n \rho_n + o(\rho_n^2),
  \yestag\label{eq:local,T4,upper}
\end{align*}
if $MD_n^3/n \to 0$ as $n \to \infty$; recall we pick $D_n = \sqrt{2\log(n^3/M)}$ so that this is true as long as $M (\log n)^{3/2}/n \to 0$.

{\bf Lower bound.} We then consider the lower bound. 
Let $f_X^{(n)}$ be the density of $X \given \{\lvert X \rvert \le D_n \}$. Since $D_n \to \infty$ as $n \to \infty$ and $f_X$ is bounded, it is true that $\{f_X^{(n)}\}_{i=1}^n$ is uniformly bounded by a universal constant.

Notice that conditional on the event $\{\max_{i \in \zahl{n}} \lvert X_i \rvert \le D_n\}$, $[X_i]_{i=1}^n$ are $n$ independent copies from $f_X^{(n)}$. We then consider the probability measure of $X_1$ conditional on $[X_i]_{i=2}^n$, the event $\{\max_{i \in \zahl{n}} \lvert X_i \rvert \le D_n\}$ and the event $\{j_m(1) \neq 1\}$. This probability measure depends on $[X_i]_{i=2}^n$, $D_n$, and $m$. For any $m \in \zahl{M}$ and $X_1 \in \R$, conditional on $\{j_m(1) \neq 1\}$, we always have
\[
  X_{j_m(1)} - X_1 \ge \inf_{x \in \R} \lvert X_{(m)}(x) - x \rvert,
\]
where $X_{(m)}(x)$ is the $m$-nearest neighbor of $x$ among $\{X_i\}_{i=2}^n$. We can then take expectation with respect to this probability measure and obtain
\[
  \E_{X_1} \Big[  \Big( X_{j_m(1)} - X_1 \Big)  \Biggiven [X_i]_{i=2}^n, j_m(1) \neq 1, \max_{i \in \zahl{n}} \lvert X_i \rvert \le D_n \Big] \ge \inf_{x \in \R} \lvert X_{(m)}(x) - x \rvert.
  \yestag\label{eq:local,dist,lower}
\]

Then from Lemma~\ref{lemma:dist,lower}, for any $y \in \R$,
\begin{align*}
  & \E \Big[ \frac{1}{\sigma} f_Y(y) \Big( X_{j_U(1)} - X_1 \Big) \rho_n\ind\Big(\max_{i \in \zahl{n}} \lvert X_i \rvert \le D_n\Big) \Big]\\
  = & \frac{1}{\sigma} f_Y(y) \rho_n \E \Big[  \Big( X_{j_U(1)} - X_1 \Big)  \Biggiven \max_{i \in \zahl{n}} \lvert X_i \rvert \le D_n \Big] \P \Big( \max_{i \in \zahl{n}} \lvert X_i \rvert \le D_n \Big)  \\
  = & \frac{1}{\sigma} f_Y(y) \rho_n  \P \Big( \max_{i \in \zahl{n}} \lvert X_i \rvert \le D_n \Big) \Big\{ \frac{1}{M} \sum_{m=1}^M \E \Big[  \Big( X_{j_m(1)} - X_1 \Big)  \Biggiven \max_{i \in \zahl{n}} \lvert X_i \rvert \le D_n \Big] \Big\} \\
  = & \frac{1}{\sigma} f_Y(y) \rho_n  \P \Big( \max_{i \in \zahl{n}} \lvert X_i \rvert \le D_n \Big) \\
  & \cdot \Big\{ \frac{1}{M} \sum_{m=1}^M \P \Big( j_m(1) \neq 1 \Biggiven \max_{i \in \zahl{n}} \lvert X_i \rvert \le D_n \Big) \E \Big[  \Big( X_{j_m(1)} - X_1 \Big)  \Biggiven j_m(1) \neq 1, \max_{i \in \zahl{n}} \lvert X_i \rvert \le D_n \Big] \Big\} \\
  \ge & \frac{1}{\sigma} f_Y(y) \rho_n  \P \Big( \max_{i \in \zahl{n}} \lvert X_i \rvert \le D_n \Big) \P \Big( j_M(1) \neq 1 \Biggiven \max_{i \in \zahl{n}} \lvert X_i \rvert \le D_n \Big)  \\
  &\quad\quad \cdot\Big\{ \frac{1}{M} \sum_{m=1}^M \E \Big[  \Big( X_{j_m(1)} - X_1 \Big)  \Biggiven j_m(1) \neq 1, \max_{i \in \zahl{n}} \lvert X_i \rvert \le D_n \Big] \Big\}.
\end{align*}

Let $\mX_{\setminus 1} = (X_2,\ldots,X_n)$. Notice that from \eqref{eq:local,dist,lower} and that the distance of $m$-NN is increasing with respect to $m$ for any point,
\begin{align*}
  & \frac{1}{M} \sum_{m=1}^M \E \Big[  \Big( X_{j_m(1)} - X_1 \Big)  \Biggiven j_m(1) \neq 1, \max_{i \in \zahl{n}} \lvert X_i \rvert \le D_n \Big] \\
  = & \frac{1}{M} \sum_{m=1}^M \E_{\mX_{\setminus 1}} \Big\{ \E_{X_1} \Big[  \Big( X_{j_m(1)} - X_1 \Big)  \Biggiven [X_i]_{i=2}^n, j_m(1) \neq 1, \max_{i \in \zahl{n}} \lvert X_i \rvert \le D_n \Big]  \Biggiven \max_{i \in \zahl{n}\setminus\{1\}} \lvert X_i \rvert \le D_n \Big\}\\
  \ge & \frac{1}{M} \sum_{m=1}^M \E_{\mX_{\setminus 1}} \Big[ \inf_{x \in \R} \big\lvert X_{(m)}(x) - x \big\rvert \Biggiven \max_{i \in \zahl{n}\setminus\{1\}} \lvert X_i \rvert \le D_n \Big]\\
  \ge & \frac{1}{M} \sum_{m=\lfloor \frac{M}{2} \rfloor}^M \E_{\mX_{\setminus 1}} \Big[ \inf_{x \in \R} \big\lvert X_{(m)}(x) - x \big\rvert \Biggiven \max_{i \in \zahl{n}\setminus\{1\}} \lvert X_i \rvert \le D_n \Big]\\
  \ge & \frac{1}{M} \sum_{m=\lfloor \frac{M}{2} \rfloor}^M \E_{\mX_{\setminus 1}} \Big[ \inf_{x \in \R} \big\lvert X_{(\lfloor \frac{M}{2} \rfloor)}(x) - x \big\rvert \Biggiven \max_{i \in \zahl{n}\setminus\{1\}} \lvert X_i \rvert \le D_n \Big]\\
  \ge & \frac{1}{2} \E_{\mX_{\setminus 1}} \Big[ \inf_{x \in \R} \big\lvert X_{(\lfloor \frac{M}{2} \rfloor)}(x) - x \big\rvert \Biggiven \max_{i \in \zahl{n}\setminus\{1\}} \lvert X_i \rvert \le D_n \Big]\\ 
  \gtrsim & \frac{M}{n},
\end{align*}
where $\lfloor \cdot\rfloor$ is the floor function and we apply Lemma~\ref{lemma:dist,lower} by taking $k = \lfloor M/2 \rfloor$ since $M/\log n \to \infty$.

Notice that conditional on the event $\{\max_{i \in \zahl{n}} \lvert X_i \rvert \le D_n\}$, $[X_i]_{i=1}^n$ are still i.i.d. Then
\[
  \P \Big( j_M(1) \neq 1 \Biggiven \max_{i \in \zahl{n}} \lvert X_i \rvert \le D_n \Big) = 1 - \frac{M}{n}.
\]

Then we obtain
\begin{align*}
  & \E \Big[ \frac{1}{\sigma} f_Y(y) \Big( X_{j_U(1)} - X_1 \Big) \rho_n\ind\Big(\max_{i \in \zahl{n}} \lvert X_i \rvert \le D_n\Big) \Big]\\
  \gtrsim & \frac{1}{\sigma} f_Y(y) \frac{M}{n} \rho_n  \P \Big( \max_{i \in \zahl{n}} \lvert X_i \rvert \le D_n \Big) \P \Big( j_M(1) \neq 1 \Biggiven \max_{i \in \zahl{n}} \lvert X_i \rvert \le D_n \Big)\\
  \gtrsim & f_Y(y) \frac{M}{n} \rho_n \Big(1 + \frac{1}{2}\rho_n^2 + o(\rho_n^2) \Big),
  \yestag\label{eq:local,T4,lower}
\end{align*}
since $\P \Big( \max_{i \in \zahl{n}} \lvert X_i \rvert \le D_n \Big)$ and $\P \Big( j_M(1) \neq 1 \Biggiven \max_{i \in \zahl{n}} \lvert X_i \rvert \le D_n \Big)$ both converge to 1 as $n \to \infty$.

Combining \eqref{eq:local,T4,upper} and \eqref{eq:local,T4,lower} implies
\[
  \frac{M}{n}\rho_n + o(\rho_n^2) \lesssim E[T_4] \lesssim \frac{M}{n}D_n \rho_n + n \exp(-D_n^2/2) + o(\rho_n^2).
\]
Plugging $D_n = \sqrt{2\log(n^3/M)}$ in, we have
\[
  \frac{M}{n}\rho_n + o(\rho_n^2) \lesssim E[T_4] \lesssim \frac{M}{n}\sqrt{\log\Big(\frac{n^3}{M}\Big)}\rho_n + \frac{M}{n^2} + o(\rho_n^2) \lesssim \frac{M}{n} \sqrt{\log n} \rho_n + \frac{M}{n^2} + o(\rho_n^2)
\]
and thus complete the proof.
\end{proof}

\subsubsection{Proof of Lemma~\ref{lemma:local,T5}}

\begin{proof}[Proof of Lemma~\ref{lemma:local,T5}]
For $T_5$, we have:
\[
  \E [\lvert T_5 \rvert] = \E \Big[ \lvert T_5 \rvert \ind\Big(\max_{i \in \zahl{n}} \lvert X_i \rvert \le D_n\Big) \Big] + \E \Big[ \lvert T_5 \rvert \ind\Big(\max_{i \in \zahl{n}} \lvert X_i \rvert > D_n\Big) \Big],
\]
where we take $D_n = \sqrt{2\log(n^3/M)}$.

Noticing $|T_5|$ is bounded by 4 since the cumulative distribution function is bounded by 1, Gaussian tail probability yields
\[
  \E \Big[ \Big\lvert T_5 \ind\Big(\max_{i \in \zahl{n}} \lvert X_i \rvert > D_n\Big) \Big\rvert \Big] \le 4\P\Big(\max_{i \in \zahl{n}} \lvert X_i \rvert > D_n\Big) \le 8n\exp(-D_n^2/2).
\]
Taylor expansion at $y$ then gives
\begin{align*}
  F_{Y|X=X_1}(y) = & F_Y(y) + f_Y(y_x) \Big(\frac{y-\rho_nX_1}{\sigma} -y \Big),\\
  F_{Y|X=X_{j_U(1)}}(y) = & F_Y(y) + f_Y(y'_x) \Big(\frac{y-\rho_nX_{j_U(1)}}{\sigma} -y \Big),
\end{align*}
where $y_x$ is between $y$ and $(y-\rho_n X_1)/\sigma$, and $y'_x$ is between $y$ and $(y-\rho_n X_{j_U(1)})/\sigma$.

Then for any $y \in \R$,
\begin{align*}
  & \Big\lvert \Big[F_{Y|X=X_1}(y) - F_Y(y)\Big] \Big[F_{Y|X=X_1}(y) - F_{Y|X=X_{j_U(1)}}(y) \Big] \Big\rvert\\
  = & \Big\lvert f_Y(y_x) \Big(\frac{y-\rho_nX_1}{\sigma} -y \Big) f_Y(y'_x) \Big(\frac{\rho_n}{\sigma} (X_{j_U(1)} - X_1) \Big) \Big\rvert\\
  \le & \frac{1}{\sigma} \lVert f_Y \rVert_{\infty}^2 \Big[ \lvert y \rvert \Big( \frac{1}{\sigma} - 1 \Big) (X_{j_U(1)} - X_1) \rho_n + \frac{1}{\sigma} \lvert X_1 \rvert (X_{j_U(1)} - X_1) \rho_n^2 \Big]\\
  = & \lVert f_Y \rVert_{\infty}^2 \Big[ \frac{1}{2} \lvert y \rvert  (X_{j_U(1)} - X_1) \rho_n^3 + \lvert X_1 \rvert (X_{j_U(1)} - X_1) \rho_n^2 \Big] (1+o(1)).
\end{align*}
Accordingly, we obtain
\begin{align*}
  & \E \Big[ \Big\lvert \Big[F_{Y|X=X_1}(y) - F_Y(y)\Big] \Big[F_{Y|X=X_1}(y) - F_{Y|X=X_{j_U(1)}}(y) \Big] \Big\rvert \ind\Big(\max_{i \in \zahl{n}} \lvert X_i \rvert \le D_n\Big) \Big] \\
  \le & \lVert f_Y \rVert_{\infty}^2 \Big[ \lvert y \rvert  \frac{M}{n} D_n \rho_n^3 + 2 \frac{M}{n} D_n^2 \rho_n^2 \Big] (1+o(1)),
\end{align*}
where the last step is due to Lemma~\ref{lemma:dist} and $\lvert X_1 \rvert \le D_n$ given the event $\{\max_{i \in \zahl{n}} \lvert X_i \rvert \le D_n\}$.

Then
\begin{align*}
  & \E \Big[ \lvert T_5 \rvert \ind\Big(\max_{i \in \zahl{n}} \lvert X_i \rvert \le D_n\Big) \Big]
  \le  \lVert f_Y \rVert_{\infty}^2 \Big[ \E[\lvert Y \rvert] \frac{M}{n} D_n \rho_n^3 + 2 \frac{M}{n} D_n^2 \rho_n^2 \Big] (1+o(1)) = o(\rho_n^2),
\end{align*}
if $MD_n^2/n \to 0$ as $n \to \infty$; given $D_n = \sqrt{2\log(n^3/M)}$, this is true as long as $M\log n/n\to 0$.

Since we take $D_n = \sqrt{2\log(n^3/M)}$, then
\[
  \E \Big[ \Big\lvert T_5 \ind\Big(\max_{i \in \zahl{n}} \lvert X_i \rvert > D_n\Big) \Big\rvert \Big] \le 8n\exp(-D_n^2/2) \lesssim \frac{M}{n^2}.
\]
Combining the above two inequalities, we have proved
\[
  \E [\lvert T_5 \rvert] \lesssim \frac{M}{n^2} + o(\rho_n^2)
\]
and thus complete the proof.
\end{proof}

\subsubsection{Proof of Lemma~\ref{lemma:local,T7}}

\begin{proof}[Proof of Lemma~\ref{lemma:local,T7}]
We have
\begin{align*}
  \E [T_7] =& \E \Big[\int \Big[F_Y(y) - F_{Y|X=X_1}(y)\Big] \Big[F_Y(y) + F_{Y|X=X_1}(y) - 1\Big] f_Y(y) \d y \ind(j_U(1) = 1)\Big]\\
  =& - \E \Big[\int \Big[F_Y(y) - F_{Y|X=X_1}(y)\Big]^2 f_Y(y) \d y \ind(j_U(1) = 1)\Big]\\
  & + \E \Big[\int \Big[F_Y(y) - F_{Y|X=X_1}(y)\Big] \Big[2 F_Y(y) - 1\Big] f_Y(y) \d y \ind(j_U(1) = 1)\Big].
\end{align*}

For the first term,
\[
  F_{Y|X=X_1}(y) - F_Y(y) = f_Y(y_x) \Big(\frac{y-\rho_nX_1}{\sigma} -y \Big),
\]
where $y_x$ is between $y$ and $(y-\rho_n X_1)/\sigma$. Together with \eqref{eq:local,diff},
\begin{align*}
  & \E \Big[\int \Big[F_Y(y) - F_{Y|X=X_1}(y)\Big]^2 f_Y(y) \d y \ind(j_U(1) = 1)\Big]\\
  \le & \E \Big[\int 2 \lVert f_Y \rVert_{\infty}^2 \Big[\frac{1}{4} y^2 \rho_n^4 + X_1^2 \rho_n^2 \Big] f_Y(y) \d y \ind(j_U(1) = 1)\Big] (1+o(1))\\
  = & 2 \lVert f_Y \rVert_{\infty}^2 \rho_n^2 \E \Big[ X_1^2 \ind(j_U(1) = 1) \Big](1+o(1)) + o(\rho_n^2)\\
  \le & 2 \lVert f_Y \rVert_{\infty}^2 \rho_n^2 \sqrt{\E [X_1^4] \P(j_U(1) = 1)} (1+o(1)) + o(\rho_n^2) = o(\rho_n^2),
\end{align*}
where the last step is from H\"older's inequality and $\P(j_U(1) = 1) = o(1)$.

For the second term,
\[
  F_{Y|X=X_1}(y) - F_Y(y) =  f_Y(y) \Big(\frac{y-\rho_nX_1}{\sigma} -y \Big) + \frac{1}{2} f'_Y(y_x) \Big(\frac{y-\rho_nX_1}{\sigma} -y \Big)^2,
\]
where $y_x$ is between $y$ and $(y-\rho_n X_1)/\sigma$.

Notice that
\[
  \int \Big[2 F_Y(y) - 1\Big] f_Y^2(y) \d y = 0.
\]

Then together with \eqref{eq:local,diff},
\begin{align*}
  & \Big\lvert \E \Big[\int \Big[F_Y(y) - F_{Y|X=X_1}(y)\Big] \Big[2 F_Y(y) - 1\Big] f_Y(y) \d y \ind(j_U(1) = 1)\Big] \Big\rvert\\
  = & \Big\lvert \E \Big[\int \Big[f_Y(y) \Big(\frac{y-\rho_nX_1}{\sigma} -y \Big) + \frac{1}{2} f'_Y(y_x) \Big(\frac{y-\rho_nX_1}{\sigma} -y \Big)^2\Big] \Big[2 F_Y(y) - 1\Big] f_Y(y) \d y \ind(j_U(1) = 1)\Big]\Big\rvert\\
  = & \Big\lvert \E \Big[\int \Big[\Big(\frac{1}{\sigma} -1 \Big) y f_Y(y) + \frac{1}{2} f'_Y(y_x) \Big(\frac{y-\rho_nX_1}{\sigma} -y \Big)^2\Big] \Big[2 F_Y(y) - 1\Big] f_Y(y) \d y \ind(j_U(1) = 1)\Big]\Big\rvert\\
  \le & \Big\lvert \int y \Big[2 F_Y(y) - 1\Big] f_Y^2(y) \d y \Big\rvert \Big(\frac{1}{\sigma} -1 \Big) \P(j_U(1) = 1) \\
  & + \E \Big[  \lVert f'_Y \rVert_{\infty} \int \Big[\frac{1}{4} y^2 \rho_n^4 + X_1^2 \rho_n^2 \Big] \Big\lvert 2F_Y(y) - 1 \Big\rvert f_Y(y) \d y \ind(j_U(1) = 1) \Big] (1+o(1)) = o(\rho_n^2),
\end{align*}
where the last step is also from H\"older's inequality and $\P(j_U(1) = 1) = o(1)$. This shows $\Big\lvert E[T_7] \Big\rvert = o(\rho_n^2)$.
\end{proof}

\subsubsection{Proof of Lemma~\ref{lemma:alter,var,decompose}}

\begin{proof}[Proof of Lemma~\ref{lemma:alter,var,decompose}]

~\\
{\bf The first term.} Conditional on $\mX$, if $j_m(i) \neq i$,
\[
  \min\{R_i,R_{j_m(i)}\} = 1 + \sum_{k \neq i, k \neq j_m(i)} \ind(Y_k \le \min\{Y_i,Y_{j_m(i)}\}).
\]

If $j_m(i) = i$, then
\[
  \min\{R_i,R_{j_m(i)}\} = 1 + \sum_{k \neq i} \ind(Y_k \le Y_i).
\]

Then
\begin{align*}
  & \E \Big[ \Var \Big[ \min\{R_i,R_{j_m(i)}\} \Biggiven \mX \Big] \Big]\\
  = & \E \Big[ \Var \Big[ \min\{R_i,R_{j_m(i)}\} \Biggiven \mX \Big] \ind(j_m(i) \neq i) \Big] + \E \Big[ \Var \Big[ \min\{R_i,R_{j_m(i)}\} \Biggiven \mX \Big] \ind(j_m(i) = i) \Big]\\
  = & \E \Big[ \Var \Big[ \sum_{k \neq i, k \neq j_m(i)} \ind(Y_k \le \min\{Y_i,Y_{j_m(i)}\}) \Biggiven \mX \Big] \ind(j_m(i) \neq i) \Big] \\
  & + \E \Big[ \Var \Big[ \sum_{k \neq i} \ind(Y_k \le Y_i) \Biggiven \mX \Big] \ind(j_m(i) = i) \Big]\\
  = & (n-2) \E \Big[ \Var \Big[ \ind(Y \le \min\{Y_i,Y_{j_m(i)}\}) \Biggiven \mX \Big] \ind(j_m(i) \neq i) \Big] \\
  & + (n-2)(n-3) \E\Big[\Cov \Big[\ind(Y \le \min\{Y_i,Y_{j_m(i)}\}),\ind(\tY \le \min\{Y_i,Y_{j_m(i)}\}) \Biggiven \mX \Big] \ind(j_m(i) \neq i) \Big]\\
  & + (n-1) \E \Big[ \Var \Big[ \ind(Y \le \min\{Y_i,Y_{j_m(i)}\}) \Biggiven \mX \Big] \ind(j_m(i) = i) \Big] \\
  & + (n-1)(n-2) \E\Big[\Cov \Big[\ind(Y \le \min\{Y_i,Y_{j_m(i)}\}),\ind(\tY \le \min\{Y_i,Y_{j_m(i)}\}) \Biggiven \mX \Big] \ind(j_m(i) = i) \Big]\\
  = & (n-2) \E \Big[ \Var \Big[ \ind(Y \le \min\{Y_i,Y_{j_m(i)}\}) \Biggiven \mX \Big] \Big] \\
  & + (n-2)(n-3) \E\Big[\Cov \Big[\ind(Y \le \min\{Y_i,Y_{j_m(i)}\}),\ind(\tY \le \min\{Y_i,Y_{j_m(i)}\}) \Biggiven \mX \Big] \Big]\\
  & + \E \Big[ \Var \Big[ \ind(Y \le \min\{Y_i,Y_{j_m(i)}\}) \Biggiven \mX \Big] \ind(j_m(i) = i) \Big]\\
  & + 2(n-2) \E\Big[\Cov \Big[\ind(Y \le \min\{Y_i,Y_{j_m(i)}\}),\ind(\tY \le \min\{Y_i,Y_{j_m(i)}\}) \Biggiven \mX \Big] \ind(j_m(i) = i) \Big],
  \yestag\label{eq:local,var,decompose1}
\end{align*}
where $Y,\tY$ are two independent random variables from $F_Y$ and are both independent with $\mX$.

Notice that the sum of the third and the fourth term in \eqref{eq:local,var,decompose1} is dominated by $(2n-3)\P (j_m(i) = i)$ since the variance and the covariance of indicator functions are bounded by 1. Then the sum is $O(M/n)=o(n)$ since $M/n \to 0$ as $n \to \infty$. Since the first term is of order $n$ and the second term is of order $n^2$, then it suffices to consider the second term.

We have
\begin{align*}
  & \E\Big[\Cov \Big[\ind(Y \le \min\{Y_i,Y_{j_m(i)}\}),\ind(\tY \le \min\{Y_i,Y_{j_m(i)}\}) \Biggiven \mX \Big] \Big] \\
  = & \E\Big[\E \Big[\ind(Y \le \min\{Y_i,Y_{j_m(i)}\}) \ind(\tY \le \min\{Y_i,Y_{j_m(i)}\}) \Biggiven \mX \Big] \Big] \\
  & - \E\Big[\E \Big[\ind(Y \le \min\{Y_i,Y_{j_m(i)}\}) \Biggiven \mX \Big] \E \Big[ \ind(\tY \le \min\{Y_i,Y_{j_m(i)}\}) \Biggiven \mX \Big] \Big] \\
  = & \E\Big[\E \Big[\ind(\max\{Y,\tY\} \le \min\{Y_i,Y_{j_m(i)}\}) \Biggiven \mX \Big] \Big] \\
  & - \E\Big[\E \Big[\ind(Y \le \min\{Y_i,Y_{j_m(i)}\}) \Biggiven \mX \Big] \E \Big[ \ind(\tY \le \min\{Y_i,Y_{j_m(i)}\}) \Biggiven \mX \Big] \Big].
\end{align*}

Proceeding in the same way as \eqref{eq:local,taylor}, where we take $U=m$, and analogous to \eqref{eq:local,rate}, we obtain
\[
  \Big\lvert \E\Big[\E \Big[\ind(\max\{Y,\tY\} \le \min\{Y_i,Y_{j_m(i)}\}) \Biggiven \mX \Big] \Big] - \frac{1}{6} \Big\rvert \lesssim r_n,
  \yestag\label{eq:local,var,rate1}
\]
where 1/6 is from the fact that
\[
  \P \Big(\max\{Y_1,Y_2\} \le \min\{Y_3,Y_4\} \Big) = \frac{1}{6},
\]
for four independent copies $Y_1,Y_2,Y_3,Y_4$ from $F_Y$.

Similar to \eqref{eq:local,taylor}, we have
\[
  \E \Big[ \ind(Y \le \min\{Y_1,Y_{j_m(1)}\}) \Biggiven \mX \Big] = \sum_{s=1}^7 T_{m,s},
  \yestag\label{eq:local,Tms}
\]
where $T_{m,s}$ corresponds to $T_s$ in $\eqref{eq:local,taylor}$ for any $s \in \zahl{7}$, while replacing $U$ by $m$.

It is easy to check that $[T_{m,s}]_{s=1}^7$ also satisfies \eqref{eq:local,T1}, \eqref{eq:local,T2}, Lemma~\ref{lemma:local,T3}, Lemma~\ref{lemma:local,T4}, Lemma~\ref{lemma:local,T5}, \eqref{eq:local,T6}, Lemma~\ref{lemma:local,T7} for any $m \in \zahl{M}$.

We then establish the following lemma.
\begin{lemma}\label{lemma:local,T,quad} We have
  \[
    \E \Big[ \Big(\sum_{s=2}^7 T_{m,s})^2 \Big] \lesssim \frac{M}{n} + \frac{M}{n}\sqrt{\log n}\rho_n + \rho_n^2.
  \]
\end{lemma}

From \eqref{eq:local,Tms} and the independence of $Y$ and $\tY$,
\[
  \E\Big[\E \Big[\ind(Y \le \min\{Y_i,Y_{j_m(i)}\}) \Biggiven \mX \Big] \E \Big[ \ind(\tY \le \min\{Y_i,Y_{j_m(i)}\}) \Biggiven \mX \Big] \Big] = \E \Big[ \Big( \sum_{s=1}^7 T_{m,s} \Big)^2 \Big].
\]

Since $T_{m,1} = 1/3$ for any $\mX$, then
\begin{align*}
  & \E\Big[\E \Big[\ind(Y \le \min\{Y_i,Y_{j_m(i)}\}) \Biggiven \mX \Big] \E \Big[ \ind(\tY \le \min\{Y_i,Y_{j_m(i)}\}) \Biggiven \mX \Big] \Big] \\
  = & \frac{1}{9} + \frac{2}{3} \E \Big[ \sum_{s=2}^7 T_{m,s} \Big] + \E \Big[ \Big( \sum_{s=2}^7 T_{m,s} \Big)^2 \Big].
\end{align*}

From \eqref{eq:local,rate} and Lemma~\ref{lemma:local,T,quad},
\[
  \Big\lvert \E\Big[\E \Big[\ind(Y \le \min\{Y_i,Y_{j_m(i)}\}) \Biggiven \mX \Big] \E \Big[ \ind(\tY \le \min\{Y_i,Y_{j_m(i)}\}) \Biggiven \mX \Big] \Big] - \frac{1}{9} \Big\rvert \lesssim r_n.
  \yestag\label{eq:local,var,rate2}
\]

Combine \eqref{eq:local,var,rate1} and \eqref{eq:local,var,rate2},
\[
  \Big\lvert \E\Big[\Cov \Big[\ind(Y \le \min\{Y_i,Y_{j_m(i)}\}),\ind(\tY \le \min\{Y_i,Y_{j_m(i)}\}) \Biggiven \mX \Big] \Big] - \frac{1}{18} \Big\rvert \lesssim r_n.
\]

Then from \eqref{eq:local,var,decompose1},
\[
  \Big\lvert \frac{1}{n^2} \E \Big[ \Var \Big[ \min\{R_i,R_{j_m(i)}\} \Biggiven \mX \Big] \Big] - \frac{1}{18} \Big\rvert \lesssim r_n.
\]

~\\
{\bf The second term.} Analogous to \eqref{eq:local,var,decompose1}, for $m \neq m'$,
\begin{align*}
  & \E \Big[ \Cov \Big[ \min\{R_i,R_{j_m(i)}\}, \min\{R_i,R_{j_{m'}(i)}\} \Biggiven \mX \Big] \Big]\\
  = & \E \Big[ \Cov \Big[ \min\{R_i,R_{j_m(i)}\}, \min\{R_i,R_{j_{m'}(i)}\} \Biggiven \mX \Big] \ind(j_m(i) \neq i) \ind(j_{m'}(i) \neq i) \Big] \\
  & + \E \Big[ \Cov \Big[ \min\{R_i,R_{j_m(i)}\}, \min\{R_i,R_{j_{m'}(i)}\} \Biggiven \mX \Big] \ind(j_m(i) = i) \ind(j_{m'}(i) \neq i)  \Big]\\
  & + \E \Big[ \Cov \Big[ \min\{R_i,R_{j_m(i)}\}, \min\{R_i,R_{j_{m'}(i)}\} \Biggiven \mX \Big] \ind(j_m(i) \neq i) \ind(j_{m'}(i) = i) \Big] \\
  & + \E \Big[ \Cov \Big[ \min\{R_i,R_{j_m(i)}\}, \min\{R_i,R_{j_{m'}(i)}\} \Biggiven \mX \Big] \ind(j_m(i) = i) \ind(j_{m'}(i) = i) \Big]\\
  = & \E \Big[ \Cov \Big[ \sum_{k \neq i, k \neq j_m(i)} \ind(Y_k \le \min\{Y_i,Y_{j_m(i)}\}), \sum_{k \neq i, k \neq j_{m'}(i)} \ind(Y_k \le \min\{Y_i,Y_{j_{m'}(i)}\}) \Biggiven \mX \Big] \\
  & \ind(j_m(i) \neq i) \ind(j_{m'}(i) \neq i) \Big] \\
  & + \E \Big[ \Cov \Big[ \sum_{k \neq i} \ind(Y_k \le \min\{Y_i,Y_{j_m(i)}\}), \sum_{k \neq i, k \neq j_{m'}(i)} \ind(Y_k \le \min\{Y_i,Y_{j_{m'}(i)}\}) \Biggiven \mX \Big]\\
  & \ind(j_m(i) = i) \ind(j_{m'}(i) \neq i)  \Big]\\
  & + \E \Big[ \Cov \Big[ \sum_{k \neq i, k \neq j_m(i)} \ind(Y_k \le \min\{Y_i,Y_{j_m(i)}\}), \sum_{k \neq i} \ind(Y_k \le \min\{Y_i,Y_{j_{m'}(i)}\}) \Biggiven \mX \Big]\\
  & \ind(j_m(i) \neq i) \ind(j_{m'}(i) = i) \Big] \\
  & + \E \Big[ \Cov \Big[ \sum_{k \neq i} \ind(Y_k \le \min\{Y_i,Y_{j_m(i)}\}), \sum_{k \neq i} \ind(Y_k \le \min\{Y_i,Y_{j_{m'}(i)}\}) \Biggiven \mX \Big]\\
  & \ind(j_m(i) = i) \ind(j_{m'}(i) = i) \Big]\\
  = & (n-3)(n-4) \E\Big[\Cov \Big[\ind(Y \le \min\{Y_i,Y_{j_m(i)}\}),\ind(\tY \le \min\{Y_i,Y_{j_{m'}(i)}\}) \Biggiven \mX \Big] \Big] + O(n),
\end{align*}
where $(n-3)(n-4)$ is from the fact that the number of pairs $(k,k')$ such that $k,k' \in \zahl{n}, k \neq k'$ and $k,k' \neq i, j_m(i), j_{m'}(i)$ given events $\{j_m(i) \neq i\}$ and $\{j_{m'}(i) \neq i\}$ is $(n-3)(n-4)$. $O(n)$ corresponds to the first term and the last two terms in \eqref{eq:local,var,decompose1} since the coefficients are of order at most $n$.

Analogous to \eqref{eq:local,var,rate1},
\[
  \Big\lvert \E\Big[\E \Big[\ind(Y \le \min\{Y_i,Y_{j_m(i)}\})\ind(\tY \le \min\{Y_i,Y_{j_{m'}(i)}\}) \Biggiven \mX \Big] \Big] - \frac{2}{15} \Big\rvert \lesssim r_n,
\]
where 2/15 is from the fact that
\[
  \P \Big(Y_1 \le \min\{Y_3,Y_4\}, Y_2 \le \min\{Y_3,Y_5\} \Big) = \frac{2}{15},
\]
for five independent copies $Y_1,Y_2,Y_3,Y_4,Y_5$ from $F_Y$.

Analogous to \eqref{eq:local,var,rate2},
\[
  \Big\lvert \E\Big[\E \Big[\ind(Y \le \min\{Y_i,Y_{j_m(i)}\}) \Biggiven \mX \Big] \E \Big[ \ind(\tY \le \min\{Y_i,Y_{j_{m'}(i)}\}) \Biggiven \mX \Big] \Big] - \frac{1}{9} \Big\rvert \lesssim r_n.
\]

Then
\[
  \Big\lvert \E\Big[\Cov \Big[\ind(Y \le \min\{Y_i,Y_{j_m(i)}\}),\ind(\tY \le \min\{Y_i,Y_{j_{m'}(i)}\}) \Biggiven \mX \Big] \Big] - \frac{1}{45} \Big\rvert \lesssim r_n.
\]

Then for any $m \neq m' \in \zahl{M}$,
\[
  \Big\lvert \frac{1}{n^2} \E \Big[ \Cov \Big[ \min\{R_i,R_{j_m(i)}\}, \min\{R_i,R_{j_{m'}(i)}\} \Biggiven \mX \Big] \Big] - \frac{1}{45} \Big\rvert \lesssim r_n.
\]

~\\
{\bf The third term.} Analogous to \eqref{eq:local,var,decompose1}, for any $i \neq l \in \zahl{n}$ and $m,m' \in \zahl{M}$,
\begin{align*}
  & \E \Big[ \Cov \Big[ \min\{R_i,R_{j_m(i)}\}, \min\{R_l,R_{j_{m'}(l)}\} \Biggiven \mX \Big] \ind\Big(j_m(i)=l \Big) \Big]\\
  = & \E \Big[ \Cov \Big[ \sum_{k \neq i, k \neq j_m(i)} \ind(Y_k \le \min\{Y_i,Y_{j_m(i)}\}), \sum_{k \neq l, k \neq j_{m'}(l)} \ind(Y_k \le \min\{Y_l,Y_{j_{m'}(l)}\}) \Biggiven \mX \Big] \\
  & ~~~\ind\Big(j_m(i)=l \Big)  \ind(j_{m'}(l) \neq l) \Big] \\
  & + \E \Big[ \Cov \Big[ \sum_{k \neq i, k \neq j_m(i)} \ind(Y_k \le \min\{Y_i,Y_{j_m(i)}\}), \sum_{k \neq l} \ind(Y_k \le \min\{Y_l,Y_{j_{m'}(l)}\}) \Biggiven \mX \Big]\\
  & ~~~\ind\Big(j_m(i)=l \Big) \ind(j_{m'}(l) = l) \Big] \\
  = & (n-3)(n-4) \E\Big[\Cov \Big[\ind(Y \le \min\{Y_i,Y_{j_m(i)}\}),\ind(\tY \le \min\{Y_l,Y_{j_{m'}(l)}\}) \Biggiven \mX \Big] \\
  & ~~~\ind\Big(j_m(i)=l \Big)  \ind(j_{m'}(l) \neq l) \Big]\\
  & + (n-2)(n-3) \E\Big[\Cov \Big[\ind(Y \le \min\{Y_i,Y_{j_m(i)}\}),\ind(\tY \le \min\{Y_l,Y_{j_{m'}(l)}\}) \Biggiven \mX \Big] \\
  & ~~~\ind\Big(j_m(i)=l \Big)  \ind(j_{m'}(l) = l) \Big]\\
  & + (n-3)\E\Big[\Cov \Big[\ind(Y \le \min\{Y_i,Y_{j_m(i)}\}),\ind(Y \le \min\{Y_l,Y_{j_{m'}(l)}\}) \Biggiven \mX \Big] \\
  & ~~~\ind\Big(j_m(i)=l \Big)  \ind(j_{m'}(l) \neq l) \Big]\\
  & + (n-3)\E\Big[\Cov \Big[\ind(Y \le \min\{Y_i,Y_{j_m(i)}\}),\ind(Y_i \le \min\{Y_l,Y_{j_{m'}(l)}\}) \Biggiven \mX \Big] \\
  & ~~~\ind\Big(j_m(i)=l \Big)  \ind(j_{m'}(l) \neq l) \Big]\\
  & + (n-3)\E\Big[\Cov \Big[\ind(Y_{j_{m'}(l)} \le \min\{Y_i,Y_{j_m(i)}\}),\ind(Y \le \min\{Y_l,Y_{j_{m'}(l)}\}) \Biggiven \mX \Big] \\
  & ~~~\ind\Big(j_m(i)=l \Big)  \ind(j_{m'}(l) \neq l) \Big]\\
  & + (n-2) \E\Big[\Cov \Big[\ind(Y \le \min\{Y_i,Y_{j_m(i)}\}),\ind(Y \le \min\{Y_l,Y_{j_{m'}(l)}\}) \Biggiven \mX \Big] \\
  & ~~~\ind\Big(j_m(i)=l \Big)  \ind(j_{m'}(l) = l) \Big]\\
  & + (n-2) \E\Big[\Cov \Big[\ind(Y \le \min\{Y_i,Y_{j_m(i)}\}),\ind(Y_i \le \min\{Y_l,Y_{j_{m'}(l)}\}) \Biggiven \mX \Big] \\
  & ~~~\ind\Big(j_m(i)=l \Big)  \ind(j_{m'}(l) = l) \Big] + O(1)\\
  = & (n-3)(n-4) \E\Big[\Cov \Big[\ind(Y \le \min\{Y_i,Y_{j_m(i)}\}),\ind(\tY \le \min\{Y_l,Y_{j_{m'}(l)}\}) \Biggiven \mX \Big] \\
  & ~~~\ind\Big(j_m(i)=l \Big)  \ind(j_{m'}(l) \neq l) \Big]\\
  & + (n-2)(n-3) \E\Big[\Cov \Big[\ind(Y \le \min\{Y_i,Y_{j_m(i)}\}),\ind(\tY \le \min\{Y_l,Y_{j_{m'}(l)}\}) \Biggiven \mX \Big] \\
  & ~~~\ind\Big(j_m(i)=l \Big)  \ind(j_{m'}(l) = l) \Big] + O(1).\\
  = & (n-3)(n-4) \E\Big[\Cov \Big[\ind(Y \le \min\{Y_i,Y_{j_m(i)}\}),\ind(\tY \le \min\{Y_l,Y_{j_{m'}(l)}\}) \Biggiven \mX \Big] \ind\Big(j_m(i)=l \Big) \Big] + O(1).
\end{align*}
In the above derivation the term $(n-3)(n-4)$ is due to that the number of pairs $(k,k')$ such that $k,k' \in \zahl{n}, k \neq k'$ and $k,k' \neq i, j_m(i), l, j_{m'}(l)$ given events $\{j_m(i)=l\} \cup \{j_{m'}(l) \neq l\}$ is $(n-3)(n-4)$. The term $(n-2)(n-3)$ is calculated in the same way. The last two steps are from the fact that $\P(j_m(i)=l)$ is of order $n^{-1}$.

Since $l$ is an arbitrary index such that $l \neq i$, then
\begin{align*}
  & \E\Big[\Cov \Big[\ind(Y \le \min\{Y_i,Y_{j_m(i)}\}),\ind(\tY \le \min\{Y_l,Y_{j_{m'}(l)}\}) \Biggiven \mX \Big] \ind\Big(j_m(i)=l \Big) \Big]\\
  = & \E\Big[\Cov \Big[\ind(Y \le \min\{Y_i,Y_{j_m(i)}\}),\ind(\tY \le \min\{Y_{j_m(i)},Y_{j_{m'}(j_m(i))}\}) \Biggiven \mX \Big] \ind\Big(j_m(i)=l \Big) \Big]\\
  = & \frac{1}{n-1} \E\Big[\Cov \Big[\ind(Y \le \min\{Y_i,Y_{j_m(i)}\}),\ind(\tY \le \min\{Y_{j_m(i)},Y_{j_{m'}(j_m(i))}\}) \Biggiven \mX \Big] \ind\Big(j_m(i) \neq i \Big) \Big]\\
  = & \frac{1}{n-1} \E\Big[\Cov \Big[\ind(Y \le \min\{Y_i,Y_{j_m(i)}\}),\ind(\tY \le \min\{Y_{j_m(i)},Y_{j_{m'}(j_m(i))}\}) \Biggiven \mX \Big]\Big] + O\Big(\frac{M}{n^2}\Big),
\end{align*}
since $\P(j_m(i)=i)$ is of order $M/n$.

~\\
{\bf The fourth term.} This is the same as the third term and accordingly omitted.

~\\
{\bf The fifth term.} Analogous to \eqref{eq:local,var,rate1},
\[
  \Big\lvert \E\Big[\E \Big[\ind(Y \le \min\{Y_i,Y_{j_m(i)}\})\ind(\tY \le \min\{Y_{j_m(i)},Y_{j_{m'}(j_m(i))}\}) \Biggiven \mX \Big] \Big] - \frac{2}{15} \Big\rvert \lesssim r_n,
\]
where 2/15 is from the fact that
\[
  \P \Big(Y_1 \le \min\{Y_3,Y_4\}, Y_2 \le \min\{Y_3,Y_5\} \Big) = \frac{2}{15},
\]
for five independent copies $Y_1,Y_2,Y_3,Y_4,Y_5$ from $F_Y$.

Analogous to \eqref{eq:local,var,rate2},
\[
  \Big\lvert \E\Big[\E \Big[\ind(Y \le \min\{Y_i,Y_{j_m(i)}\}) \Biggiven \mX \Big] \E \Big[ \ind(\tY \le \min\{Y_{j_m(i)},Y_{j_{m'}(j_m(i))}\}) \Biggiven \mX \Big] \Big] - \frac{1}{9} \Big\rvert \lesssim r_n.
\]

Then
\[
  \Big\lvert \E\Big[\Cov \Big[\ind(Y \le \min\{Y_i,Y_{j_m(i)}\}),\ind(\tY \le \min\{Y_{j_m(i)},Y_{j_{m'}(j_m(i))}\}) \Biggiven \mX \Big]\Big] - \frac{1}{45} \Big\rvert \lesssim r_n.
\]

Then for any $i \neq l \in \zahl{n}$,
\[
  \Big\lvert \frac{1}{n} \E \Big[ \Cov \Big[ \min\{R_i,R_{j_m(i)}\}, \min\{R_l,R_{j_{m'}(l)}\} \Biggiven \mX \Big] \ind\Big(j_m(i)=l \Big) \Big] - \frac{1}{45} \Big\rvert \lesssim r_n.
\]

We can establish in the same way that
\[
  \Big\lvert \frac{1}{n} \E \Big[ \Cov \Big[ \min\{R_i,R_{j_m(i)}\}, \min\{R_l,R_{j_{m'}(l)}\} \Biggiven \mX \Big] \ind\Big(j_{m'}(l)=i \Big) \Big] - \frac{1}{45} \Big\rvert \lesssim r_n,
\]
and
\[
  \Big\lvert \frac{1}{n} \E \Big[ \Cov \Big[ \min\{R_i,R_{j_m(i)}\}, \min\{R_l,R_{j_{m'}(l)}\} \Biggiven \mX \Big] \ind\Big(j_m(i)=j_{m'}(l) \Big) \Big] - \frac{1}{45} \Big\rvert \lesssim r_n.
\]

~\\
{\bf The sixth term.} Analogous to \eqref{eq:local,var,decompose1},
\begin{align*}
  & \E \Big[ \Cov \Big[ \min\{R_i,R_{j_m(i)}\}, \min\{R_l,R_{j_{m'}(l)}\} \Biggiven \mX \Big] \ind\Big(j_m(i) \neq l, j_{m'}(l) \neq i, j_m(i) \neq j_{m'}(l) \Big) \Big]\\
  = & \E \Big[ \Cov \Big[ \sum_{k \neq i, k \neq j_m(i)} \ind(Y_k \le \min\{Y_i,Y_{j_m(i)}\}), \sum_{k \neq l, k \neq j_{m'}(l)} \ind(Y_k \le \min\{Y_l,Y_{j_{m'}(l)}\}) \Biggiven \mX \Big] \\
  & \ind\Big(j_m(i) \neq l, j_{m'}(l) \neq i, j_m(i) \neq j_{m'}(l) \Big) \ind(j_m(i) \neq i) \ind(j_{m'}(l) \neq l) \Big] \\
  & + \E \Big[ \Cov \Big[ \sum_{k \neq i} \ind(Y_k \le \min\{Y_i,Y_{j_m(i)}\}), \sum_{k \neq l, k \neq j_{m'}(l)} \ind(Y_k \le \min\{Y_l,Y_{j_{m'}(l)}\}) \Biggiven \mX \Big]\\
  & \ind\Big(j_m(i) \neq l, j_{m'}(l) \neq i, j_m(i) \neq j_{m'}(l) \Big) \ind(j_m(i) = i) \ind(j_{m'}(l) \neq l)  \Big]\\
  & + \E \Big[ \Cov \Big[ \sum_{k \neq i, k \neq j_m(i)} \ind(Y_k \le \min\{Y_i,Y_{j_m(i)}\}), \sum_{k \neq l} \ind(Y_k \le \min\{Y_l,Y_{j_{m'}(l)}\}) \Biggiven \mX \Big]\\
  & \ind\Big(j_m(i) \neq l, j_{m'}(l) \neq i, j_m(i) \neq j_{m'}(l) \Big) \ind(j_m(i) \neq i) \ind(j_{m'}(l) = l) \Big] \\
  & + \E \Big[ \Cov \Big[ \sum_{k \neq i} \ind(Y_k \le \min\{Y_i,Y_{j_m(i)}\}), \sum_{k \neq l} \ind(Y_k \le \min\{Y_l,Y_{j_{m'}(l)}\}) \Biggiven \mX \Big]\\
  & \ind\Big(j_m(i) \neq l, j_{m'}(l) \neq i, j_m(i) \neq j_{m'}(l) \Big) \ind(j_m(i) = i) \ind(j_{m'}(l) = l) \Big]\\
  = & (n-4) \E \Big[ \Cov \Big[ \ind(Y \le \min\{Y_i,Y_{j_m(i)}\}), \ind(Y \le \min\{Y_l,Y_{j_{m'}(l)}\}) \Biggiven \mX \Big]\\
  & \ind\Big(j_m(i) \neq l, j_{m'}(l) \neq i, j_m(i) \neq j_{m'}(l) \Big) \Big] \\
  & + (n-4) \E \Big[ \Cov \Big[ \ind(Y \le \min\{Y_i,Y_{j_m(i)}\}), \ind(Y_i \le \min\{Y_l,Y_{j_{m'}(l)}\}) \Biggiven \mX \Big]\\
  & \ind\Big(j_m(i) \neq l, j_{m'}(l) \neq i, j_m(i) \neq j_{m'}(l) \Big) \Big] \\
  & + (n-4) \E \Big[ \Cov \Big[ \ind(Y \le \min\{Y_i,Y_{j_m(i)}\}), \ind(Y_{j_m(i)} \le \min\{Y_l,Y_{j_{m'}(l)}\}) \Biggiven \mX \Big]\\
  & \ind\Big(j_m(i) \neq l, j_{m'}(l) \neq i, j_m(i) \neq j_{m'}(l) \Big) \Big] \\
  & + (n-4) \E \Big[ \Cov \Big[ \ind(Y_l \le \min\{Y_i,Y_{j_m(i)}\}), \ind(Y \le \min\{Y_l,Y_{j_{m'}(l)}\}) \Biggiven \mX \Big]\\
  & \ind\Big(j_m(i) \neq l, j_{m'}(l) \neq i, j_m(i) \neq j_{m'}(l) \Big) \Big] \\
  & + (n-4) \E \Big[ \Cov \Big[ \ind(Y_{j_{m'}(l)} \le \min\{Y_i,Y_{j_m(i)}\}), \ind(Y \le \min\{Y_l,Y_{j_{m'}(l)}\}) \Biggiven \mX \Big]\\
  & \ind\Big(j_m(i) \neq l, j_{m'}(l) \neq i, j_m(i) \neq j_{m'}(l) \Big) \Big] \\
  & + (n-4)(n-5) \E\Big[\Cov \Big[\ind(Y \le \min\{Y_i,Y_{j_m(i)}\}),\ind(\tY \le \min\{Y_l,Y_{j_{m'}(l)}\}) \Biggiven \mX \Big] \\
  & \ind\Big(j_m(i) \neq l, j_{m'}(l) \neq i, j_m(i) \neq j_{m'}(l) \Big)\Big]\\
  & + n(1+o(1)) \E\Big[\Cov \Big[\ind(Y \le \min\{Y_i,Y_{j_m(i)}\}),\ind(\tY \le \min\{Y_l,Y_{j_{m'}(l)}\}) \Biggiven \mX \Big] \\
  & \ind\Big(j_m(i) \neq l, j_{m'}(l) \neq i, j_m(i) \neq j_{m'}(l) \Big)\ind(j_m(i) = i) \ind(j_{m'}(l) \neq l)\Big]\\
  & + n(1+o(1)) \E\Big[\Cov \Big[\ind(Y \le \min\{Y_i,Y_{j_m(i)}\}),\ind(\tY \le \min\{Y_l,Y_{j_{m'}(l)}\}) \Biggiven \mX \Big] \\
  & \ind\Big(j_m(i) \neq l, j_{m'}(l) \neq i, j_m(i) \neq j_{m'}(l) \Big)\ind(j_m(i) \neq i) \ind(j_{m'}(l) = l)\Big]\\
  & + n(1+o(1)) \E\Big[\Cov \Big[\ind(Y \le \min\{Y_i,Y_{j_m(i)}\}),\ind(\tY \le \min\{Y_l,Y_{j_{m'}(l)}\}) \Biggiven \mX \Big] \\
  & \ind\Big(j_m(i) \neq l, j_{m'}(l) \neq i, j_m(i) \neq j_{m'}(l) \Big)\ind(j_m(i) = i) \ind(j_{m'}(l) = l)\Big] + O(1)\\
  = & (n-4) \E \Big[ \Cov \Big[ \ind(Y \le \min\{Y_i,Y_{j_m(i)}\}), \ind(Y \le \min\{Y_l,Y_{j_{m'}(l)}\}) \Biggiven \mX \Big]\\
  & \ind\Big(j_m(i) \neq l, j_{m'}(l) \neq i, j_m(i) \neq j_{m'}(l) \Big) \Big] \\
  & + (n-4) \E \Big[ \Cov \Big[ \ind(Y \le \min\{Y_i,Y_{j_m(i)}\}), \ind(Y_i \le \min\{Y_l,Y_{j_{m'}(l)}\}) \Biggiven \mX \Big]\\
  & \ind\Big(j_m(i) \neq l, j_{m'}(l) \neq i, j_m(i) \neq j_{m'}(l) \Big) \Big] \\
  & + (n-4) \E \Big[ \Cov \Big[ \ind(Y \le \min\{Y_i,Y_{j_m(i)}\}), \ind(Y_{j_m(i)} \le \min\{Y_l,Y_{j_{m'}(l)}\}) \Biggiven \mX \Big]\\
  & \ind\Big(j_m(i) \neq l, j_{m'}(l) \neq i, j_m(i) \neq j_{m'}(l) \Big) \Big] \\
  & + (n-4) \E \Big[ \Cov \Big[ \ind(Y_l \le \min\{Y_i,Y_{j_m(i)}\}), \ind(Y \le \min\{Y_l,Y_{j_{m'}(l)}\}) \Biggiven \mX \Big]\\
  & \ind\Big(j_m(i) \neq l, j_{m'}(l) \neq i, j_m(i) \neq j_{m'}(l) \Big) \Big] \\
  & + (n-4) \E \Big[ \Cov \Big[ \ind(Y_{j_{m'}(l)} \le \min\{Y_i,Y_{j_m(i)}\}), \ind(Y \le \min\{Y_l,Y_{j_{m'}(l)}\}) \Biggiven \mX \Big]\\
  & \ind\Big(j_m(i) \neq l, j_{m'}(l) \neq i, j_m(i) \neq j_{m'}(l) \Big) \Big] + O(1),
  \yestag\label{eq:local,var,decompose4}
\end{align*}
where the first $n-4$ is from the fact that the number of $k \in \zahl{n}$ such that $k \neq i, j_m(i), l, j_{m'}(l)$ given events $\{j_m(i) \neq l, j_{m'}(l) \neq i, j_m(i) \neq j_{m'}(l)\}$ is $n-4$. Other four are in the same way. The last equation is from the fact that $Y,Y_i,Y_{j_m(i)},\tY,Y_l,Y_{j_{m'}(l)}$ are mutually independent given events $\{j_m(i) \neq l, j_{m'}(l) \neq i, j_m(i) \neq j_{m'}(l)\}$.

We first consider the first term in \eqref{eq:local,var,decompose4}. Since $\P(\{j_m(i) = l\} \cup \{j_{m'}(l) = i\} \cup \{j_m(i) = j_{m'}(l)\})$ is of order $n^{-1}$, then
\begin{align*}
  & \E \Big[ \Cov \Big[ \ind(Y \le \min\{Y_i,Y_{j_m(i)}\}), \ind(Y \le \min\{Y_l,Y_{j_{m'}(l)}\}) \Biggiven \mX \Big] \ind\Big(j_m(i) \neq l, j_{m'}(l) \neq i, j_m(i) \neq j_{m'}(l) \Big) \Big]\\
  = & \E \Big[ \Cov \Big[ \ind(Y \le \min\{Y_i,Y_{j_m(i)}\}), \ind(Y \le \min\{Y_l,Y_{j_{m'}(l)}\}) \Biggiven \mX \Big] \Big] + O\Big( \frac{1}{n} \Big).
\end{align*}

Analogous to \eqref{eq:local,var,rate1},
\[
  \Big\lvert \E\Big[\E \Big[\ind(Y \le \min\{Y_i,Y_{j_m(i)}\})\ind(Y \le \min\{Y_l,Y_{j_{m'}(l)}\}) \Biggiven \mX \Big] \Big] - \frac{1}{5} \Big\rvert \lesssim r_n,
\]
where 1/5 is from the fact that
\[
  \P \Big(Y_1 \le \min\{Y_2,Y_3\}, Y_1 \le \min\{Y_4,Y_5\} \Big) = \frac{1}{5},
\]
for five independent copies $Y_1,Y_2,Y_3,Y_4,Y_5$ from $F_Y$.

Analogous to \eqref{eq:local,var,rate2},
\[
  \Big\lvert \E\Big[\E \Big[\ind(Y \le \min\{Y_i,Y_{j_m(i)}\}) \Biggiven \mX \Big] \E \Big[ \ind(Y \le \min\{Y_l,Y_{j_{m'}(l)}\}) \Biggiven \mX \Big] \Big] - \frac{1}{9} \Big\rvert \lesssim r_n.
\]

Then
\[
  \Big\lvert \E \Big[ \Cov \Big[ \ind(Y \le \min\{Y_i,Y_{j_m(i)}\}), \ind(Y \le \min\{Y_l,Y_{j_{m'}(l)}\}) \Biggiven \mX \Big] \Big]- \frac{4}{45} \Big\rvert \lesssim r_n.
\]

For the second term in \eqref{eq:local,var,decompose4}, we also have
\begin{align*}
  & \E \Big[ \Cov \Big[ \ind(Y \le \min\{Y_i,Y_{j_m(i)}\}), \ind(Y_i \le \min\{Y_l,Y_{j_{m'}(l)}\}) \Biggiven \mX \Big]\\
  & \ind\Big(j_m(i) \neq l, j_{m'}(l) \neq i, j_m(i) \neq j_{m'}(l) \Big) \Big]\\
  = & \E \Big[ \Cov \Big[ \ind(Y \le \min\{Y_i,Y_{j_m(i)}\}), \ind(Y_i \le \min\{Y_l,Y_{j_{m'}(l)}\}) \Biggiven \mX \Big] \Big] + O\Big( \frac{1}{n} \Big).
\end{align*}

Analogous to \eqref{eq:local,var,rate1},
\[
  \Big\lvert \E\Big[\E \Big[\ind(Y \le \min\{Y_i,Y_{j_m(i)}\})\ind(Y_i \le \min\{Y_l,Y_{j_{m'}(l)}\}) \Biggiven \mX \Big] \Big] - \frac{1}{15} \Big\rvert \lesssim r_n,
\]
where 1/15 is from the fact that
\[
  \P \Big(Y_1 \le \min\{Y_2,Y_3\}, Y_2 \le \min\{Y_4,Y_5\} \Big) = \frac{1}{15},
\]
for five independent copies $Y_1,Y_2,Y_3,Y_4,Y_5$ from $F_Y$.

Analogous to \eqref{eq:local,var,rate2},
\[
  \Big\lvert \E\Big[\E \Big[\ind(Y \le \min\{Y_i,Y_{j_m(i)}\}) \Biggiven \mX \Big] \E \Big[ \ind(Y_i \le \min\{Y_l,Y_{j_{m'}(l)}\}) \Biggiven \mX \Big] \Big] - \frac{1}{9} \Big\rvert \lesssim r_n.
\]

Then
\[
  \Big\lvert \E \Big[ \Cov \Big[ \ind(Y \le \min\{Y_i,Y_{j_m(i)}\}), \ind(Y_i \le \min\{Y_l,Y_{j_{m'}(l)}\}) \Biggiven \mX \Big] \Big] + \frac{2}{45} \Big\rvert \lesssim r_n.
\]

We can establish in the same way that
\begin{align*}
  & \Big\lvert \E \Big[ \Cov \Big[ \ind(Y \le \min\{Y_i,Y_{j_m(i)}\}), \ind(Y_i \le \min\{Y_l,Y_{j_{m'}(l)}\}) \Biggiven \mX \Big] \Big] + \frac{2}{45} \Big\rvert \lesssim r_n\\
  & \Big\lvert \E \Big[ \Cov \Big[ \ind(Y \le \min\{Y_i,Y_{j_m(i)}\}), \ind(Y_i \le \min\{Y_l,Y_{j_{m'}(l)}\}) \Biggiven \mX \Big] \Big] + \frac{2}{45} \Big\rvert \lesssim r_n\\
  & \Big\lvert \E \Big[ \Cov \Big[ \ind(Y \le \min\{Y_i,Y_{j_m(i)}\}), \ind(Y_i \le \min\{Y_l,Y_{j_{m'}(l)}\}) \Biggiven \mX \Big] \Big] + \frac{2}{45} \Big\rvert \lesssim r_n.
\end{align*}

Then from \eqref{eq:local,var,decompose4}, for any $i \neq l \in \zahl{n}$,
\begin{align*}
  & \Big\lvert \frac{1}{n} \E \Big[ \Cov \Big[ \min\{R_i,R_{j_m(i)}\}, \min\{R_l,R_{j_{m'}(l)}\} \Biggiven \mX \Big] \ind\Big(j_m(i) \neq l, j_{m'}(l) \neq i, j_m(i) \neq j_{m'}(l) \Big) \Big]\\
  & + \frac{4}{45} \Big\rvert \lesssim r_n.
\end{align*}
This completes the proof.
\end{proof}

\subsubsection{Proof of Lemma~\ref{lemma:alter,var,decompose2}}

\begin{proof}[Proof of Lemma~\ref{lemma:alter,var,decompose2}]
From \eqref{eq:xin},
\begin{align*}
  \Var \Big[\E\Big[\xi_{n,M} \Biggiven \mX\Big] \Big] &= \Var \Big[ \E \Big[-2 + \frac{6 \sum_{i=1}^{n} \sum_{m=1}^M \min\{R_{j_m(i)},R_i\} }{(n+1)[nM+M(M+1)/4]} \Biggiven \mX\Big] \Big]\\
  &= \frac{36}{(n+1)^2[nM+M(M+1)/4]^2} \Var \Big[ \E\Big[ \sum_{i=1}^{n} \sum_{m=1}^M \min\{R_i,R_{j_m(i)}\} \Biggiven \mX\Big] \Big].
\end{align*}

Since for any $i \in \zahl{n},m \in \zahl{M}$,
\[
  \min\{R_i,R_{j_m(i)}\} = \sum_{k=1}^n \ind(Y_k \le \min\{Y_i,Y_{j_m(i)}\}),
\]
we obtain
\begin{align*}
  & \Var \Big[ \E\Big[ \sum_{i=1}^{n} \sum_{m=1}^M \min\{R_i,R_{j_m(i)}\} \Biggiven \mX\Big] \Big]
  =  \Var \Big[ \E\Big[ \sum_{i=1}^{n} \sum_{m=1}^M \sum_{k=1}^n \ind(Y_k \le \min\{Y_i,Y_{j_m(i)}\}) \Biggiven \mX\Big] \Big].
\end{align*}

To apply Efron-Stein inequality (Theorem 3.1 in \cite{boucheron2013concentration}), we consider $\mX = (X_1,\ldots,X_n)$ and for any $\ell \in \zahl{n}$,
\[
  \mX_\ell = (X_1,\ldots,X_{\ell-1},\tX_\ell,X_{\ell+1},\ldots,X_n),
\]
where $\{\tX_\ell\}_{\ell=1}^n$ are independent copies of $\{X_\ell\}_{\ell=1}^n$.

We fix $\ell \in \zahl{n}$. For any $i \in \zahl{n}$ and $m \in \zahl{M}$, denote the $m$th right nearest neighbor of $i$ in samples $\mX$ by $j_m(i)$ and in samples $\mX_\ell$ by $\tilde{j}_m(i)$. Define $A_M(i) := \{j:j_m(j) = i{~\rm for~some~}m \in \zahl{M}\}$ in samples $\mX$ and $\tilde{A}_M(i) := \{j:\tilde{j}_m(j) = i{~\rm for~some~}m \in \zahl{M}\}$ in samples $\mX_\ell$.

For any $i \in \zahl{n}$, we consider
\[
  \E\Big[ \sum_{m=1}^M \sum_{k=1}^n \ind(Y_k \le \min\{Y_i,Y_{j_m(i)}\}) \Biggiven \mX\Big] - \E\Big[ \sum_{m=1}^M \sum_{k=1}^n \ind(Y_k \le \min\{Y_i,Y_{\tilde{j}_m(i)}\}) \Biggiven \mX_\ell \Big].
\]

If $i=\ell$, then
\begin{align*}
  & \E\Big[ \sum_{m=1}^M \sum_{k=1}^n \ind(Y_k \le \min\{Y_i,Y_{j_m(i)}\}) \Biggiven \mX\Big] - \E\Big[ \sum_{m=1}^M \sum_{k=1}^n \ind(Y_k \le \min\{Y_i,Y_{\tilde{j}_m(i)}\}) \Biggiven \mX_\ell \Big]\\
  = & \E\Big[ \sum_{m=1}^M \sum_{k=1}^n \ind(Y_k \le \min\{Y_\ell,Y_{j_m(\ell)}\}) \Biggiven \mX\Big] - \E\Big[ \sum_{m=1}^M \sum_{k=1}^n \ind(Y_k \le \min\{Y_\ell,Y_{\tilde{j}_m(\ell)}\}) \Biggiven \mX_\ell \Big].\\
\end{align*}

If $i \in A_M(\ell) \setminus \tilde{A}_M(\ell)$ and $i \neq l$, then $\ell$ is one of the $M$ right nearest neighbors of $i$ in samples $\mX$, and is not one of the $M$ right nearest neighbors of $i$ when we replace $X_\ell$ by $\tX_\ell$. Since the values of other data points except for the $\ell$-th remain unchanged, the change in the $M$ right nearest neighbors of $i$ is simply removing $\ell$ and introducing new $\tilde{j}_M(\ell)$. Also notice that the conditional distribution of $Y_\ell$ conditional on $\mX$ and that conditional on $\mX_\ell$ is different. Then
\begin{align*}
  & \E\Big[ \sum_{m=1}^M \sum_{k=1}^n \ind(Y_k \le \min\{Y_i,Y_{j_m(i)}\}) \Biggiven \mX\Big] - \E\Big[ \sum_{m=1}^M \sum_{k=1}^n \ind(Y_k \le \min\{Y_i,Y_{\tilde{j}_m(i)}\}) \Biggiven \mX_\ell \Big]\\
  = & \E\Big[ \sum_{k=1,k \neq \ell}^n \ind(Y_k \le \min\{Y_i,Y_\ell\}) \Biggiven \mX\Big] - \E\Big[ \sum_{k=1,k \neq \ell}^n \ind(Y_k \le \min\{Y_i,Y_{\tilde{j}_M(\ell)}\}) \Biggiven \mX_\ell \Big]\\
  & + \E\Big[ \sum_{m=1}^M \ind(Y_\ell \le \min\{Y_i,Y_{j_m(i)}\}) \Biggiven \mX\Big] - \E\Big[ \sum_{m=1}^M \ind(Y_\ell \le \min\{Y_i,Y_{\tilde{j}_m(i)}\}) \Biggiven \mX_\ell \Big].
\end{align*}

If $i \in \tilde{A}_M(\ell) \setminus A_M(\ell)$ and $i \neq l$, then in the same way,
\begin{align*}
  & \E\Big[ \sum_{m=1}^M \sum_{k=1}^n \ind(Y_k \le \min\{Y_i,Y_{j_m(i)}\}) \Biggiven \mX\Big] - \E\Big[ \sum_{m=1}^M \sum_{k=1}^n \ind(Y_k \le \min\{Y_i,Y_{\tilde{j}_m(i)}\}) \Biggiven \mX_\ell \Big]\\
  = & \E\Big[ \sum_{k=1,k \neq \ell}^n \ind(Y_k \le \min\{Y_i,Y_{j_M(i)}\}) \Biggiven \mX\Big] - \E\Big[ \sum_{k=1,k \neq \ell}^n \ind(Y_k \le \min\{Y_i,Y_\ell\}) \Biggiven \mX_\ell \Big]\\
  & + \E\Big[ \sum_{m=1}^M \ind(Y_\ell \le \min\{Y_i,Y_{j_m(i)}\}) \Biggiven \mX\Big] - \E\Big[ \sum_{m=1}^M \ind(Y_\ell \le \min\{Y_i,Y_{\tilde{j}_m(i)}\}) \Biggiven \mX_\ell \Big].
\end{align*}

If $i \in \tilde{A}_M(\ell) \cap A_M(\ell)$ and $i \neq \ell$, then the $M$ right nearest neighbors of $i$ remain unchanged but the conditional distribution of $Y_\ell$ changes. Then
\begin{align*}
  & \E\Big[ \sum_{m=1}^M \sum_{k=1}^n \ind(Y_k \le \min\{Y_i,Y_{j_m(i)}\}) \Biggiven \mX\Big] - \E\Big[ \sum_{m=1}^M \sum_{k=1}^n \ind(Y_k \le \min\{Y_i,Y_{\tilde{j}_m(i)}\}) \Biggiven \mX_\ell \Big]\\
  = & \E\Big[ \sum_{k=1,k \neq \ell}^n \ind(Y_k \le \min\{Y_i,Y_\ell\}) \Biggiven \mX\Big] - \E\Big[ \sum_{k=1,k \neq \ell}^n \ind(Y_k \le \min\{Y_i,Y_\ell\}) \Biggiven \mX_\ell \Big]\\
  & + \E\Big[ \sum_{m=1}^M \ind(Y_\ell \le \min\{Y_i,Y_{j_m(i)}\}) \Biggiven \mX\Big] - \E\Big[ \sum_{m=1}^M \ind(Y_\ell \le \min\{Y_i,Y_{\tilde{j}_m(i)}\}) \Biggiven \mX_\ell \Big].
\end{align*}

If $i \notin \tilde{A}_M(\ell) \cup A_M(\ell)$ and $i \neq \ell$, then the $M$ right nearest neighbors of $i$ remain unchanged and do not contain $\ell$. Then
\begin{align*}
  & \E\Big[ \sum_{m=1}^M \sum_{k=1}^n \ind(Y_k \le \min\{Y_i,Y_{j_m(i)}\}) \Biggiven \mX\Big] - \E\Big[ \sum_{m=1}^M \sum_{k=1}^n \ind(Y_k \le \min\{Y_i,Y_{\tilde{j}_m(i)}\}) \Biggiven \mX_\ell \Big]\\
  = & \E\Big[ \sum_{m=1}^M \ind(Y_\ell \le \min\{Y_i,Y_{j_m(i)}\}) \Biggiven \mX\Big] - \E\Big[ \sum_{m=1}^M \ind(Y_\ell \le \min\{Y_i,Y_{\tilde{j}_m(i)}\}) \Biggiven \mX_\ell \Big].
\end{align*}

Let
\begin{align*}
  &T = \E\Big[ \sum_{i=1}^{n} \sum_{m=1}^M \sum_{k=1}^n \ind(Y_k \le \min\{Y_i,Y_{j_m(i)}\}) \Biggiven \mX\Big]\\
~~~{\rm and}~~&  T^\ell = \E\Big[ \sum_{i=1}^{n} \sum_{m=1}^M \sum_{k=1}^n \ind(Y_k \le \min\{Y_i,Y_{\tilde{j}_m(i)}\}) \Biggiven \mX_\ell \Big].
\end{align*}
Then
\begin{align*}
  T - T^\ell =& \E\Big[ \sum_{m=1}^M \sum_{k=1}^n \ind(Y_k \le \min\{Y_\ell,Y_{j_m(\ell)}\}) \Biggiven \mX\Big] - \E\Big[ \sum_{m=1}^M \sum_{k=1}^n \ind(Y_k \le \min\{Y_\ell,Y_{\tilde{j}_m(\ell)}\}) \Biggiven \mX_\ell \Big]\\
  & + \E\Big[ \sum_{\substack{i=1,i \neq l \\ i \in A_M(\ell) \setminus \tilde{A}_M(\ell)}}^n \sum_{k=1,k \neq \ell}^n \ind(Y_k \le \min\{Y_i,Y_\ell\}) \Biggiven \mX\Big] \\
  & - \E\Big[ \sum_{\substack{i=1,i \neq l \\ i \in A_M(\ell) \setminus \tilde{A}_M(\ell)}}^n \sum_{k=1,k \neq \ell}^n \ind(Y_k \le \min\{Y_i,Y_{\tilde{j}_M(\ell)}\}) \Biggiven \mX_\ell \Big]\\
  & + \E\Big[ \sum_{\substack{i=1,i \neq l \\ i \in \tilde{A}_M(\ell) \setminus A_M(\ell)}}^n \sum_{k=1,k \neq \ell}^n \ind(Y_k \le \min\{Y_i,Y_{j_M(i)}\}) \Biggiven \mX\Big] \\
  & - \E\Big[ \sum_{\substack{i=1,i \neq l \\ i \in \tilde{A}_M(\ell) \setminus A_M(\ell)}}^n \sum_{k=1,k \neq \ell}^n \ind(Y_k \le \min\{Y_i,Y_\ell\}) \Biggiven \mX_\ell \Big]\\
  & + \E\Big[ \sum_{\substack{i=1,i \neq l \\ i \in \tilde{A}_M(\ell) \cap A_M(\ell)}}^n \sum_{k=1,k \neq \ell}^n \ind(Y_k \le \min\{Y_i,Y_\ell\}) \Biggiven \mX\Big] \\
  & - \E\Big[ \sum_{\substack{i=1,i \neq l \\ i \in \tilde{A}_M(\ell) \cap A_M(\ell)}}^n \sum_{k=1,k \neq \ell}^n \ind(Y_k \le \min\{Y_i,Y_\ell\}) \Biggiven \mX_\ell \Big]\\
  & + \E\Big[ \sum_{i=1,i \neq l}^n \sum_{m=1}^M \ind(Y_\ell \le \min\{Y_i,Y_{j_m(i)}\}) \Biggiven \mX\Big] \\
  & - \E\Big[ \sum_{i=1,i \neq l}^n \sum_{m=1}^M \ind(Y_\ell \le \min\{Y_i,Y_{\tilde{j}_m(i)}\}) \Biggiven \mX_\ell \Big].
\end{align*}

From the definition of $A_M(\ell)$ and $\tilde{A}_M(\ell)$, for any given $\mX$ and $\mX_\ell$, we have $\lvert A_M(\ell) \setminus \{\ell\} \rvert \le M$ and $\lvert \tilde{A}_M(\ell) \setminus \{\ell\} \rvert \le M$ since $\ell$ can be the $M$ right nearest neighbors of at most $M$ units. Then for each term above, the number of indicator functions in the sum is of order $nM$. Then after subtracting 1/3 for each indicator function, we apply Lemma~\ref{lemma:local,T,quad} and Cauchy–Schwarz inequality, and then
\[
  \E [(T - T^\ell)^2] \lesssim n^2M^2 \Big( \frac{M}{n} + \frac{M}{n}\sqrt{\log n}\rho_n + \rho_n^2 \Big),
\]
where the constant on the righthand side does not depend on $\ell$.

Then from Efron-Stein inequality,
\begin{align*}
  \Var \Big[ \E\Big[ \sum_{i=1}^{n} \sum_{m=1}^M \min\{R_i,R_{j_m(i)}\} \Biggiven \mX\Big] \Big] & \lesssim \sum_{\ell=1}^n \E [(T - T^\ell)^2] \\
  & \lesssim n^3M^2 \Big( \frac{M}{n} + \frac{M}{n}\sqrt{\log n}\rho_n + \rho_n^2 \Big).
\end{align*}

Then we obtain
\[
  \Var \Big[\E\Big[\xi_{n,M} \Biggiven \mX\Big] \Big] \lesssim \frac{1}{n} \Big( \frac{M}{n} + \frac{M}{n}\sqrt{\log n}\rho_n + \rho_n^2 \Big)
\]
and thus finish the proof.
\end{proof}

\subsection{Proofs of results in the appendix}

\subsubsection{Proof of Lemma~\ref{lemma:perm1}}

\begin{proof}[Proof of Lemma~\ref{lemma:perm1}]

In the following proof we remove the superscript $U$ for notation simplicity. Notice that for any $i,j \in \zahl{n}$ and $i \neq j$,
\[
  R_i = 1 + \sum_{k \neq i} \ind(U_k \le U_i),~ \min\{R_i,R_j\} = 1 + \sum_{k \neq i, k \neq j} \ind(U_k \le \min\{U_i,U_j\}).
\]

From simple calculation,
\begin{align*}
  & \Cov [\ind(U_2 \le U_1),U_1] = \frac{1}{12},~ \Cov [\ind(U_2 \le U_1),U_2] = -\frac{1}{12},~ \Cov [\ind(U_2 \le U_1),\min\{U_1,U_2\}] = 0,\\
  & \Cov[\ind(U_2 \le U_1),\min\{U_2,U_3\}] = -\frac{1}{24},~ \Cov[\ind(U_1 \le U_2),\min\{U_2,U_3\}] = \frac{1}{24},\\
  & \Cov[\ind(U_3 \le \min\{U_1,U_2\}),\min\{U_3,U_4\}] = -\frac{2}{45},~ \Cov[\ind(U_3 \le \min\{U_1,U_2\}),\min\{U_1,U_2\}] = \frac{1}{18},\\
  & \Cov[\ind(U_3 \le \min\{U_1,U_2\}),\min\{U_1,U_3\}] = -\frac{1}{36},~ \Cov[\ind(U_3 \le \min\{U_1,U_2\}),\min\{U_1,U_4\}] = \frac{1}{45}.
\end{align*}

(1) For $\Cov [R_1,U_1]$, we have
\[
  \Cov [R_1,U_1] = (n-1)\Cov[\ind(U_2 \le U_1),U_1] = \frac{n-1}{12}.
\]

(2) For $\Cov[R_1,U_2]$, we have
\[
  \Cov[R_1,U_2] = \Cov[\ind(U_2 \le U_1),U_2] + (n-2)\Cov[\ind(U_3 \le U_1),U_2] = -\frac{1}{12}.
\]

(3) For $\Cov[R_1,\min\{U_2,U_3\}]$, we have
\begin{align*}
  & \Cov[R_1,\min\{U_2,U_3\}] \\
  = & 2 \Cov[\ind(U_2 \le U_1),\min\{U_2,U_3\}] + (n-3)\Cov[\ind(U_4 \le U_1),\min\{U_2,U_3\}] = -\frac{1}{12}.
\end{align*}

(4) For $\Cov[R_1,\min\{U_1,U_2\}]$, we have
\begin{align*}
  & \Cov[R_1,\min\{U_1,U_2\}] \\
  = & \Cov[\ind(U_2 \le U_1),\min\{U_1,U_2\}] + (n-2)\Cov[\ind(U_3 \le U_1),\min\{U_1,U_2\}] = \frac{n-2}{24}.
\end{align*}

(5) For $\Cov[\min\{R_1,R_2\},\min\{U_3,U_4\}]$, we have
\begin{align*}
  & \Cov[\min\{R_1,R_2\},\min\{U_3,U_4\}] \\
  = & 2 \Cov[\ind(U_3 \le \min\{U_1,U_2\}),\min\{U_3,U_4\}] + (n-4)\Cov[\ind(U_5 \le \min\{U_1,U_2\}),\min\{U_3,U_4\}]\\
  = & -\frac{4}{45}.
\end{align*}

(6) For $\Cov[\min\{R_1,R_2\},\min\{U_1,U_3\}]$, we have
\begin{align*}
  & \Cov[\min\{R_1,R_2\},\min\{U_1,U_3\}] \\
  = & \Cov[\ind(U_3 \le \min\{U_1,U_2\}),\min\{U_1,U_3\}] + (n-3)\Cov[\ind(U_3 \le \min\{U_1,U_2\}),\min\{U_1,U_4\}]\\
  = & \frac{4n-17}{180}.
\end{align*}

(7) For $\Cov[\min\{R_1,R_2\},\min\{U_1,U_2\}]$, we have
\[
  \Cov[\min\{R_1,R_2\},\min\{U_1,U_2\}] = (n-2) \Cov[\ind(U_3 \le \min\{U_1,U_2\}),\min\{U_1,U_2\}] = \frac{n-2}{18}.
\]

The whole proof is thus complete.
\end{proof}

\subsubsection{Proof of Lemma~\ref{lemma:dist}}

\begin{proof}[Proof of Lemma~\ref{lemma:dist}]

We first define $L_i :=W_{n,j_m(i)} - W_{n,i}$. Let $\ell_i := W_{n,j_1(i)} - W_{n,i}$, then
\begin{align*}
  L_i &=W_{n,j_m(i)} - W_{n,i}\\
  &= \Big[ W_{n,j_m(i)} - W_{n,i} \Big] \ind(j_m(i) \neq i)\\
  &= \Big[ \sum_{k=0}^{m-1} \Big( W_{n,j_{k+1}(i)} - W_{n,j_k(i)} \Big) \Big] \ind(j_m(i) \neq i)\\
  &= \Big[ \sum_{k=0}^{m-1} \ell_{j_k(i)} \Big] \ind(j_m(i) \neq i)\\
  &= \sum_{k=0}^{m-1} \ell_{j_k(i)} \ind(j_m(i) \neq i),
\end{align*}
where we take $j_0(i)=i$ for any $i \in \zahl{n}$. We then have
\[
  \sum_{i=1}^n L_i = \sum_{i=1}^n \sum_{k=0}^{m-1} \ell_{j_k(i)} \ind(j_m(i) \neq i).
\]
For any $j \in \zahl{n}$, the number of $i \in \zahl{n}$ such that $j_m(i) \neq i$ and $j=j_k(i)$ for some $k \in \zahl{m-1}$ can be at most $m-1$. Then $\ell_i$ can appear at most $m$ times in $\sum_{l=1}^n L_l$ for any $i \in \zahl{n}$, implying
\[
  \sum_{i=1}^n L_i \le m \sum_{i=1}^n \ell_i.
\]
Let $I_i := [W_{n,i},W_{n,j_1(i)})$ and $I_i := \emptyset$ if $j_1(i) = i$. Then $[I_i]_{i=1}^n$ are disjoint and $I_i$ has length $\ell_i$. Since by assumptioin $[W_{n,i}]_{i=1}^n \subseteq [-D_n,D_n]$, $\bigcup_{i=1}^n I_i \subseteq [-D_n,D_n]$. Then
\[
  \sum_{i=1}^n L_i \le m \sum_{i=1}^n \ell_i \le 2D_n m.
\]

To finish the proof, noticing that the fact $[W_{n,i}]_{i=1}^n$ are i.i.d. yields that $[L_i]_{i=1}^n$ are i.i.d.,
\[
  \E [W_{n,j_m(1)} - W_{n,1}] = \E [L_1] \le \frac{1}{n} \E \Big[ \sum_{i=1}^n L_i \Big] \le 2\frac{m}{n}D_n,
\]
which concludes the proof.
\end{proof}

\subsubsection{Proof of Lemma~\ref{lemma:local,T,quad}}

\begin{proof}[Proof of Lemma~\ref{lemma:local,T,quad}]
Since
\[
  \E \Big[ \Big(\sum_{s=2}^7 T_{m,s})^2 \Big] \le 6 \sum_{s=2}^7 \E [T_{m,s}^2],
\]
then it suffices to consider the upper bound of each term seperately.

(1) For $T_{m,2}$, from Taylor's expansion,
\begin{align*}
  \E [T_{m,2}^2] = & \E \Big[ \int \Big[1 - F_Y(y)\Big] \Big[F_Y(y) -F_{Y|X=X_1}(y)\Big] f_Y(y) \d y \Big]^2\\
  = & \E \Big[ \int \Big[1 - F_Y(y)\Big] f_Y(y_x) \Big(\frac{y-\rho_nX_1}{\sigma} -y \Big) f_Y(y) \d y \Big]^2\\
  \le & \E \Big[ \lVert f_Y \rVert_{\infty} \int \Big[1 - F_Y(y)\Big] f_Y(y) \d y \Big(\frac{1}{2}\lvert y \rvert \rho_n^2 + \lvert X_1 \rvert \rho_n \Big) (1+o(1))\Big]^2\\
  \le & \Big[ \frac{1}{2} \Big[ \lVert f_Y \rVert_{\infty} \int y^2 \Big[1 - F_Y(y)\Big] f_Y(y) \d y \Big]^2 \rho_n^4 \\
  & + 2 \Big[ \lVert f_Y \rVert_{\infty} \int \Big[1 - F_Y(y)\Big] f_Y(y) \d y \Big]^2 \E [X_1^2] \rho_n^2 \Big] (1+o(1))\\
  \lesssim & \rho_n^2.
\end{align*}

(2) For $T_{m,3}$,
\[
  T_{m,3} = \int \Big[F_Y(y) -F_{Y|X=X_1}(y)\Big]^2 f_Y(y) \d y.
\]

Then $0 \le T_{m,3} \le 1$. From Lemma~\ref{lemma:local,T3},
\[
  \E [T_{m,3}^2] \le \E [T_{m,3}] \lesssim \rho_n^2.
\]

(3) For $T_{m,4}$,
\[
  T_{m,4} = \int \Big[1 - F_Y(y)\Big] \Big[F_{Y|X=X_1}(y) - F_{Y|X=X_{j_m(1)}}(y) \Big] f_Y(y) \d y.
\]

Then $0 \le T_{m,4} \le 1$. From Lemma~\ref{lemma:local,T4},
\[
  \E [T_{m,4}^2] \le \E [T_{m,4}] \lesssim \frac{M}{n} \sqrt{\log n} \rho_n + \frac{M}{n^2} + o(\rho_n^2).
\]

(4) For $T_{m,5}$,
\[
  T_{m,5} = \int \Big[F_Y(y) -F_{Y|X=X_1}(y)\Big] \Big[F_{Y|X=X_1}(y) - F_{Y|X=X_{j_U(1)}}(y) \Big] f_Y(y) \d y.
\]

Then $\lvert T_{m,5} \rvert \le 1$. From Lemma~\ref{lemma:local,T5},
\[
  \E [T_{m,5}^2] \le \E [\lvert T_{m,5} \rvert ] = \frac{M}{n^2} + o(\rho_n^2).
\]

(5) For $T_{m,6}$,
\[
  T_{m,6} = \int \Big[1-F_Y(y)\Big] F_Y(y) f_Y(y) \d y \ind(j_m(1) = 1)
\]

Then
\[
  \E [T_{m,6}^2] = \Big[ \int \Big[1-F_Y(y)\Big] F_Y(y) f_Y(y) \d y \Big]^2 \P(j_m(1) = 1) \lesssim \frac{M}{n}.
\]

(6) For $T_{m,7}$,
\[
  T_{m,7} = \int \Big[F_Y(y) - F_{Y|X=X_1}(y)\Big] \Big[F_Y(y) + F_{Y|X=X_1}(y) - 1\Big] f_Y(y) \d y \ind(j_U(1) = 1).
\]

Then
\[
  \E [T_{m,7}^2] \le \P(j_m(1) = 1) \lesssim \frac{M}{n}.
\]
Putting them together completes the proof.
\end{proof}

{
\bibliographystyle{apalike}
\bibliography{AMS}
}

\end{document}